\documentclass[11pt,reqno]{amsart}
\usepackage{amsmath,amssymb,amsthm,amsfonts}
\usepackage{graphicx}
\usepackage{mathrsfs}
\usepackage{bbm}  % cmds: \mathbbm{}, \mathbbmss{}, \mathbbmtt{}, 
\usepackage{xcolor}
\usepackage{stmaryrd}
\usepackage{upgreek}
\usepackage{esint}
\usepackage{appendix}
\usepackage{mathtools} % similaire ?  empheq
\usepackage{ulem}
\usepackage{cancel}

\usepackage[bbgreekl]{mathbbol} % cmd: \bbmu, ...

\definecolor{darkred}{rgb}{0.5,0,0}
\definecolor{darkgreen}{rgb}{0,0.5,0}
\definecolor{darkblue}{rgb}{0,0,0.5}
\usepackage[colorlinks = true, linkcolor=darkblue,
filecolor=darkgreen,
urlcolor=darkred,
citecolor=darkgreen,pdftex]{hyperref}

\newtheorem{remark}{Remark}
\newtheorem{theorem}{Theorem}
\newtheorem{lemma}{Lemma}

\def\namedproof{\par{\vspace{0.5em}\it Proof} \ignorespaces}
\def\endnamedproof{{\hfill $\square$\\ \vspace{0.5em}}}

\def \R{{\mathbb{R}}}
\def \N{{\mathbb{N}}}
\def \Q{{\mathbb{Q}}}

\def \Z{{\mathbb{Z}}}

\def \T{{\mathbb{T}}}
\def \P{{\mathbb{P}}}
\def \I{{\mathbb{I}}}

\newcommand{\myparallel}{{\mkern3mu\vphantom{\perp}\vrule depth 0pt\mkern2mu\vrule depth 0pt\mkern3mu}}

\newcommand{\vertiii}[1]{{\left\vert\kern-0.25ex\left\vert\kern-0.25ex\left\vert #1 \right\vert\kern-0.25ex\right\vert\kern-0.25ex\right\vert}}

\newcommand{\dsp}[2]{\left\langle\hspace{-#1pt}\left\langle#2\right\rangle\hspace{-#1pt}\right\rangle}

\DeclareMathAlphabet{\mathbbmsl}{U}{bbm}{m}{sl} % cmd: \mathbbmsl{}
\DeclareMathAlphabet{\mathbbmssit}{U}{bbmss}{m}{it} % cmd: \mathbbmssit{}
\DeclareSymbolFontAlphabet{\mathbb}{AMSb} % cmd: \mathbb{}
\DeclareSymbolFontAlphabet{\mathbbl}{bbold} % cmd: \mathbbl{}

\title%[Asymptotic analysis of compressible MHD]
      {Singular limits of anistropic weak solutions to compressible magnetohydrodynamics
      }
\author[N. Besse]{Nicolas Besse}
\address[Nicolas Besse]{Laboratoire J.-L. Lagrange,
                        Observatoire de la C\^ote d'Azur,
                        Universit\'e C\^ote d'Azur,
                        Bvd de l'Observatoire CS 34229,
                        06304 Nice Cedex 4, France}
\email{\tt Nicolas.Besse@oca.eu}

\author[C. Cheverry]{Christophe Cheverry}
\address[Christophe Cheverry]{Institut de Recherche Math\'ematique de Rennes,
                              Universit\'e de Rennes,
                              Campus de Beaulieu,
                              263 avenue du G\'en\'eral Leclerc CS 74205,
                              35042 Rennes Cedex, France}
\email{\tt Christophe.Cheverry@univ-rennes.fr}

\begin{document}

\maketitle

\renewcommand{\thefootnote}{\fnsymbol{footnote}}
\footnotetext{{2020 Mathematics Subject Classification.}  
}
\renewcommand{\thefootnote}{\arabic{footnote}}

\begin{abstract}
The aim is to justify rigorously the so-called reduced magnetohydrodynamic model (abbreviated as RMHD), which is widely used in fusion, space and astrophysical plasmas. Motivated by physics, the focus is on plasmas that are simultaneously strongly magnetized and anisotropic. We consider conducting fluids that can be described by viscous and resistive barotropic compressible magnetohydrodynamic equations. The purpose is to study the asymptotic behaviour of global weak solutions, which do exist, for strongly anisotropic plasmas such as the large aspect ratio framework. We prove that such anisotropic weak solutions converge to the weak solutions of the RMHD equations. Rigorous justification of this limit is performed both in a periodic domain and in the whole space. It turns out that the resulting  system is incompressible only in the perpendicular direction to the external strong magnetic field, whereas it involves compressible features in the parallel direction. In order to pass to the singular limit in the perpendicular direction we exploit, among others,  tools elaborated for proving the low Mach number limit of compressible fluid flows such as the introduction of a fast oscillatory unitary group associated to the dynamics of  transverse fast magnetosonic waves. In the parallel direction, we bring out compactness arguments and  particular cancellations coming from the structure of our  equations.
\end{abstract}

%\tableofcontents

{\small \parskip=1pt
\setcounter{tocdepth}{1}
\tableofcontents
}

\section{Introduction}
\label{s:intro}
Equations of reduced magnetohydrodynamics, hereafter abbreviated as RMHD, are extensively used in  fusion, space and astrophysical plasmas. They are highly prized by plasma physicists for the following reasons. First, they allow interesting theoretical and analytical developments; second they are the source of numerically tractable models which are used to bring forth codes that are routinely exploited \cite{BNS98,CH08,KHC98,NLB12,OBH15,OMD17,RHJ12}. The RMHD model was introduced in the seventhies \cite{KP74, Str76} in the context of fusion plasmas. It was followed by many systematic studies  and  generalizations including more and more physical effects and refinements \cite{DS12, DA83, Haz83, IMS83, KW97, KHC98, Str77, ZM92, ZM93}. At present, there is a vast literature about formal derivations and applications of RMHD models. The two references \cite{Bis93, SCD09} are good introductions to the subject, with many references. RMHD equations are still a very active research field, including recent progress in extended magnetohydrodynamics \cite{BC1}. At the same time, the work \cite{BC2} has highlighted the importance and the inherent  difficulties of working under anisotropic conditions.

%%%%%%%%%%%%%%%%%%%%%%%%

\subsection{The penalized system}
\label{ss:penalized}
Let $\Omega$ be a  three-dimensional domain, which is either the periodic box $\T^3 :=(\R/2 \pi \Z)^3 $ or the whole space $\R^3$. The time evolution on $ \Omega $ of plasmas is basically  described by isentropic compressible magnetohydrodynamics. The unknowns are made of the fluid density $\uprho \in \R^+ $, the fluid velocity $ \text{v} \in \R^3 $ and the magnetic field $ \text{B} \in \R^3 $.  Including viscous ($ \upmu >0 $ and $  \uplambda >0 $) and resistive ($ \upeta >0 $) effects, we consider on $ \mathbb R_+ \times \Omega $ the system of MHD equations
\begin{equation}
\left \{
\begin{aligned}
&\partial_{\text{t}} \uprho + \nabla \cdot ( \uprho \,  \text{v}) = 0 \,, \\
&\partial_{\text{t}} ( \uprho \, \text{v}) + \nabla \cdot (\uprho \, \text{v} \otimes \text{v})
  + \nabla \text{p} +  \text{B} \times (\nabla \times \text{B} ) -\upmu \,  \Delta \text{v} -  \uplambda \, \nabla (\nabla \cdot \text{v}) =0 \,,  \\
&\partial_{\text{t}} \text{B} +  \nabla \times ( \text{B}  \times \text{v})
-  \upeta \, \Delta \text{B} =0\, ,
\end{aligned}
\right.
\label{mhdorigin}
\end{equation}
together with the divergence-free condition $  \nabla \cdot \text{B} =0 $ and the barotropic state law $ \text{p}= \mathfrak{a}\uprho^\gamma $, where $ \mathfrak{a} >0 $ and $ \gamma >1 $. The symbol $\times$ denotes the cross product of vectors of $\R^3$, while the symbol $\otimes $ stands for the tensor product of vectors. Given $ (n,m) \in \N_*^2 $, we take the convention
\[
\big(\nabla \cdot (u \otimes v)\big)_i = \sum_{j=1}^n \partial_j (u_i v_j) \, , \quad u = {}^t (u_1,\cdots,u_m) \in \R^m \, , \quad v = {}^t (v_1,\cdots,v_n) \in \R^n \, .
\]
Following preceding results about  compressible Navier--Stokes equations  \cite{FNP01,Lio98}, the existence of global (in time) weak solutions to system \eqref{mhdorigin} has been obtained in \cite{HW10}. Now, most plasmas are magnetized. This means that $ \text{B} $ is, in a first rough approximation, a given  external non-zero magnetic field $\text{B}_{\mathbbl{e}} $. For simplicity, we assume that $ \text{B}_{\mathbbl{e}} $ is constant. After rotation, it can always be adjusted in such a way that $ \text{B}_{\mathbbl{e}} = \mathbbmss{B} \, e_\myparallel $ with $ \mathbbmss{B} \in \R_+^* $ and $ e_\myparallel :={}^t(0,0,1) $. The direction $e_\myparallel $ is called {\it parallel}. Given a vector field like $ \text{B} $ (or $ \text{v} $), we can decompose $ \text{B} $ into its parallel component $ \text{B}_\myparallel =\text{B}_3 := \text{B} \cdot e_\myparallel \in \R $ and its perpendicular component $ \text{B}_\perp := {}^t (\text{B}_1,\text{B}_2) \in \R^2 $ so that $ \text{B} = {}^t ({}^t \text{B}_\perp, \text{B}_\myparallel )$. We work away from vacuum, near a constant density which can always be put in the form $ \mathbbmss{B}^2 \, \bbrho $ for some $  \bbrho \in \R_+^* $. Observe that $ (\mathbbmss{B}^2 \, \bbrho, 0, \text{B}_{\mathbbl{e}}) $ is a constant solution to \eqref{mhdorigin}. Motivated by physics, particularly  by considerations of large aspect ratio and geometrical optics (see Section~\ref{s:scaling}), we incorporate a strong spatial anisotropy. More precisely, we keep $ \text{x}_\perp := (\text{x}_1,\text{x}_2) = x_\perp = (x_1,x_2) $ and we replace the vertical direction by $ x_\myparallel = x_3 := \varepsilon \, \text{x}_3 $ with $ 0 < \varepsilon \ll 1 $. The above gradient operator becomes
\begin{equation}
  \label{nablaeps}
  \nabla_\varepsilon :=  {}^t ({}^t \nabla_\perp,0) + \varepsilon \,  \nabla_\myparallel \, , \qquad \nabla_\perp := {}^t (\partial_1,\partial_2) \, , \qquad
  \nabla_\myparallel := e_\myparallel \,  \partial_\myparallel
  \, , \qquad \partial_\myparallel \equiv \partial_3 \,.
\end{equation}
Accordingly,  a distinction must be  drawn between $ \Delta_\perp := \partial^2_{11} + \partial^2_{22} $ and $  \Delta_\myparallel := \partial^2_{33} $. We want to study the behavior at large time scales $ t := \varepsilon \, \text{t} \sim 1 $ of small perturbations, of size $ \varepsilon $, of the stationary solution $ ( \mathbbmss{B}^2 \, \bbrho, 0, \text{B}_{\mathbbl{e}}) $. To this end, we seek solutions in the form
\[
( \uprho ,\text{v},\text{B}) (\text{t},\text{x}) = \bigl(   \mathbbmss{B}^2 \, \rho^\varepsilon ,\, \varepsilon \, v^\varepsilon, \, \mathbbmss{B} \, ( e_\myparallel + \varepsilon \, B^\varepsilon ) \bigr) (\varepsilon \text{t},{x}_1,\text{x}_2, \varepsilon \text{x}_3 ) \, .
\]
The unknowns are now $ (\rho^\varepsilon, v^\varepsilon , B^\varepsilon ) (t,x)$, while the pressure is given by $p^\varepsilon = a \, (\rho^\varepsilon)^\gamma $, with $a= \mathbbmss{B}^{2(\gamma-1)}\mathfrak{a}$. Then, the system \eqref{mhdorigin} can be reformulated according to the following equations
\begin{equation}
\left \{
\begin{aligned}
&\partial_t \rho^\varepsilon + \nabla_\varepsilon \cdot (\rho^\varepsilon v^\varepsilon) = 0 \,, \\
&\partial_t (\rho^\varepsilon v^\varepsilon) + \nabla_\varepsilon \cdot (\rho^\varepsilon v^\varepsilon \otimes v^\varepsilon)
  + \frac{1}{\varepsilon^2} \, \nabla_\varepsilon {p}^\varepsilon + \frac{1}{\varepsilon} (e_\myparallel + \varepsilon B^\varepsilon) \times (\nabla_\varepsilon \times B^\varepsilon ) \\
& \qquad\quad \ \, - \mu_\perp^\varepsilon \Delta_\perp v^\varepsilon - \mu_\myparallel^\varepsilon \Delta_\myparallel v^\varepsilon
- \lambda^\varepsilon \nabla_\varepsilon (\nabla_\varepsilon \cdot v^\varepsilon ) =0 \,, \\
&\partial_t B^\varepsilon + \frac{1}{\varepsilon} \nabla_\varepsilon \times \big( (e_\myparallel + \varepsilon B^\varepsilon ) \times v^\varepsilon \big)
- \eta_\perp^\varepsilon \Delta_\perp B^\varepsilon - \eta_\myparallel^\varepsilon \Delta_\myparallel B^\varepsilon =0\, ,
\end{aligned}
\right.
\label{mhdeps}
\end{equation}
together with
\begin{equation}
  \label{divbesp}
 \nabla_\varepsilon \cdot B^\varepsilon = 0 \,.
\end{equation}
Without loss of generality, just to simplify the presentation, we can work with $  \bbrho = 1 $. The second equation (for the momentum $ \rho^\varepsilon v^\varepsilon $) in \eqref{mhdeps} indicates that $ \rho^\varepsilon $ should be like $ \rho^\varepsilon = \bbrho + \mathcal{O}(\varepsilon)$. With this in mind, we can introduce the new state variable $ \varrho^\varepsilon $ as indicated below, and expand ${p}^\varepsilon $ in powers of $ \varepsilon $ to obtain
\begin{equation}
  \label{expanesp}
\rho^\varepsilon = 1 + \varepsilon \, \varrho^\varepsilon \, , \qquad {p}^\varepsilon = a + \varepsilon \, \mathbbmssit{p}^\varepsilon  + \mathcal{O}(\varepsilon^2) \, , \qquad  \mathbbmssit{p}^\varepsilon := b \, \varrho^\varepsilon\, , \qquad b := a \, \gamma \,.
\end{equation}
To see the heuristics which lead to our model, it is instructive to interpret \eqref{mhdeps} in terms of $ (\varrho^\varepsilon, v^\varepsilon , B^\varepsilon ) $, and then to extract the singular part. We find 
\begin{equation}
\left \{
\begin{aligned}
&\partial_t \varrho^\varepsilon + \frac{1}{\varepsilon} \, \nabla_\perp  \cdot v^\varepsilon_\perp = \mathcal{O}(1) \,, \\
&\partial_t v^\varepsilon_\perp + \frac{1}{\varepsilon} \, \nabla_\perp ( \mathbbmssit{p}^\epsilon + B^\varepsilon_\myparallel) = \mathcal{O}(1) \,, \qquad \partial_t v^\varepsilon_\myparallel = \mathcal{O}(1) \,,\\
&\partial_t B^\varepsilon_\myparallel + \frac{1}{\varepsilon} \,  \nabla_\perp \cdot v_\perp^\varepsilon = \mathcal{O}(1) \,, \qquad \qquad \partial_t B^\varepsilon_\perp = \mathcal{O}(1)\, .
\end{aligned}
\right.
\label{mhdepssing}
\end{equation}
According  to the terminology of Schochet \cite{Sch94}, the asymptotic regime is called:
\begin{itemize}
 \item {\it slow}  when the first-order time derivative of the solution remains bounded uniformly with respect to the small paramater $\varepsilon $ (as $\varepsilon \rightarrow 0$). In view of \eqref{mhdepssing}, this means that
 \begin{equation}
\nabla_\perp  \cdot v^\varepsilon_\perp = \mathcal{O}(\varepsilon) \,, \qquad \nabla_\perp ( \mathbbmssit{p}^\epsilon + B^\varepsilon_\myparallel) = \mathcal{O}(\varepsilon) \, .
\label{slowprep}
\end{equation}
\item {\it fast} when it is not slow. In this case, rapid  oscillations with non-vanishing amplitudes can persist on a long time scale, preventing the convergence in a usual strong sense.
Since the singular part involves the sole action of the operator $ \nabla_\perp $, it induces a propagation which can only be achieved with respect to the perpendicular direction. Then, because  Alfv\'en waves do not propagate in the directions orthogonal to the ambiant magnetic field (here $ e_\myparallel $), we are necessarily concerned with transverse fast magnetosonic waves. This claim is justified in Section~\ref{ss:LAR}, where the eigenmodes of the linear (singular) system \eqref{mhdepssing} are investigated.
\end{itemize}
At time $ t=0 $, we start with
\begin{equation}
  \label{initial-data-eps}
  \rho^\varepsilon_{\mid t=0 } =\rho_0^\varepsilon\,, \qquad  v^\varepsilon_{\mid t=0 } =v_0^\varepsilon\,, \qquad  B^\varepsilon_{\mid t=0 }=B_0^\varepsilon \, .
\end{equation}
In coherence with \eqref{divbesp}, we must impose $ \nabla_\varepsilon \cdot B_0 ^\varepsilon = 0 $. The initial data is said to be {\it prepared} when
 \begin{equation}
\nabla_\perp  \cdot v^\varepsilon_{0\perp} = \mathcal{O}(\varepsilon) \,, \qquad \nabla_\perp ( b \, \varrho^\epsilon_0 + B^\varepsilon_{0 \myparallel }) = \mathcal{O}(\varepsilon) \, .
\label{slowprep0}
\end{equation}
 At this stage, it should be noted that the structure of the penalized terms inside \eqref{mhdepssing} and of the subsequent condition \eqref{slowprep} are   different from isotropic  situations \cite{JSX19}: the fluid should be almost incompressible only in the perpendicular direction (the action of the operator $ \nabla_\perp $ appears in place of the full gradient $ \nabla $); the components $  \varrho^\epsilon_0 $ and $ B^\varepsilon_{0 \myparallel } $ must be approximately linked together. To our knowledge, the asymptotic study of systems like \eqref{mhdeps} has not yet been undertaken neither in a smooth context or for weak solutions.

From now on, we assume that the positive  perpendicular and parallel shear viscosities $\mu_\perp^\varepsilon>0$ and $\mu_\myparallel^\varepsilon>0$, as well as the positive bulk viscosity $\lambda^\varepsilon >0 $ are adjusted in such a way that
\begin{equation}
  \label{visco-norm}
  \mu_\perp^\varepsilon \xrightarrow{\quad}  \mu_\perp >0 \,, \quad  \mu_\myparallel^\varepsilon \xrightarrow{\quad}  \mu_\myparallel > 0 \,,
  \quad \lambda^\varepsilon  \xrightarrow{\quad} \lambda > 0 \,, \quad \mbox{as} \quad  \varepsilon \xrightarrow{\quad} 0_+\,.
\end{equation}
Similarly the positive perpendicular and parallel resistivities $\eta_\perp^\varepsilon>0$ and $\eta_\myparallel^\varepsilon>0$ must satisfy
\begin{equation}
  \label{resis-norm}
  \eta_\perp^\varepsilon \xrightarrow{\quad}  \eta_\perp >0 \,, \quad  \eta_\myparallel^\varepsilon \xrightarrow{\quad}  \eta_\myparallel > 0 \,,
 \quad \mbox{as} \quad \varepsilon \xrightarrow{\quad} 0_+\,.
\end{equation}
The system \eqref{mhdeps} is equipped with a conserved energy. We mainly assume that the energy of the initial data $ (\rho_0^\varepsilon,v_0^\varepsilon,B_0^\varepsilon) $ is bounded uniformly with respect to $ \varepsilon $, see \eqref{ineqC0-P} and \eqref{ineqC0-R}. We add technical conditions which are distinct when $ \Omega = \T^3 $ (Subsection  \ref{ss:rT3}) and when $ \Omega = \R^3 $ (Subsection  \ref{ss:rR3}) to guarantee that the difference $ \rho^\varepsilon_0 - 1 $ vanishes in $ L^\gamma_{\rm loc}(\Omega) $ when $ \varepsilon $ goes to zero. Then,
 up to a subsequence and at least in a weak sense (specified further), we have
\begin{equation*}
 \varrho^\epsilon_0 :=\frac{ \rho^\varepsilon_0  - 1}{\varepsilon} \xrightharpoonup{\quad} \varrho_0 \, , \qquad
 v^\varepsilon_0 \xrightharpoonup{\quad} v_0 \in L^2\,,
 \qquad B^\varepsilon_0 \xrightharpoonup{\quad} B_0 \in L^2 \,, \quad \mbox{as} \quad \varepsilon\xrightarrow{\quad} 0_+\, .
\end{equation*}
We will show (Theorems \ref{TH-CV-P} and \ref{TH-CV-R}) that, always up to a subsequence, the difference $ \rho^\varepsilon  - 1 $ vanishes strongly in $ L_{\rm loc}^\infty ( \R_+; L^\gamma_{\rm loc}(\Omega)) $ and that, at least in a weak sense, we have\footnote{In the periodic case, let $
\overline{\rho_0^\varepsilon} $ be the constant (close to $ 1 $) defined by \eqref{ineqC0-P++}. When $ \Omega = \T^3 $,
as stated in Theorem \ref{TH-CV-P}, the definition of $ \varrho^\epsilon $ should be replaced by $ \varrho^\epsilon:= ( \rho^\varepsilon  -
\overline{\rho_0^\varepsilon} )/\varepsilon $.}
\begin{equation*}
   \varrho^\epsilon:=\frac{ \rho^\varepsilon  - 1}{\varepsilon} \xrightharpoonup{\quad} \varrho \,, \qquad
   v^\varepsilon \xrightharpoonup{\quad} v \,,
 \qquad B^\varepsilon  \xrightharpoonup{\quad} B \,,  \quad \mbox{as} \quad \varepsilon\xrightarrow{\quad} 0_+\, .
\end{equation*}
The next stage is to identify the limit $ (\varrho,v,B) $.

%%%%%%%%%%%%%%%%%%%%%%%%

\subsection{The RMHD model}\label{ss:RHMD}
The perpendicular component $ B_\perp := {}^t(B_1,B_2) \in \R^2 $ of $ B $ and the perpendicular component $ v_\perp := {}^t(v_1,v_2) \in \R^2 $  of $ v $ can be identified independently by solving the following nonlinear closed system
\begin{equation}
\left \{
\begin{aligned}
  &\partial_t B_\perp -\partial_\myparallel v_\perp + \nabla_\perp \cdot (B_\perp \otimes v_\perp - v_\perp \otimes B_\perp)
  -\eta_\perp \Delta_\perp B_\perp -\eta_\myparallel \Delta_\myparallel B_\perp=0\,, \\
  &\partial_t v_\perp -\partial_\myparallel B_\perp + \nabla_\perp \cdot (v_\perp \otimes v_\perp - B_\perp \otimes B_\perp) + \nabla_\perp \uppi
  -\mu_\perp \Delta_\perp v_\perp -\mu_\myparallel \Delta_\myparallel v_\perp=0\, ,
\end{aligned}
\right.
\label{rmhdincomp}
\end{equation}
together with the (transverse velocity) divergence-free condition
\begin{equation}
  \label{divperp}
 \nabla_\perp \cdot v_\perp =0 \, ,
\end{equation}
and the initial data
\begin{equation}
  \label{inidatperp}
  (B_\perp,v_\perp)_{\mid t=0} = (\P_\perp B_{0 \perp} , \P_\perp v_{0 \perp}) \in L^2(\Omega;\R^4)\, ,
\end{equation}
where the projection $ \P_\perp $ denotes the (two-dimensional) transverse Leray operator. Passing to the weak limit in $ \nabla_\varepsilon \cdot B_0^\varepsilon = 0 $, we can easily infer that $ \nabla_\perp \cdot B_{0\perp} =0$, and therefore $ B_{0 \perp} = \P_\perp B_{0 \perp} $. The same applies concerning $ v_0^\varepsilon $ in the case of prepared data. For unprepared data $ v_0^\varepsilon $, in general,we find that $ v_{0 \perp} \not = \P_\perp v_{0 \perp}$. Still, we will show in Section~\ref{ss:cv-rvperp-P} that  the limit initial condition is $\P_\perp v_{0\perp}$ and not $ v_{0\perp} $. This passage from $ v_{0 \perp} $ to $ \P_\perp v_{0 \perp} $ reveals the underlying presence of a time boundary layer (which may arise in the absence of preparation). In the second equation of \eqref{rmhdincomp}, the pressure $\uppi$ plays the role of a Lagrange multiplier to ensure the transverse incompressibility of the flow. Since
\[
\nabla_\perp \cdot (B_\perp \otimes v_\perp - v_\perp \otimes B_\perp) = {}^t \bigl( \partial_2 (B_1 \, v_2 - v_1 \, B_2) , - \partial_1 (B_1 \, v_2 - v_1 \, B_2)\bigr) \, ,
\]
exploiting \eqref{divperp}, we can assert that
\[
\partial_t (\nabla_\perp \cdot B_\perp) -\eta_\perp \Delta_\perp (\nabla_\perp \cdot B_\perp) -\eta_\myparallel \Delta_\myparallel (\nabla_\perp \cdot B_\perp) =0\, .
\]
It follows that the divergence-free condition on $ B_{0 \perp} $ is propagated. Retain that
\begin{equation}
  \label{divperppourB}\nabla_\perp \cdot B_\perp =0 \, .
\end{equation}
The existence of global-in-time weak solutions to \eqref{rmhdincomp}-\eqref{divperp}-\eqref{inidatperp} can be obtained from classical methods in \cite{DL72, ST83}. Note that it can be deduced indirectly from the existence (for all $\varepsilon >0$) of weak solutions to \eqref{mhdeps}. Indeed, as will be seen, the rigorous justification of the passage to the limit ($\varepsilon \rightarrow 0$) in the system \eqref{mhdeps} provides another way to construct global weak solutions of \eqref{rmhdincomp}-\eqref{divperp}-\eqref{inidatperp}. Observe also that the system is linear in the parallel direction. Thus, these global  existence results remain true even when $ \eta_\myparallel = 0 $ and $ \mu_\myparallel = 0 $.

Now, the treatment of the parallel components $ (B_\myparallel , v_\myparallel ) $ differs completely from what is done  usually. This is due to the (unconventional) anisotropic context which forces to look at $ (B^\varepsilon_\myparallel , v^\varepsilon_\myparallel ) $ separately. On the one hand, $  (B^\varepsilon_\myparallel , v^\varepsilon_\myparallel ) $ appears as a leading order term, and therefore its weak limit $  (B_\myparallel , v_\myparallel ) $ must contribute to the main description of the flow. As such, it must be incorporated in the RMHD  model. On the other hand, to some extent, $  (B^\varepsilon_\myparallel , v^\varepsilon_\myparallel ) $ is (partly) dealt in the equations as a second order term. It follows that the  determination of its weak limit $ (B_\myparallel , v_\myparallel ) $ is decoupled from the one of $ (B_\perp , v_\perp ) $. In fact, knowing the content of $ (B_\perp , v_\perp ) $, with the constant $ \mathbbl{c} := 1 + (1/b) > 1 $, we have access to $ (B_\myparallel , v_\myparallel ) $ through
\begin{equation}
\left \{      
\begin{aligned}
& \mathbbl{c} \bigl( \partial_t B_\myparallel + (v_\perp \cdot \nabla_\perp) B_\myparallel \bigr) - \partial_\myparallel v_\myparallel - (B_\perp \cdot \nabla_\perp) v_\myparallel
  -\eta_\perp \Delta_\perp B_\myparallel -\eta_\myparallel \Delta_\myparallel B_\myparallel=0\, , \\
  &\partial_t v_\myparallel + (v_\perp \cdot \nabla_\perp) v_\myparallel -\partial_\myparallel B_\myparallel  - (B_\perp \cdot \nabla_\perp) B_\myparallel
  -\mu_\perp \Delta_\perp v_\myparallel -\mu_\myparallel \Delta_\myparallel v_\myparallel = 0\, ,
\end{aligned}
\right.
\label{rmhd}
\end{equation}
and the initial data
\begin{equation}
  \label{inidatparal}
  (B_\myparallel,v_\myparallel)_{\mid t=0} = (\mathbbmsl{B}_{0\myparallel} , v_{0\myparallel} ) \, , \qquad
 \mathbbmsl{B}_{0\myparallel} := (B_{0\myparallel} - \varrho_0 )/\mathbbl{c} \, .
\end{equation}
This is the viscous version of a symmetric linear system implying the known  variable coefficients $ v_\perp $ and $ B_\perp $. For smooth data $ v_\perp $ and $ B_\perp $, the global existence is obvious. Moreover, due to \eqref{divperp} and \eqref{divperppourB}, usual energy estimates concerning $ (B_\myparallel,v_\myparallel) $ do apply without consuming any regularity on $ v_\perp $ and $ B_\perp $. It follows that global solutions do exist even when the coefficients $ v_\perp $ and $ B_\perp $ are issued from the weak solution $ (v_\perp , B_\perp ) $ in $ L_{\rm loc}^\infty(\R_+; L^2(\Omega)) \cap L_{\rm loc}^2(\R_+, \dot H^1(\Omega)) $ to  \eqref{rmhdincomp}.

Given $ \varepsilon >0 $, the fluid is slightly compressible since $ \rho^\varepsilon = 1 + \varepsilon \, \varrho^\varepsilon $ with $  \varrho^\varepsilon \sim \varrho $. The expression $ \varrho^\varepsilon $ (and  its weak limit $ \varrho $) plays at the level of \eqref{mhdeps} the part of a one order corrector which keeps track of the original compressibility. Now, looking at \eqref{mhdepssing}, it acts in the equations with the same order as the components $ v^\varepsilon_\perp $ and $  B^\varepsilon_\myparallel $. It is therefore reasonable to find a link between $ \varrho $, $  v_\perp $ and $ B_\myparallel $. In view of the second relation inside  \eqref{slowprep}, we can already infer that $ \mathbbmssit{p} + B_\myparallel = 0 $, where $ \mathbbmssit{p} := b \, \varrho $ is the weak limit of $ \mathbbmssit{p}^\varepsilon $. By this way, $ B_\myparallel $ acquires asymptotically the status of a pressure which can serve to  measure  some  compressibility in the parallel direction. For prepared data, that is when $  B_{0\myparallel} + b \varrho_0 = 0 $, we start with $ \mathbbmsl{B}_{0\myparallel} = B_{0\myparallel} $. Otherwise, for general unprepared data (which is our framework), we find that $ \mathbbmsl{B}_{0\myparallel}\not \equiv B_{0\myparallel} $, see Subsection \ref{ss:cv-Bpara-P}. Again, this is the hallmark of a boundary layer occuring at time $ t = 0 $ concerning the component $ B^\varepsilon_\myparallel $.

%%%%%%%%%%%%%%%%%%%%%%%%

\subsection{Global overview}\label{ss:overview} This paper is devoted to the rigorous justification of the convergence of the global weak solutions to \eqref{mhdeps} to those of \eqref{rmhdincomp}-\eqref{inidatparal}. As already mentioned, the nature of the singular limit depends on many factors.

%%%%%%%%%%%%%%%%%%%%%%%%

\subsubsection{Preceding results}\label{sss:preceding}
The hyperbolic version of system \eqref{mhdeps}, which is obtained by removing  viscosities and  resistivities, falls into the framework of the theory of singular limits of quasilinear hyperbolic systems with large parameters. This approach is restricted to smooth solutions (say $ H^s $ with $ s $ large enough). It was originally developped by Klainerman and Majda \cite{KM81, Maj84}. In these circumstances, retain that:
\begin{itemize}
\item In the smooth prepared setting, as a corollary of Theorem~3 in  \cite{KM81}, a convergence result does exist \cite{Gui15} concerning \eqref{mhdeps}. It holds as long as the solution of the limit equations remains smooth. In a related framework, namely with a strong constant magnetic field but without spatial anisotropy, the authors of \cite{JSX19} study the singular limit of the local-in-time smooth solution of the ideal MHD on a bounded domain, with convenient boundary conditions and prepared initial data. It turns out that the limit system in  \cite{JSX19} is essentially two-dimensional since the spatial variable $x_3$ plays the role of a label (no differential nor integral operator with respect to $x_3$). In \cite{BC2}, this method is successfully applied to the (more complex) XMHD system. Note also that we can appeal to Theorem~4 in \cite{KM81} (dealing with the diffusive version of Theorem~3 in \cite{KM81}) to justify the strong convergence on a fixed time interval of smooth prepared solutions to \eqref{mhdeps} to those of \eqref{rmhdincomp}-\eqref{inidatparal}. 
  
 \item In the smooth unprepared setting, one possible strategy \cite{Sch94} is first to exhibit a smooth (in its arguments) limit profile with a double number of variables (one set representing slow variations, the other set fast ones) satisfying an appropriate limit equation (called the modulation equation). Second, it is to prove that the smooth solution of the original system converges in a strong sense on a uniform time interval to this profile evaluated at the slow and fast variables.
\end{itemize}

Weak solutions can also be considered, provided that parabolic contributions are incorporated. This allows to relax the regularity conditions, to reach all times, and therefore to reinforce the universality of reduced models. A way to make progress in this direction has been initiated in \cite{LM98} which (for unprepared data) exploits the unitary group method \cite{Gre97,Sch94} and compactness arguments to construct a filtered profile for the irrotational part  of the velocity field. From this filtered profile, the authors of \cite{LM98} construct a sequence of approximations to the limit solution. Then, they exploit this sequence to pass to the limit in the nonlinear terms. By doing so, they observe that  the solenoidal part of the velocity field inherits a strong convergence, while only weak convergence results are available concerning the irrotational part.

The discussion is very sensitive to the type of domain: $\T^3 $ or $\R^3$. In the case of the whole space, the proof of \cite{LM98} has been  simplified in \cite{DG99} by using Strichartz estimates \cite{GV95, KT98}. This allows to improve the convergence result of the irrotational part of the velocity field, which is precisely the part containing the rapid oscillating acoustic waves.  Indeed, the authors of \cite{DG99} remark that this  irrotational part satisfies a  linear (isotropic) wave equation. From there, due to dispersive effects (in all spatial variables), it must asymptotically vanish in a strong sense.

 %%%%%%%%%%%%%%%%%%%%%%%%

\subsubsection{The anisotropic complications}\label{sss:complications}
We clarify here the important unsolved specificities induced by the implementation of distinct spatial scales. In the smooth prepared context, new problems already arise. For instance, as observed in \cite{BC2}, the anisotropy can  preclude obtaining a complete WKB expansion. Even in the smooth (prepared or not) case, the particularities related to the asymptotic study of \eqref{mhdeps} have not yet been explored. Inspired by \cite{LM98}, our aim is to go directly to weak solutions. We consider viscous and resistive situations vs. (almost) hyperbolic; global weak solutions vs. local strong solutions; $ L^p $ and periodic solutions vs. Sobolev solutions; and general data vs. prepared data. In so doing, the smooth strategies do not help. The good benchmark is \cite{LM98}. But MHD equations are quite different from compressible fluid equations \cite{LM98}. And thus, the discussion  must be adapted to cover the magnetic effects. There are many important challenges to elucidate, especially:
\begin{itemize}
\item The unitary group  method involves $ \Q_\perp v^\varepsilon_\perp $ and $ b \, \varrho^\varepsilon + B^\varepsilon_\myparallel $. It allows to filter out fast oscillating magnetosonic waves propagating in the transverse directions in ways that have not yet been investigated (even in the smooth context). Note in particular that $ b \, \varrho + B_\myparallel = 0 $, instead of simply $ \varrho = 0 $ in \cite{LM98}.
 \item The nonlinear expressions involving $ B^\varepsilon $ are, of course, absent in \cite{LM98}.
 \item Even the tensor product $ \rho^\varepsilon v^\varepsilon \otimes v^\varepsilon $ must be dealt differently. Indeed, in our setting, both $ v^\varepsilon_\myparallel $ and $ \P_\perp v^\varepsilon_\perp $ are left aside by the filtering. Other arguments must be introduced to understand what happens at the level of $ v^\varepsilon_\myparallel $ and $B^\varepsilon_\myparallel $, that is how to recover \eqref{rmhd}. To deal with this issue, we exploit  particular cancellations provided by the structure of system \eqref{mhdeps} that we combine with some compactness results developed in \cite{Lio98} in order to prove the existence of weak solutions to the compressible Navier--Stokes equations. Indeed, in particular from \eqref{mhdepssing}, we observe that $\partial_t(B^\varepsilon_\myparallel- \varrho^\varepsilon ) = \mathcal{O}(1)$. Then, the quantity $(B^\varepsilon_\myparallel- \varrho^\varepsilon )$ will yield the right unknown to prove the limit equation for $ B_\myparallel$. Moreover, in the case of the whole space, in contrast to  \cite{DG99}, we cannot exploit  (isotropic) Strichartz estimates to obtain the strong convergence of $ \Q_\perp v^\varepsilon_\perp $, since the resulting wave equation is posed only in the perpendicular spatial variables $x_\perp$, the parallel spatial variable $x_\myparallel$ being seen as a continuous label. Therefore, a natural and interesting open question arises: could only transverse dispersive effects (and thus some kind of anisotropic Strichartz's estimates) be used to show that $ \Q_\perp v^\varepsilon_\perp $ vanishes strongly ? This question will be addressed in further work.
\end{itemize}

%%%%%%%%%%%%%%%%%%%%%%%%

\subsubsection{Plan of the work}\label{sss:plan} The paper is organized as follows. In Section~\ref{s:results}, we  state our main results. In Section~\ref{s:scaling}, we  come back to the physical motivations and to the origin of  our anisotropic scaling. In Sections~\ref{s:AAPD} and \ref{s:AAWS}, we prove the convergence of the compressible MHD equations \eqref{mhdeps}-\eqref{divbesp} to the RMHD equations \eqref{rmhdincomp}-\eqref{divperp}-\eqref{rmhd}. We start in Section~\ref{s:AAPD} with the case of a periodic domain. Then, in Section~\ref{s:AAWS}, we perform this investigation in the whole space. Finally, in Appendix~\ref{s:tools}, we recall functional analysis results which are exploited throughout the paper.

%%%%%%%%%%%%%%%%%%%%%%%%%%%%%%
%%
%%%%%%%%%%%%%%%%%%%%%%%%%%%%
\section{Main results}
\label{s:results}

In Subsection \ref{ss:notations}, we specify some notations. In Subsection \ref{ss:ws}, we
recall the notion of weak solutions. In Subsection \ref{ss:rT3}, we state our main results in the case of $ \T^3 $. In Subsection  \ref{ss:rR3}, we do the same for $ \R^3 $.

%%%%%%%%%%%%%%%%%%%%%%%%

\subsection{Notation}\label{ss:notations}
Let $\Omega$ be either the periodic domain $\T^3$ or the whole space $\R^3$. For $s\in \R$ and $1\leq p \leq \infty$, we shall use the  standard non-homogeneous Sobolev spaces
\[
W^{s,p}(\Omega)=({\rm I}-\Delta)^{-s/2}L^p(\Omega) \, , \qquad H^s(\Omega)=W^{s,2}(\Omega) \, ,
\]
and their  homogeneous versions
\[
\dot W^{s,p}(\Omega)=(-\Delta)^{-s/2}L^p(\Omega) \, , \qquad \dot H^s(\Omega)=\dot W^{s,2}(\Omega) \, .
\]
We introduce the transverse Leray projection operator $\P_\perp : W^{s,p}(\Omega; \R^2) \rightarrow  W^{s,p}(\Omega; \R^2)$ onto vector fields which are divergence-free in the perpendicular direction
\[
v_\perp =\P_\perp v_\perp + \Q_\perp v_\perp\,, \qquad \nabla_\perp \cdot (\P_\perp v_\perp ) = 0\,, \qquad \nabla_\perp \times  (\Q_\perp v_\perp) =0\,, \qquad \forall v_\perp \in L^2(\Omega;\R^2)\,,
\]
where $\nabla_\perp \times v_\perp :=  \partial_1  v_2 -  \partial_2 v_1 $. From the Mikhlin--H\"ormander Fourier multipliers theorem, the operators $\P_\perp$ and $\Q_\perp$ are continuous maps from the Sobolev space $W^{s,p}(\Omega; \R^2)$ into itself for  $s\in \R$ and $1<p<\infty$. In addition \cite{LM98}, for all $ \delta >0 $, we have the following continuous  embedding $ \P_\perp (L^1(\Omega)) \hookrightarrow W^{-\delta,1}(\Omega) $. This embedding can be justified simply by observing that on the one hand the operators $\P_\perp$ and $\Q_\perp$ are continuous maps from  $L^1(\Omega; \R^2)$ into the Lorentz space $L^{1,\infty}(\Omega)$ (or weak $L^1(\Omega)$; see, e.g., Theorem 5.3.3 in \cite{Gra14}) and on the other hand the continuous embedding $L^{1,\infty}(\Omega) \hookrightarrow W^{-\delta,1}(\Omega)$ holds. We denote by $\mathscr{C}(0,T; L_{\rm weak}^p(\Omega))$, the space of functions which are continuous with respect to $t\in [0,T]$, with values in $L^p(\Omega)$, with the weak topology. Moreover, we introduce the differential operator $D_\varepsilon \equiv {}^t \nabla_\varepsilon$. The scalar product between two matrices $M_1$ and $M_2$ is  defined as $M_1:M_2 = \sum_{ij} M_{1ij} M_{2ij}$. Moreover, given two vectors $ B_\perp \in \R^3 $ and $ v_\perp \in \R^3 $, we adopt below the convention $ B_\perp \times v_\perp := B_1 \, v_2 - B_2 \, v_1 \in \R $.

%%%%%%%%%%%%%%%%%%%%%%%%

\subsection{Weak solutions}
\label{ss:ws} Weak solutions of  RMHD equations will be recovered by passing to the limit ($\varepsilon \rightarrow  0_+$) in the weak formulation associated with \eqref{mhdeps}-\eqref{divbesp}. It is therefore important to specify what is meant by a  weak solution to \eqref{mhdeps}-\eqref{divbesp} and to \eqref{rmhdincomp}-\eqref{divperp}-\eqref{rmhd} when $\Omega=\T^3$ and when $\Omega=\R^3$. Given initial data as in \eqref{initial-data-eps}, with
\begin{equation}
  \label{initial-data-ws}
  \rho_0^\varepsilon \in L_{\rm loc}^1(\Omega)\,, \quad \ \ 
  v_0^\varepsilon\,, \, B_0^\varepsilon, \, \sqrt{\rho_0^\varepsilon} v_0^\varepsilon \in L^2(\Omega)\,,\quad \ \ 
  \nabla_\varepsilon \cdot B_0^\varepsilon =0 \ \, \mbox{in }\  \mathcal{D}'(\Omega)\,,
\end{equation}
a triplet $(\rho^\varepsilon,v^\varepsilon, B^\varepsilon)$ satisfying
\begin{equation}
  \label{reg-triplet}
  \rho^\varepsilon \in L_{\rm loc}^\infty (\R_+; L_{\rm loc}^\gamma(\Omega))\,, \qquad
  (v^\varepsilon ,B^\varepsilon,  \sqrt{\rho^\varepsilon} v^\varepsilon ) \in L_{\rm loc}^\infty (\R_+; L^2(\Omega))\,,
\end{equation}
is said to be a weak solution of \eqref{mhdeps}-\eqref{divbesp} if for all $ \psi={}^t({}^t\psi_\perp,\psi_\myparallel) \in \mathscr{C}_c^\infty(\R_+\times\Omega;\R^3)$ and for all $ \varphi\in \mathscr{C}_c^\infty(\R_+\times\Omega;\R)$ with $ p^\varepsilon = a \, (\rho^\varepsilon)^\gamma $, we have
\begin{equation}
\label{eq:rho-eps}
\int_\Omega dx \rho_0^\varepsilon \varphi(0) +
\int_0^\infty dt \int_\Omega dx \, \rho^\varepsilon\big(\partial_t  \varphi+ v^\varepsilon \cdot \nabla_\varepsilon \varphi\big) =0 \,,
\end{equation}      
\begin{multline}
\label{eq:vperp-eps}
\int_\Omega dx \,\rho_0^\varepsilon v_{0 \perp }^\varepsilon \cdot \psi_\perp(0)
+\int_0^\infty dt \int_\Omega dx\,  \bigg( \rho^\varepsilon v_{\perp }^\varepsilon \cdot \partial_t \psi_\perp
+ \big(\rho^\varepsilon v_{\perp }^\varepsilon \otimes  v^\varepsilon 
- B_{\perp }^\varepsilon \otimes  B^\varepsilon \big): D_\varepsilon \psi_\perp 
\\ + \bigg\{ \frac{1}{\varepsilon^2} p^\varepsilon + \frac{B_\myparallel^\varepsilon}{\varepsilon}
+ \frac{|B^\varepsilon|^2}{2}\bigg\} \nabla_\perp \cdot \psi_\perp
- B_\perp^\varepsilon \cdot \partial_\myparallel \psi_\perp 
+ \mu_\perp v_{\perp }^\varepsilon \cdot \Delta_\perp \psi_\perp + \mu_\myparallel v_{\perp }^\varepsilon \cdot \Delta_\myparallel \psi_\perp 
\bigg) =0 \,,
\end{multline}
\begin{multline}
\label{eq:vparal-eps}
 \int_\Omega dx \,\rho_0^\varepsilon v_{0 \myparallel}^\varepsilon \psi_\myparallel(0)
+ \int_0^\infty dt \int_\Omega dx\,  \bigg( \rho^\varepsilon v_{\myparallel }^\varepsilon \partial_t \psi_\myparallel
+ \big( \rho^\varepsilon v_{\myparallel }^\varepsilon \otimes  v^\varepsilon
- B_{\myparallel }^\varepsilon \otimes  B^\varepsilon \big) : D_\varepsilon \psi_\myparallel 
\\ 
+ \bigg\{ \frac{1}{\varepsilon}p^\varepsilon  + \varepsilon\frac{|B^\varepsilon|^2}{2}\bigg\} \partial_\myparallel  \psi_\myparallel
+ \mu_\perp v_{\myparallel }^\varepsilon  \Delta_\perp \psi_\myparallel + \mu_\myparallel v_{\myparallel }^\varepsilon \Delta_\myparallel \psi_\myparallel
+ \varepsilon \lambda^\varepsilon v^\varepsilon \cdot \nabla_\varepsilon (\partial_\myparallel \psi_\myparallel)  
\bigg) =0 \,,
\end{multline}    
\begin{multline}
\label{eq:Bperp-eps}
 \int_\Omega dx \,B_{0\perp}^\varepsilon \cdot \psi_\perp(0)
+ \int_0^\infty dt \int_\Omega dx\,  \bigg(  B_{\perp }^\varepsilon \cdot \partial_t \psi_\perp
- v_{\perp }^\varepsilon \cdot \partial_\myparallel \psi_\perp - (B_\perp^\varepsilon \times v_\perp^\varepsilon) \nabla_\perp \times \psi_\perp
\\ 
+ \varepsilon v_\myparallel^\varepsilon B_\perp^\varepsilon \cdot \partial_\myparallel \psi_\perp - \varepsilon B_\myparallel^\varepsilon v_\perp^\varepsilon \cdot \partial_\myparallel \psi_\perp
+ \eta_\perp B_{\perp }^\varepsilon \cdot \Delta_\perp \psi_\perp + \eta_\myparallel B_{\perp }^\varepsilon \cdot \Delta_\myparallel \psi_\perp 
\bigg) =0 \,,
\end{multline}
\begin{multline}
\label{eq:Bparal-eps}
 \int_\Omega dx \,  B_{0 \myparallel}^\varepsilon \psi_\myparallel(0)
+ \int_0^\infty dt \int_\Omega dx \bigg( B_{\myparallel }^\varepsilon \partial_t \psi_\myparallel
+ \frac{1}{\varepsilon} v_\perp^\varepsilon \cdot \nabla_\perp \psi_\myparallel 
\\ 
-v_\myparallel^\varepsilon  B_\perp^\varepsilon \cdot \nabla_\perp \psi_\myparallel + B_\myparallel^\varepsilon v_\perp^\varepsilon \cdot \nabla_\perp \psi_\myparallel
+ \eta_\perp B_{\myparallel }^\varepsilon  \Delta_\perp \psi_\myparallel + \eta_\myparallel B_{\myparallel }^\varepsilon \Delta_\myparallel \psi_\myparallel
\bigg) =0 \,,
\end{multline}
\begin{equation}
\label{eq:zero-divB}
\int_\Omega dx\, B^\varepsilon(t) \cdot \nabla_\varepsilon \varphi (t) = 0\,, \quad \forall t\in \R_+\, .
\end{equation}
The notion of weak solution to \eqref{rmhdincomp}-\eqref{divperp} is obtained by testing \eqref{rmhdincomp} against all $ \psi_\perp \in \mathscr{C}_c^\infty(\R_+\times\Omega;\R^2)$ which are such that $\nabla_\perp \cdot \psi_\perp =0$. Concerning \eqref{rmhd}, it suffices to select scalar functions.

%%%%%%%%%%%%%%%%%%%%%%%%%
%%%%%%%%%%%%%%%%%%%%%%%%%

\subsection{The periodic case}
\label{ss:rT3} This is when $\Omega=\T^3$. The functional framework is based on \cite{HW10,LM98}. Select initial data satisfying
\begin{equation}
  \label{reg-ic-mhdeps-P}
   \left\{
   \begin{aligned}
     & \rho^\varepsilon_{|_{t=0}}=\rho_0^\varepsilon\in L^\gamma(\T^3)\,, \qquad  \rho_0^\varepsilon \geq 0\,,
     \qquad \sqrt{\rho_0^\varepsilon} v_0^\varepsilon \in L^2(\T^3)\,, \\
     &(\rho^\varepsilon v^\varepsilon)_{|_{t=0}} = m_0^\varepsilon =
     \left\{
     \begin{aligned}
       & \rho_0^\varepsilon v_0^\varepsilon & \mbox{if } \ \rho_0^\varepsilon\neq 0\\
       & 0  & \mbox{if }\ \rho_0^\varepsilon= 0
     \end{aligned}
     \right\}
     \in L^{2\gamma/(\gamma+1)}(\T^3)\,, \\
     & B^\varepsilon_{|_{t=0}}=B_0^\varepsilon \in L^2(\T^3)\, , \qquad \nabla_\varepsilon \cdot B_0^\varepsilon=0 \, , \qquad  \int_{\T^3} dx \, B_0^\varepsilon=0 \, .
   \end{aligned}
  \right.
\end{equation}
We assume that these regularity assumptions are uniform with respect to $\varepsilon$. Furthermore, given a constant $ C_0 $ (not depending on $ \varepsilon $), we impose
\begin{equation}
  \label{ineqC0-P}
 \frac 12 \int_{\T^3} dx\, \big( \rho_0^\varepsilon |v_0^\varepsilon|^2 + |B_0^\varepsilon|^2\big) + \frac{a}{\varepsilon^2(\gamma-1)}
   \int_{\T^3} dx\, \big( (\rho_0^\varepsilon)^\gamma - \gamma  \rho_0^\varepsilon (\overline{\rho^\varepsilon_0})^{\gamma-1}
   + (\gamma -1)(\overline{\rho^\varepsilon_0})^{\gamma}\big) \leq C_0 \, .
\end{equation}
This is completed by
\begin{equation}
  \label{ineqC0-P++}
\overline{\rho_0^\varepsilon} := \frac{1}{|\T^3|} \int_{\T^3} dx\,\rho_0^\varepsilon \, \xrightarrow{\quad} 1\,, \ \mbox{ as } \varepsilon \xrightarrow{\quad} 0_+\,.   
\end{equation}
This bound gives access to weak compactness. Modulo the extraction of subsequences (which are not  specified), we can say that $\sqrt{\rho_0^\varepsilon } v_0^\varepsilon$ and $B_0^\varepsilon$ converge weakly in $L^2(\T^3)$ to $u_0$ and $B_0$ respectively. From \eqref{ineqC0-P}, some information on $\rho_0^\varepsilon $ and $ \varrho_0^\varepsilon :=(\rho_0^\varepsilon-\overline{\varrho_0^\varepsilon})/\varepsilon $ can also be extracted. We will first show (see the proof of Lemma~\ref{lem:cv-rho}) that $\rho_0^\varepsilon \rightarrow 1$ in $L^\gamma(\T^3)$--strong. Then, we will see that  $\varrho_0^\varepsilon \rightharpoonup \varrho_0$ in $L^\kappa(\T^3)$--weak for $\kappa :=\min\{2,\gamma\}$. Using $\sqrt{\rho_0^\varepsilon } v_0^\varepsilon \rightharpoonup u_0$ in  $L^2(\T^3)$--weak and $\rho_0^\varepsilon \rightarrow 1$ in $L^\gamma(\T^3)$--strong, we obtain $\rho_0^\varepsilon  v_0^\varepsilon  \rightharpoonup u_0 = v_0 $ in  $L^{2\gamma/(\gamma +1)}(\T^3)$--weak.

As soon as $\gamma>3/2$, the contribution  \cite{HW10} furnishes a  weak solution to \eqref{mhdeps}-\eqref{divbesp} with
\begin{equation}
  \label{reg-ws-mhdeps-P}
  \left\{
  \begin{aligned}
    &\rho^\varepsilon \in L_{\rm loc}^\infty \bigl(\R_+; L^\gamma(\T^3)\bigr)\,, \quad  v^\varepsilon \in   L_{\rm loc}^2\bigl(\R_+; H^1(\T^3)\bigr)\,, \\
    & \sqrt{\rho^\varepsilon}v^\varepsilon \in  L_{\rm loc}^\infty\bigl(\R_+; L^2(\T^3)\bigr)\,, \quad \rho^\varepsilon v^\varepsilon
    \in  L_{\rm loc}^\infty\bigl(\R_+; L^{2\gamma/(\gamma+1)}(\T^3)\big) \cap\mathscr{C}_{\rm loc}\big(\R_+; L_{\rm weak}^{2\gamma/(\gamma+1)}(\T^3)\big) \,,\\
    & B^\varepsilon \in L_{\rm loc}^\infty\bigl(\R_+; L^2(\T^3)\bigr) \cap  \mathscr{ C}_{\rm loc}\bigl(\R_+; L_{\rm weak}^2(\T^3)\bigr)  \cap  L_{\rm loc}^2\bigl(\R_+; H^1(\T^3)\bigr)\,, \qquad  \int_{\T^3} dx \, B^\varepsilon=0 \,.
  \end{aligned}
  \right.
\end{equation}
The mass is conserved
\[
\overline{\rho^\varepsilon} := \frac{1}{|\T^3|} \int_{\T^3} dx\,\rho^\varepsilon  = \overline{\rho_0^\varepsilon} \,,
\]
and thus, using \eqref{ineqC0-P++}, we deduce that  $\overline{\rho^\varepsilon} \rightarrow 1$, as $\varepsilon \rightarrow 0_+$. Moreover, we have two  energy inequalities
\begin{equation}
\label{nrj-ineq-mhdeps}
\mathbbm{E}^\varepsilon(t) + \int_0^t ds \, \mathbbm{D}^\varepsilon(s) \leq  \mathbbm{E}_0^\varepsilon\,, \ \mbox{ a.e. } t \in [0,+\infty)\,,
\end{equation}
with $ \mathbbm{E} \in \{\mathcal{E}_1,\,\mathcal{E}_2\}$, where for $i=1,2$, 
\begin{equation}
\label{nrj-def}
\mathcal{E}_i^\varepsilon(t)=
\int_{\Omega} dx\, \Big( \tfrac 12 \rho^\varepsilon |v^\varepsilon|^2 + \tfrac 12 |B^\varepsilon|^2 + \Pi_i(\rho^\varepsilon)\Big)\,,
\quad
\mathcal{E}_{i0}^\varepsilon=
 \int_{\Omega} dx\, \Big( \tfrac 12\rho_0^\varepsilon |v_0^\varepsilon|^2 + \tfrac 12|B_0^\varepsilon|^2 + \Pi_i(\rho_0^\varepsilon) \Big)\,, 
\end{equation}
\begin{equation}
\label{dissip-D}
 \mathbbm{D}^\varepsilon=
\int_{\Omega} dx\, \big( \mu_\perp^\varepsilon |\nabla_\perp v^\varepsilon|^2 + \mu_\myparallel^\varepsilon |\partial_\myparallel v^\varepsilon|^2 +
\lambda^\varepsilon  |\nabla_\varepsilon \cdot v^\varepsilon|^2 +
\eta_\perp^\varepsilon |\nabla_\perp B^\varepsilon|^2 + \eta_\myparallel^\varepsilon |\partial_\myparallel B^\varepsilon|^2
\big)\,, 
\end{equation}
and
\begin{equation}
\label{defPI-P}
\Pi_1(\rho^\varepsilon)=\frac{a}{\varepsilon^2(\gamma-1)}(\rho^\varepsilon)^\gamma\,, \quad \
\Pi_2(\rho^\varepsilon)=\frac{a}{\varepsilon^2(\gamma-1)}\big( (\rho^\varepsilon)^\gamma
- \gamma \rho^\varepsilon (\overline{\rho^\varepsilon})^{\gamma -1} +(\gamma -1)(\overline{\rho^\varepsilon})^{\gamma}
\big)\,.
\end{equation}
For $ i=1 $, the inequality \eqref{nrj-ineq-mhdeps} is a consequence of straightforward calculation involving  \eqref{mhdeps}. Using the mass conservation, we can check that the inequality \eqref{nrj-ineq-mhdeps} for $ i= 2 $ is equivalent to \eqref{nrj-ineq-mhdeps} for $ i= 1 $. The case $ i= 2 $ is introduced because it allows a better comparison of $ \rho^\varepsilon $ with $ \overline{\rho^\varepsilon} $.

\begin{theorem}[Convergence of MHD to RMHD on a periodic domain] \label{TH-CV-P} Assume $ \Omega = \T^3 $ and $\gamma>3/2$. Consider a sequence $\{(\rho^\varepsilon,v^\varepsilon, B^\varepsilon)\}_{\varepsilon >0}$ of weak solutions to the compressible MHD system \eqref{mhdeps}-\eqref{divbesp} with initial data  $\{(\rho_0^\varepsilon,v_0^\varepsilon, B_0^\varepsilon)\}_{\varepsilon >0}$ as in \eqref{ineqC0-P}. Let us set $\varrho^\varepsilon:=(\rho^\varepsilon-\overline{\rho^\varepsilon})/\varepsilon$. Then, up to a subsequence, the family $\{(\rho^\varepsilon,\varrho^\varepsilon, v^\varepsilon, B^\varepsilon)\}_{\varepsilon >0}$  converges to $(1,\varrho, v, B)$  as indicated below
\begin{equation*}
  \label{cv-th-P}
  \begin{aligned}
    & \rho^\varepsilon \xrightarrow{\quad}  1 \ \ \mbox{\rm in }\  L_{\rm loc}^\infty(\R_+; L^\gamma(\T^3))\!-\!\mbox{\rm strong} \,,  \\
    & \varrho^\varepsilon \xrightharpoonup{\quad}  \varrho  \ \  \mbox{\rm in } \ L_{\rm loc}^\infty(\R_+; L^\kappa(\T^3))\!-\!\mbox{\rm weak--}\ast\,, \quad \kappa=\min\{2,\gamma\} \,,  \\
    &\P_\perp v_\perp^\varepsilon \xrightarrow{\quad} \P_\perp v_\perp=v_\perp  \ \ \mbox{\rm in }\  L_{\rm loc}^2(\R_+; L^p\cap H^s(\T^3))\!-\!\mbox{\rm strong}\,, \quad 1 \leq p < 6\,, \quad 0\leq s<1\,,  \\
    &\Q_\perp v_\perp^\varepsilon  \xrightharpoonup{\quad}  0  \ \ \mbox{\rm in } \ L_{\rm loc}^2(\R_+; H^1(\T^3))\!-\!\mbox{\rm weak}\,,  \\
    & v_\myparallel^\varepsilon \xrightharpoonup{\quad}  v_\myparallel  \ \ \mbox{\rm in } \ L_{\rm loc}^2(\R_+; H^1(\T^3))\!-\!\mbox{\rm weak}\,,  \\
    &B_\perp^\varepsilon \xrightarrow{\quad} B_\perp  \ \ \mbox{\rm in } \ L_{\rm loc}^r(\R_+; L^2(\T^3))\!-\!\mbox{\rm strong}\,, \quad 1 \leq r < \infty\,,  \\
    &B_\myparallel^\varepsilon \xrightharpoonup{\quad} B_\myparallel\ \ \mbox{\rm in } \ L_{\rm loc}^2(\R_+; H^1(\T^3))\!-\!\mbox{\rm weak} \, \cap \,
     L_{\rm loc}^\infty(\R_+; L^2(\T^3))\!-\!\mbox{\rm weak--}\!\ast\,.
  \end{aligned}
\end{equation*}
The limit point $(v, B)$ is a weak solution to the RMHD equations \eqref{rmhdincomp}-\eqref{divperp}-\eqref{rmhd} with initial data
\[
(B_\perp, v_\perp)_{|_{t=0}} = (B_{0 \perp},\P_\perp v_{0 \perp}) \in L^2(\T^3) \, , \qquad (B_\myparallel, v_\myparallel)_{|_{t=0}} = (\mathbbmsl{B}_{0 \myparallel}, v_{0 \myparallel}) \in L^2(\T^3) \, ,
\]
where $ \mathbbmsl{B}_{0\myparallel} $ is as in \eqref{inidatparal}, and it satisfies the following regularity properties
\[
B\in L_{\rm loc}^\infty (\R_+; L^2(\T^3)) \cap L_{\rm loc}^2 (\R_+; H^1(\T^3)) \, , \qquad
v\in L_{\rm loc}^\infty (\R_+; L^2(\T^3)) \cap L_{\rm loc}^2 (\R_+; H^1(\T^3)) \, .
\]
Moreover, the components $ \varrho $ and $ B_\myparallel $ are linked together by the relation $ b \, \varrho + B_\myparallel = 0 $, for a.e. $(t,x) \in ]0,+\infty[ \times \T^3$.
\end{theorem}

%%%%%%%%%%%%%%%%%%%%%%%%%%%%%%
%%%%%%%%%%%%%%%%%%%%%%%%%%%%%%
  
\subsection{The whole space case}
\label{ss:rR3}
This is when  $\Omega=\R^3$. In order to define weak solutions in the whole space, we need to introduce the following special type of Orlicz spaces $L_q^p(\Omega)$ (see Appendix~A of \cite{Lio98} for more details on these spaces),
\begin{equation}
  \label{def_orlicz}
  L_q^p(\Omega) = \big\{
  f\in L_{\rm loc}^1(\Omega) \ \big | \ f\mathbbm{1}_{\{|f|\leq \delta\}} \in L^q(\Omega), \ \  f\mathbbm{1}_{\{|f|> \delta\}} \in L^p(\Omega), \ \delta >0
  \big\}\,,
\end{equation}
where the function $\mathbbm{1}_S$ denotes the indicator function of the set $S$. Obviously $L_p^p(\Omega)\equiv L^p(\Omega)$.

The functional framework is based on \cite{HW10,LM98}. Select initial data satisfying
\begin{equation}
  \label{reg-ic-mhdep-R}
   \left\{
   \begin{aligned}
     & \rho^\varepsilon_{|_{t=0}}=\rho_0^\varepsilon\in L_{\rm loc }^1(\R^3)\,, 
     \quad \rho_0^\varepsilon -1 \in L_2^\gamma(\R^3)\,,  \quad  \gamma>3/2\,,
     \quad \rho_0^\varepsilon \geq 0\,,
     \quad \sqrt{\rho_0^\varepsilon} v_0^\varepsilon \in L^2(\R^3)\,, \\
     &(\rho^\varepsilon v^\varepsilon)_{|_{t=0}} = m_0^\varepsilon =
     \left\{
     \begin{aligned}
       & \rho_0^\varepsilon v_0^\varepsilon & \mbox{if } \ \rho_0^\varepsilon\neq 0\\
       & 0  & \mbox{if }\ \rho_0^\varepsilon= 0
     \end{aligned}
     \right\}
     \in L_{\rm loc}^{1}(\R^3)\,, \\
     & B^\varepsilon_{|_{t=0}}=B_0^\varepsilon \in L^2(\R^3)\,, \qquad \nabla_\varepsilon \cdot B_0^\varepsilon=0 \, , \qquad \int_{\R^3} dx \, B_0^\varepsilon=0 \,, \\
      & \rho_0^\varepsilon \xrightarrow{\quad} 1\,, \quad v_0^\varepsilon \xrightarrow{\quad} 0\,, \quad
    B_0^\varepsilon \xrightarrow{\quad} 0\,, \quad \mbox{as } |x| \xrightarrow{\quad} \infty\,. 
   \end{aligned}
  \right.
\end{equation}
We assume that these regularity assumptions are uniform with respect to $\varepsilon$. Furthermore, given a constant $C_0$ (not depending on $\varepsilon$), we impose 
\begin{equation}
  \label{ineqC0-R}
  \frac 12 \int_{\R^3} dx\, \big( \rho_0^\varepsilon |v_0^\varepsilon|^2 + |B_0^\varepsilon|^2\big) + \frac{a}{\varepsilon^2(\gamma-1)}
  \int_{\R^3} dx\, \big( (\rho_0^\varepsilon)^\gamma - \gamma  \rho_0^\varepsilon 
  + \gamma -1\big) \leq C_0\,.
\end{equation}
This bound gives access to weak compactness. Modulo the extraction of subsequences (which are not specified) we can say that $\sqrt{\rho_0^\varepsilon } v_0^\varepsilon$ and $B_0^\varepsilon$ converge weakly in $L^2(\R^3)$ to $u_0$ and $B_0$ respectively. From \eqref{ineqC0-P}, we will show (see the proof of Lemma~\ref{lem:cv-rho-R}) the subsequent results. First, we will  obtain (uniformly in $\varepsilon$) the bounds $ \rho_0^\varepsilon  \in  L_{\rm loc }^\gamma(\R^3)$, $ \varrho_0^\varepsilon \in  L_2^\kappa \cap L_{\rm loc}^\kappa (\R^3)$, with $\kappa=\min\{2,\gamma\}$, as well as $ \rho_0^\varepsilon v_0^\varepsilon \in L_{\rm loc}^{2\gamma/(\gamma+1)} (\R^3)$. Second, we will obtain $\rho_0^\varepsilon \rightarrow 1$ in $L_2^\gamma\cap L_{\rm loc}^\gamma(\R^3)$--strong, and $\varrho_0^\varepsilon \rightharpoonup \varrho_0$ in $L_{\rm loc}^\kappa(\R^3)$--weak. Moreover, using $\sqrt{\rho_0^\varepsilon } v_0^\varepsilon \rightharpoonup u_0$ in  $L^2(\R^3)$--weak and $\rho_0^\varepsilon \rightarrow 1$ in $L_{\rm loc}^\gamma(\R^3)$--strong, we obtain $\rho_0^\varepsilon v_0^\varepsilon \rightharpoonup u_0 = v_0$ in $ L_{\rm loc}^{2\gamma/(\gamma+1)}(\R^3)$--weak.

As soon as $\gamma>3/2$, the contribution  \cite{HW10} furnishes a  weak solution to \eqref{mhdeps}-\eqref{divbesp} with
\begin{equation}
  \label{reg-ws-mhdeps-R}
  \left\{
  \begin{aligned}
    &\rho^\varepsilon \in L_{\rm loc}^\infty(\R_+; L_{\rm loc}^\gamma(\R^3))\,, \quad \rho^\varepsilon -1 \in L_{\rm loc}^\infty(\R_+; L_2^\gamma(\R^3))\,,
    \quad  \nabla v^\varepsilon \in   L_{\rm loc}^2(\R_+; L^2(\R^3))\,, \\
    & \sqrt{\rho^\varepsilon}v^\varepsilon \in  L_{\rm loc}^\infty(\R_+; L^2(\R^3))\,, \quad \rho^\varepsilon v^\varepsilon
    \in  L_{\rm loc}^\infty\big(\R_+; L_{\rm loc}^{2\gamma/(\gamma+1)}(\R^3)\big) \cap\mathscr{C}_{\rm loc}\big(\R_+; L_{\rm loc \, weak}^{2\gamma/(\gamma+1)}(\R^3)\big) \,,\\
    & B^\varepsilon \in L_{\rm loc}^\infty(\R_+; L^2(\R^3))\cap  \mathscr{C}_{\rm loc}(\R_+; L_{\rm weak}^2(\R^3))  \cap  L_{\rm loc}^2(\R_+; H^1(\T^3))\,, \quad  \int_{\T^3} dx \, B^\varepsilon=0 \,,\\
    & \rho^\varepsilon \xrightarrow{\quad} 1\,, \quad v^\varepsilon \xrightarrow{\quad} 0\,, \quad
    B^\varepsilon \xrightarrow{\quad} 0\,, \quad \mbox{as } |x| \xrightarrow{\quad} \infty\,.
  \end{aligned}
  \right.
\end{equation}
Moreover, we have the energy inequality \eqref{nrj-ineq-mhdeps} with $\mathbbm E^\varepsilon= \mathcal{E}_3^\varepsilon$ given by the formula \eqref{nrj-def} and $\Pi_i=\Pi_3$, where $\Pi_3$ is defined by
\begin{equation}
\label{defPI-R}
\Pi_3(\rho^\varepsilon)=\frac{a}{\varepsilon^2(\gamma-1)}\big( (\rho^\varepsilon)^\gamma
- \gamma \rho^\varepsilon + \gamma - 1\big)\,.
\end{equation}
This energy inequality is the consequence of straightforward calculation involving \eqref{mhdeps}.

\begin{theorem}[Convergence of MHD to RMHD on the whole space] Assume $\Omega=\R^3$ and $\gamma>3/2$. Consider $\{(\rho^\varepsilon,v^\varepsilon, B^\varepsilon)\}_{\varepsilon >0}$ a sequence of weak solutions to the compressible MHD system \eqref{mhdeps}-\eqref{divbesp} with initial data  $\{(\rho_0^\varepsilon,v_0^\varepsilon, B_0^\varepsilon)\}_{\varepsilon >0}$ as in \eqref{ineqC0-R}. Let us set $\varrho^\varepsilon:=(\rho^\varepsilon - 1)/\varepsilon$.  Then, up to a subsequence, the family $\{(\rho^\varepsilon,\varrho^\varepsilon, v^\varepsilon, B^\varepsilon)\}_{\varepsilon >0}$  converge to $(1,\varrho, v, B)$  as indicated below  
  \begin{equation*}
    \label{cv-th-R}
    \begin{aligned}
      & \rho^\varepsilon \xrightarrow{\quad}  1 \ \ \mbox{\rm in }\  L_{\rm loc}^\infty(\R_+; L_2^\gamma \cap L_{\rm loc}^\gamma \cap H^{-\upalpha}(\R^3))\!-\!\mbox{\rm strong} \,,  \quad \upalpha \geq 1/2\,,\\
      &\varrho^\varepsilon \xrightharpoonup{\quad}  \varrho  \ \  \mbox{\rm in } \ L_{\rm loc}^\infty(\R_+;  L_{\rm loc}^\kappa \cap H^{-\upalpha}(\R^3))\!-\!\mbox{\rm weak--}\ast\,, \quad \kappa=\min\{2,\gamma\} \,, \quad \upalpha \geq 1/2\,,  \\
      &\P_\perp v_\perp^\varepsilon \xrightarrow{\quad} \P_\perp v_\perp=v_\perp  \ \ \mbox{\rm in }\  L_{\rm loc}^2(\R_+; L_{\rm loc}^p\cap H_{\rm loc}^s(\R^3))\!-\!\mbox{\rm strong}\,, \quad 1 \leq p < 6\,, \quad 0\leq s <1\,,  \\
      &\Q_\perp v_\perp^\varepsilon  \xrightharpoonup{\quad}  0  \ \ \mbox{\rm in } \ L_{\rm loc}^2(\R_+; H^1(\R^3))\!-\!\mbox{\rm weak}\,, \\
      & v_\myparallel^\varepsilon \xrightharpoonup{\quad}  v_\myparallel  \ \ \mbox{\rm in } \ L_{\rm loc}^2(\R_+; H^1(\R^3))\!-\!\mbox{\rm weak}\,,  \\
      &B_\perp^\varepsilon \xrightarrow{\quad} B_\perp  \ \ \mbox{\rm in } \ L_{\rm loc}^r(\R_+; L_{\rm loc}^2(\R^3))\!-\!\mbox{\rm strong}\,, \quad 1 \leq r < \infty\,,  \\
      &B_\myparallel^\varepsilon \xrightharpoonup{\quad} B_\myparallel  \ \ \mbox{\rm in } \ L_{\rm loc}^2(\R_+; H^1(\R^3))\!-\!\mbox{\rm weak}\, \cap \,
      L_{\rm loc}^\infty(\R_+; L^2(\R^3))\!-\!\mbox{\rm weak--}\!\ast\,. 
    \end{aligned}
  \end{equation*}
  The limit point $(v, B)$ is a weak solution to the RMHD equations \eqref{rmhdincomp}-\eqref{divperp}-\eqref{rmhd} with initial data
  \[
  (B_\perp, v_\perp)_{|_{t=0}} = (\P_\perp B_{0 \perp},\P_\perp v_{0 \perp}) \in L^2(\R^3) \, , \qquad (B_\myparallel, v_\myparallel)_{|_{t=0}} = (
 \mathbbmsl{B}_{0\myparallel}, v_{0 \myparallel}) \in L^2(\R^3) \, ,
  \]
where $ \mathbbmsl{B}_{0\myparallel} $ is as in \eqref{inidatparal}, and it satisfies the following regularity properties
  \[
  B\in L_{\rm loc}^\infty (\R_+; L^2(\R^3)) \cap L_{\rm loc}^2 (\R_+; H^1(\R^3)) \, , \qquad
  v\in L_{\rm loc}^\infty (\R_+; L^2(\T^3)) \cap L_{\rm loc}^2 (\R_+; H^1(\R^3)) \, .
  \]
  Moreover, the components $ \varrho $ and $ B_\myparallel $ are linked together by the relation $ b \, \varrho + B_\myparallel = 0 $, for a.e. $(t,x) \in ]0,+\infty[ \times \R^3$.
  \label{TH-CV-R}
\end{theorem}

\section{Physical motivations and scaling}
\label{s:scaling}
The dimensional magnetohydrodynamic equations reads 
\begin{equation}
\left \{      
\begin{aligned}    
&\partial_{\text{t}} \uprho + \nabla \cdot (\uprho \text{v}) = 0 \,, \\
&\partial_{\text{t}} (\uprho \text{v}) + \nabla \cdot (\uprho \text{v} \otimes \text{v})
  + \nabla \text{p} + \text{B} \times (\nabla \times \text{B} ) 
  - \upmu_\perp \Delta_\perp \text{v} - \upmu_\myparallel\Delta_\myparallel \text{v}
- \uplambda\nabla (\nabla \cdot \text{v}) =0 \,, \\
&\partial_{\text{t}} \text{B} + \nabla \times ( \text{B} \times \text{v})
- \upeta_\perp \Delta_\perp \text{B} - \upeta_\myparallel \Delta_\myparallel \text{B} =0\,,
\end{aligned}
\right.
\label{undim-mhd}
\end{equation}
with the divergence-free condition $\nabla \cdot \text{B}=0$, and the barotropic closure $\text{p}=\text{p}(\uprho)=\mathfrak a \uprho^\gamma$, $\gamma>1$. The triplet $ (\uprho,\, \text{v},\, \text{B})=(\uprho,\, \text{v},\, \text{B})(\text{t},\text{x}_\perp,\text{x}_\myparallel)\in \R_+\times \R^3 \times\R^3$ denotes respectively the dimensional fluid density, fluid velocity, and magnetic field. The variable $\text{t}$ represents the dimensional time variable, while the two-dimensional (resp. one-dimensional) variable $\text{x}_\perp$ (resp. $\text{x}_\myparallel$) represents the perpendicular (resp. parallel) dimensional space variable.

\subsection{Large aspect ratio framework}
\label{ss:LAR}
Anisotopic plasmas with a strong background magnetic field are ubiquitus in astrophysical, space and fusion sciences. As an example, for fusion plasmas, the straight rectangular tokamak model involves a very long periodic column, whose section is a small periodic rectangle. The corresponding geometry and scalings are detailed carefully in \cite{Str76}. Another example comes from various astrophysical plasmas such as the solar wind or the magnetosheath for which the underlying RMHD ordering is precisely described in \cite{SCD09}. In order to obtain the dimensionless MHD equations \eqref{mhdeps}, we must first  adimensionalize equations \eqref{undim-mhd}, and then choose a scaling. Putting dimensions into the ``bar'' quantities, we define the  dimensionless unknowns and variables as $\text{t}=\bar{\text t}\, t$, $\text{x}_\perp=\bar{\text{x}}_\perp\,x_\perp$,  $\text{x}_\myparallel=\bar{\text{x}}_\myparallel\,x_\myparallel$, $\upmu_\perp=\bar{\upmu}_\perp\, \mu_\perp$,  $\upmu_\myparallel=\bar{\upmu}_\myparallel\, \mu_\myparallel$, $\uplambda=\bar{\uplambda}\, \lambda$, $\upeta_\perp=\bar{\upeta}_\perp\, \eta_\perp$,  $\upeta_\myparallel=\bar{\upeta}_\myparallel\, \eta_\myparallel$, $\mathfrak a=\bar{\mathfrak a}\, a$, $\uprho = \bar{\uprho}\, \rho$, $\text{v}=\bar{\text{v}}\, v$, and  $\text{B}=\bar{\text{B}}\, B$.  From this, and the barotropic state law, we deduce the dimensionless pressure as $\text{p}=\bar{\text{p}}\, p$ with $\bar{\text{p}}=\bar{\mathfrak a} \bar{\uprho}^\gamma$ and $p=a\rho^\gamma$. We also define important physical quantities such as the Alfv\'en velocity $v_A:=\bar{\text{B}}/\sqrt{\bar{\uprho}}$, the sound velocity $v_s:=\sqrt{\gamma\bar{\text{p}}/ \bar{\uprho}}$, the parameter $\beta:=\bar{\text{p}}/|\bar{\text{B}}|^2=v_s^2/(\gamma v_A^2)$, the Alfv\'en number $\varepsilon_A:=\bar{\text{v}}/v_A$ and  the Mach number $\varepsilon_M:=\bar{\text{v}}/v_s=\varepsilon_A/(\sqrt{\gamma \beta})$. Here, we suppose that the parameter $\beta$ is of order one, and thus we set $\beta=1$. This configuration is called the \textit{high} $\beta$ ordering \cite{Str77}, and often appears in space plasmas \cite{BNS98, KW97, OMD17}. In other situations, such as plasmas of tokamaks \cite{DA83, IMS83, KP74}, the parameter  $\beta$ can be relatively small; this is the so-called \textit{low} $\beta$ regim. Indeed, since $\beta$ measures the ratio of the fluid pressure to the magnetic pressure, a magnetically well-confined plasmas is achieved for low $\beta$. Since here we choose $\beta=1$, we have $v_s=\sqrt{\gamma} v_A \simeq v_A$.

In order to understand which parts of the solution of the MHD equations \eqref{mhdeps} are eliminated in the reduced model \eqref{rmhdincomp}-\eqref{inidatparal}, we now recall the different types of (linear) waves propagating in a plasma governed by the MHD equations \eqref{undim-mhd}. Dropping viscosities and resistivities terms, it is well-known \cite{RB96} that the system \eqref{undim-mhd} is hyperbolic, but not strictly hyperbolic since some eigenvalues may coincide. Linearizing the system \eqref{undim-mhd} around the constant stationary solution $(\bar \uprho,0, \bar{\text{B}}\,\mathbbl{b})$, where $\mathbbl{b}$ is a unit vector, we obtain a linear system whose the Jacobian has real eigenvalues \cite{GP04}. The set of MHD eigenvalues and associated waves can be splitted into three groups. Introducing the unit vector $\mathbbl{n}$ as the direction of propagation of any wave, the sound speed  $V_s:=\sqrt{\gamma \text{p}/\uprho}=\sqrt{a}v_s$ (with $\rho=1$) and the Alfv\'en  velocity $V_A:=|\text{B}|/\sqrt{\uprho}=v_A$ (with $|B|=1$), these three groups are  \cite{GP04}:    
\begin{itemize}
\item Fast magnetosonic waves:
  \[
  \lambda_F^\pm= \pm\, \mathcal C_F\,, \quad \mathcal C_F^2=\frac12 \Big(V_s^2+V_A^2 + \sqrt{(V_s^2+V_A^2)^2 -4V_s^2V_A^2(\mathbbl{b}\cdot\mathbbl{n})^2} \Big)\,.
  \]
\item Alfv\'en waves:
  \[
  \lambda_A^\pm= \pm  \,\mathcal C_A\,, \quad  \mathcal C_A^2=V_A^2(\mathbbl{b}\cdot\mathbbl{n})^2\,.
  \]
\item Slow magnetosonic waves:
  \[
  \lambda_S^\pm= \pm \, \mathcal C_S\,, \quad  \mathcal C_S^2=\frac12 \Big(V_s^2+V_A^2 - \sqrt{(V_s^2+V_A^2)^2 -4V_s^2V_A^2(\mathbbl{b}\cdot\mathbbl{n})^2} \Big)\,.
  \]  
\end{itemize}
Since here $\mathbbl{b}:=e_\myparallel$, Alfv\'en waves cannot propagate in the perpendicular direction to $e_\myparallel$. Indeed, it is well-known \cite{GP04} that Alfv\'en waves propagate mainly along the direction ($\mathbbl{b}:=e_\myparallel$) of the ambiant magnetic field. For a wave propagating in the perpendicular direction to $e_\myparallel$, we obtain  $\lambda_F^\pm=\pm (V_s^2+V_A^2)^{1/2}\simeq \pm V_A \simeq \pm V_s$, whereas   $\lambda_S^\pm=0$. Note that in dimensionless variables we have $\lambda_F^\pm=\pm\sqrt{b+1}$, with $b=a\gamma$. Indeed, normalizing the velocity to the Alfv\'en velocity $v_A$ and taking $\beta=1$  in $\lambda_F^\pm=\pm v_A\sqrt{\beta a \gamma+1}$, we obtain the desired result. In order to understand now the nature of the waves that are filtered out from the singular part of the linear system \eqref{mhdepssing}, we rewrite it in the fast time variable $\text{t}$ to obtain
\[
\partial_{\text{t}} \varrho + \nabla_\perp \cdot v_\perp=0\,, \qquad
\partial_{\text{t}} v_\perp + \nabla_\perp (b\varrho + B_\myparallel)=0\,, \qquad
\partial_{\text{t}} B_\myparallel + \nabla_\perp \cdot v_\perp=0\,.
\]
With $U:={}^t(\varrho,v_1, v_2,  B_\myparallel)$, the previous system can be recast as $\partial_{\text{t}} U + (A_1\partial_{x_1}+A_1\partial_{x_2}) U=0$, where the  matrices $A_i$ have constant coefficients depending on $b$.
With $\mathbbl{n}_\perp={}^t(\mathbbl{n}_1,\mathbbl{n}_2)$ a unit vector in the perpendicular direction, the matrix $\mathcal{A}:=\mathbbl{n}_1 A_1+ \mathbbl{n}_2 A_2$ is diagonalizable with the real eigenvalues $\lambda_0(\mathcal{A})=0$ (of multiplicity two), $\lambda_+(\mathcal{A})=\sqrt{b+1}$, and $\lambda_-(\mathcal{A})=-\sqrt{b+1}$. Then, the waves associated with the singular part of the linear system \eqref{mhdepssing} are the transverse (linear) fast magnetosonic waves.

Therefore, here, we aim at filtering out the fast dynamics associated with the perpendicular fast magnetosonic waves, and keep the dynamics of waves which propagate at a speed slower than the perpendicular fast magnetosonic waves $ \mathcal C_F \simeq v_A$. Defining the time $\uptau_\perp$ as the time needed by a fast magnetosonic waves to cross the device in the perpendicular direction, we then have $\uptau_\perp v_A=\bar{\text{x}}_\perp$. Since we want to describe the dynamics on a time scale longer that $\uptau_\perp$, we set $\bar{\text t}=\uptau_\perp/\varepsilon$, with $\varepsilon \ll 1$. This is equivalent to describe the dynamics of waves which propagate with velocity slower than $ \mathcal C_F$ (or $v_A$). Hence, we have $\bar{\text{v}}=\varepsilon v_A$, $\varepsilon_A=\varepsilon$,  and $\varepsilon_M=\varepsilon/(\sqrt{\gamma \beta})\simeq \varepsilon$. Moreover, we suppose a strong anisotropy between the perpendicular and parallel direction, that is $\bar{\text{x}}_\perp/\bar{\text{x}}_\myparallel=\varepsilon$. In other words the dimensional gradient $\nabla_{\text{x}}$ becomes the anisotropic dimensionless gradient $\nabla_\varepsilon$ of  \eqref{nablaeps}. In addition, we suppose the presence of a strong constant background magnetic field in the parallel direction ($B=e_\myparallel +\varepsilon B^\varepsilon$). Since $\varepsilon$ is present in the resulting dimensionless system, the  velocity field $v$ and the density $\rho$ will depend on $\varepsilon$, hence we set $v=v^\varepsilon$ and $\rho=\rho^\varepsilon$. Finally, it remains to choose some scalings with respect to the small parameter $\varepsilon$ for the dimensionless viscosities and resistivities. We choose $\{\mu_\perp=\varepsilon \mu_\perp^\varepsilon,\,  \mu_\myparallel= \mu_\myparallel^\varepsilon/\varepsilon,\, \lambda=\varepsilon\lambda^\varepsilon\}$, where viscosities $\{\mu_\perp^\varepsilon,\, \mu_\myparallel^\varepsilon,\, \lambda^\varepsilon\}$ satisfy \eqref{visco-norm}, and $\{\eta_\perp=\varepsilon \eta_\perp^\varepsilon,\,  \eta_\myparallel= \eta_\myparallel^\varepsilon/\varepsilon\}$, where resistivities $\{\eta_\perp^\varepsilon,\, \eta_\myparallel^\varepsilon\}$ satisfy \eqref{resis-norm}.

All the above considerations allows us to pass from the dimensional MHD equations \eqref{undim-mhd} to the dimensionless ones \eqref{mhdeps}.

\subsection{Nonlinear optics framework} Conducting fluids are traversed by electromagnetic waves, which can interact with the medium in various ways. These phenomena can be modeled by adjusting the dimensionless parameters to account for special regimes, and by incorporating (high frequency) oscillating source terms or equivalently (high frequency) oscillating initial data into the equations. Here, we choose viscosities and resistivities which accommodate the propagation of oscillating waves with wavelengths approximately $\varepsilon$. For this, we impose viscosities $\{\upmu_\perp=\varepsilon^2 \mu_\perp^\varepsilon,\,  \upmu_\myparallel= \mu_\myparallel^\varepsilon,\, \uplambda=\varepsilon^2 \lambda^\varepsilon\}$, where dimensionless viscosities $\{\mu_\perp^\varepsilon,\, \mu_\myparallel^\varepsilon,\, \lambda^\varepsilon\}$ satisfy \eqref{visco-norm}, and resistivities $\{\upeta_\perp=\varepsilon^2 \eta_\perp^\varepsilon,\,  \upeta_\myparallel= \eta_\myparallel^\varepsilon\}$, where dimensionless resistivities $\{\eta_\perp^\varepsilon,\, \eta_\myparallel^\varepsilon\}$ satisfy \eqref{resis-norm}. We then look for solution like
\begin{align}
  \left(\begin{tabular}{c}
  $\uprho/\bar{\uprho}$\\ 
  $\text{v}/\bar{\text{v}}$\\
  $\text{B}/\bar{\text{B}}$
\end{tabular}
\right)
(\text{t},\text{x}_\perp,\text{x}_\myparallel) 
&= \left(\begin{tabular}{c}
  $\rho^\varepsilon\big(\text{t},{\varepsilon}^{-1}\,\text{x}_\perp,\text{x}_\myparallel\big)$\vspace{2pt}\\  
  $\varepsilon \, \,v^\varepsilon\big(\text{t}, {\varepsilon}^{-1}\,\text{x}_\perp ,\text{x}_\myparallel\big)$\vspace{2pt}\\
  $1+\varepsilon\, B^\varepsilon\big(\text{t},{\varepsilon}^{-1}\,\text{x}_\perp,\text{x}_\myparallel\big)$
\end{tabular}
\right)=
\left(\begin{tabular}{c}
  $ \overline{\rho^\varepsilon}$\\ 
  $ 0$\\
  $ 1$
\end{tabular}
\right)
+ \varepsilon
\left(\begin{tabular}{c}
  $ \varrho^\varepsilon$\\ 
  $ v^\varepsilon$\\
  $B^\varepsilon$
\end{tabular}
\right)\big(\text{t},{\varepsilon}^{-1}\,\text{x}_\perp,\text{x}_\myparallel\big) \nonumber\\
&=
\left(\begin{tabular}{c}
  $\overline{\rho^\varepsilon}$\\ 
  $ 0$\\
  $ 1$
\end{tabular}
\right)
+ \varepsilon
\left(\begin{tabular}{c}
  $\ \varrho^\varepsilon$\\ 
  $v^\varepsilon$\\
  $B^\varepsilon$
\end{tabular}
\right)\big(t,x_\perp,x_\myparallel\big) \,, \label{nlov}
\end{align}  
where $\varrho^\varepsilon =(\rho^\varepsilon -\overline{\rho^\varepsilon})/\varepsilon$. Plugging \eqref{nlov} into \eqref{undim-mhd} leads to \eqref{mhdeps}. Therefore, we investigate the dynamics of a magnetized plasma near a fixed large constant magnetic field  where anisotropic oscillations in space can develop. The first term of the right-hand side of \eqref{nlov}, which is of order of unity, is a stationary solution of \eqref{undim-mhd}. The second term of the right-hand side of \eqref{nlov} is the perturbation, which is of small amplitude ($\varepsilon \ll 1$) and of high frequency ($ \varepsilon^{-1}\gg 1$). Such a framework belongs to the so called weakly nonlinear geometric optics regim.

%%%%%%%%%%%%%%%%%%%%%%%%

\section{Asymptotic analysis in a periodic domain}
\label{s:AAPD}
This section is devoted to the proof of Theorem~\ref{TH-CV-P}. First, we obtain some weak compactness properties for the sequences, $\varrho^\varepsilon$,  $B_\myparallel^\varepsilon$, $v^\varepsilon$,  $\rho^\varepsilon v^\varepsilon$, and $\Q_\perp v_\perp^\varepsilon$, and strong ones for the sequences $\rho^\varepsilon$, $B_\perp^\varepsilon$, $\P_\perp v_\perp^\varepsilon$,  $\P_\perp(\rho^\varepsilon v_\perp^\varepsilon)$ and $(\rho^\varepsilon v^\varepsilon - v^\varepsilon)$. Using these compactness results, we justify the passage to the limit, in order, in the equations of $\rho^\varepsilon$, $\rho^\varepsilon v_\perp^\varepsilon$, $\rho^\varepsilon v_\myparallel^\varepsilon$, and $B_\myparallel^\varepsilon$ (or equivalently $\varrho^\varepsilon$). For the equations of  $\rho^\varepsilon v_\perp^\varepsilon$, we use the unitary group method, while for the equations of $\rho^\varepsilon v_\myparallel^\varepsilon$, and $B_\myparallel^\varepsilon$, we use some particular cancellations and a compactness argument (Lemma~\ref{lem:compactness} of Appendix~\ref{s:tools}).

%%%%%%%%%%%%%%%%%%%%%%%%%%%%%

\subsection{Compactness of $\rho^\varepsilon $ and  $\varrho^\varepsilon $ }
Here, we aim at proving the following lemma.
\begin{lemma}
  \label{lem:cv-rho} Assume $ \gamma>3/2$. The sequences ${\rho^\varepsilon}$ and  $\varrho^\varepsilon:=(\rho^\varepsilon -\overline{\rho^\varepsilon})/\varepsilon$  satisfy the following properties.
  \begin{align*}
    & \rho^\varepsilon \xrightarrow{\quad} 1 \ \ \mbox{ \rm in } \  L_{\rm loc}^\infty(\R_+; L^\gamma(\T^3))\!-\!\mbox{\rm strong}\,,  \\
    & \varrho^\varepsilon \xrightharpoonup{\quad} \varrho\ \  \mbox{ \rm in } \  L_{\rm loc}^\infty(\R_+; L^\kappa(\T^3))\!-\!\mbox{\rm weak--}\ast\,, \quad \kappa=\min\{2,\gamma\}\,.
  \end{align*}
\end{lemma}
\begin{proof}
  Let us start with $\rho^\varepsilon$. From energy inequality \eqref{nrj-ineq-mhdeps}-\eqref{defPI-P} with the pressure term $\Pi_1$, we already know that $\Pi_1(\rho^\varepsilon)$ is  bounded in $ L_{\rm loc}^\infty(\R_+; L^1(\T^3))$, uniformly with respect to $\varepsilon$. Thus, by weak compactness, we have  $\rho^\varepsilon \rightarrow 1$ in $L_{\rm loc}^\infty(\R_+; L^\gamma(\T^3))\!-\!\mbox{\rm weak--}\ast$. In addition, since $\overline{\rho^\varepsilon}=\overline{\rho_0^\varepsilon}\rightarrow 1$ as $\varepsilon\rightarrow 0$, for $\varepsilon$ small enough we have $\overline{\rho^\varepsilon}\in (1/2,3/2)$. Then, using Lemma~\ref{lem:convex}, we claim that there exists $\eta=\eta_\delta=\eta(\gamma,\delta) >0$, such that for $|x-\overline{\rho^\varepsilon}|\geq \delta$ and $x\geq 0$, we have
  \begin{equation}
    \label{eqn:cvx-1}
    x^\gamma  - \gamma x\bar{x}^{\gamma -1} + (\gamma-1)\bar{x}^\gamma \geq \eta_\delta \big |x-\overline{\rho^\varepsilon}\big |^\gamma.   
  \end{equation}
  Indeed, using  Lemma~\ref{lem:convex} with  $\bar x = \overline{\rho^\varepsilon} \in  (1/2,3/2)$, we obtain, $\eta=\nu_3$, for $1<\gamma < 2$, and $3/2<R<x$;  $\eta=\delta^{2-\gamma}\nu_2$,  for $1< \gamma < 2$, and $x\leq R$; and  $\eta=\nu_1\sup_{x\in\T^3}|x-\overline{\rho^\varepsilon}|^{2-\gamma}>0$,  for $\gamma \geq 2$, and $x\geq 0$. Therefore, using \eqref{eqn:cvx-1}, inequality $(a/2 +b/2)^\gamma \leq (a^\gamma + b^\gamma)/2$ (by convexity of $x\mapsto x^\gamma$), and energy inequality \eqref{nrj-ineq-mhdeps}-\eqref{defPI-P} with the pressure term $\Pi_2$, we obtain
  \begin{multline*}
    \sup_{t\geq 0}\int_{\T^3} dx\, |\rho^\varepsilon-1|^\gamma \leq 2^{\gamma-1}\big|\T^3\big|\big |\overline{\rho^\varepsilon}-1\big|^\gamma+
    2^{\gamma-1} \sup_{t\geq 0}\bigg\{ \int_{|\rho^\varepsilon-\overline{\rho^\varepsilon}| \leq \delta} dx
    +\int_{|\rho^\varepsilon-\overline{\rho^\varepsilon}| > \delta} dx \bigg\}
    \big |\rho^\varepsilon-\overline{\rho^\varepsilon}\big|^\gamma
    \\ \leq  2^{\gamma-1}\bigg\{\big|\T^3\big| \big |\overline{\rho^\varepsilon}-1\big|^\gamma+ \big|\T^3\big|\delta^\gamma +  \frac{C_0 \varepsilon^2}{\eta_\delta} \bigg\}.
  \end{multline*}
  In the previous estimate, taking first $\varepsilon \rightarrow 0$, and then  $\delta\rightarrow 0$, lead to the convergence of $\rho^\varepsilon$ as stated in Lemma~\ref{lem:cv-rho}. We continue with   $\varrho^\varepsilon:=(\rho^\varepsilon -\overline{\rho^\varepsilon})/\varepsilon$. Using  Lemma~\ref{lem:convex} with  $\bar x = \overline{\rho^\varepsilon} \in  (1/2,3/2)$ and $x=\rho^\varepsilon$, and using  energy inequality \eqref{nrj-ineq-mhdeps}-\eqref{defPI-P} with the pressure term $\Pi_2$, there exists a constant $C$ independent of $\varepsilon$ such that
  \begin{equation}
    \label{eqn:varrhoeps-est}\left \{
    \begin{aligned}
    &\mbox{if } \ \gamma \geq 2,  \ \ \|\varrho^\varepsilon\|_{L_{\rm loc}^\infty(\R_+;L^2(\T^3))} \leq C\,,  \\
    & \mbox{if } \ \gamma < 2,\ \forall R\in\big(\tfrac 3 2,+\infty\big),\ \|\varrho^\varepsilon \mathbbm{1}_{\rho^\varepsilon<R}\|_{L_{\rm loc}^\infty(\R_+;L^2(\T^3))} \leq C\,, \
    \|\varrho^\varepsilon \mathbbm{1}_{\rho^\varepsilon\geq R}\|_{L_{\rm loc}^\infty(\R_+;L^\gamma(\T^3))}\leq C\varepsilon^{(2/\gamma)-1}\,.
    \end{aligned}
    \right.
  \end{equation}
  Using this last estimate, $\varrho^\varepsilon$ is bounded, uniformly with respect to $\varepsilon$, in $L_{\rm loc}^\infty(\R_+;L^\kappa(\T^3))$, with $\kappa=\min\{2,\gamma\}$. Hence, weak compactness leads to the convergence of $\varrho^\varepsilon$ stated in Lemma~\ref{lem:cv-rho}.
\end{proof}

%%%%%%%%%%%%%%%%%%%%%%%%%
  
\subsection{Compactness of $B^\varepsilon $}
Here, we aim at proving the following lemma.
\begin{lemma}
  \label{lem:cv-Bperp}
  The sequence ${B^\varepsilon}$ satisfies the following properties.
  \begin{align*}
    & B^\varepsilon \xrightharpoonup{\quad} B \ \ \mbox{ \rm in } \  L_{\rm loc}^2(\R_+;L^6\cap H^1(\T^3))\!-\!\mbox{\rm weak} \cap  L_{\rm loc}^\infty(\R_+ ;L^2(\T^3))\!-\!\mbox{\rm weak--}\ast\,,\\
    & \nabla_\varepsilon \cdot B^\varepsilon   \xrightharpoonup{\quad}\nabla_\perp \cdot B_\perp =0 \ \ \mbox{ \rm in } \  L_{\rm loc}^2(\R_+; L^2(\T^3))\!-\!\mbox{\rm weak}\,,\\
    &  B_\perp^\varepsilon \xrightarrow{\quad} B_\perp \ \ \mbox{ \rm in } \  L_{\rm loc}^r(\R_+; L^2(\T^3))\!-\!\mbox{\rm strong}\,, \quad 1\leq r < +\infty \,.
  \end{align*}
\end{lemma}

\begin{proof}
  The first limit of Lemma~\ref{lem:cv-Bperp}, comes on the one hand from weak compactness, and on the other hand from  energy inequality \eqref{nrj-ineq-mhdeps}-\eqref{defPI-P} with the pressure term $\Pi_2$, and the continuous Sobolev embeddings $H^1(\T^3) \hookrightarrow L^6(\T^3)$, which implies that $B^\varepsilon$ is  bounded in $L_{\rm loc}^2(\R_+;  H^1\cap L^6(\T^3)) $ and in $ L_{\rm loc}^\infty(\R_+; L^2(\T^3))$, uniformly with respect to $\varepsilon$.\\
We continue with the second assertion. Since $B^\varepsilon$ is uniformly bounded in $L_{\rm loc}^2(\R_+; H^1(\T^3))$ and $\nabla_\perp \cdot B_\perp^\varepsilon=-\varepsilon \partial_\myparallel B_\myparallel^\varepsilon$ in $\mathcal D'$, we obtain $\|\nabla_\perp \cdot B_\perp^\varepsilon \|_{L_{\rm loc}^2(\R_+; L^2(\T^3))} \leq \varepsilon \|\partial_\myparallel B_\myparallel^\varepsilon \|_{L_{\rm loc}^2(\R_+; L^2(\T^3))} \leq C_0\varepsilon$. On the one hand $ \nabla_\perp \cdot B_\perp^\varepsilon $ is bounded in $ L_{\rm loc}^2(\R_+; L^2(\T^3))$ and goes to $ \nabla_\perp \cdot B_\perp \in \mathcal D'$. On the other hand, it must vanish as $\varepsilon \rightarrow 0$. Hence, the second line.\\
For the third assertion, we apply  Lemma~\ref{lem:ALL} of Appendix~\ref{s:tools} with  $\mathfrak B_0 = H^1(\T^3)$,  $\mathfrak B = L^2(\T^3)$, $\mathfrak B_1 = H^{-1}(\T^3)$,  $p=r$, and $q=\infty$. To this end, we have to check the corresponding hypotheses. Below (and after when there is no possible ambiguity), {\it bounded} means ``uniformly bounded with respect to $\varepsilon$''.
\begin{itemize}
 \item From energy inequality \eqref{nrj-ineq-mhdeps}-\eqref{defPI-P} and the continuous Sobolev embedding $H^1(\T^3) \hookrightarrow L^6(\T^3)$, we obtain that $B^\varepsilon$ is bounded in $L_{\rm loc}^\infty(\R_+; L^2(\T^3)) \cap L_{\rm loc}^2(\R_+;   H^1\cap L^6 (\T^3))$.
  \item Obviously,  $B_\perp^\varepsilon$ is bounded in $L_{\rm loc}^\infty(\R_+; L^2(\T^3))  \cap L_{\rm loc}^1(\R_+; H^1(\T^3)) $.
  \item The final step is to estimate $\partial_t B_\perp^\varepsilon$. We can exploit equation \eqref{mhdorigin} to express this time derivative. Observe that, as can be seen at the level of \eqref{eq:Bperp-eps}, there is no singular term in $ \varepsilon $. The Laplacian parts are clearly bounded in $ L_{\rm loc}^1(\R_+; H^{-1} (\T^3)) $. Let us consider the products of components of $ B^\varepsilon $ and $ v^\varepsilon $. We refer to  (the proof of) Lemma \ref{lem:cv-vperp} which guarantees that $ v^\varepsilon $ is bounded in $L_{\rm loc}^2(\R_+; H^1\cap L^6 (\T^3))$. Hence,
by H\"older inequality, these products are bounded in $ L_{\rm loc}^1(\R_+; L^3 (\T^3)) $. Since $L^3(\T^3) \hookrightarrow L^2(\T^3) $, after spatial derivation, we find as required a bound in $ L_{\rm loc}^1(\R_+; H^{-1} (\T^3)) $ for  $\partial_t B_\perp^\varepsilon$.
\end{itemize}
\end{proof}  

%%%%%%%%%%%%%%%%%%%%%%%%%%

\subsection{Compactness of $v^\varepsilon $ and $\rho^\varepsilon v^\varepsilon $}
Here we aim at proving the following lemma.
\begin{lemma}
  \label{lem:cv-vperp} Assume $\gamma > 3/2$.
  Let $\mathfrak s := \max\{1/2, 3/\gamma-1\} \in [1/2,1)$. The sequences $v^\varepsilon$ and $\rho^\varepsilon v^\varepsilon $ satisfy the following properties.
    \begin{align*}
      & v^\varepsilon \xrightharpoonup{\quad} v \ \ \mbox{ \rm in } \  L_{\rm loc}^2(\R_+; L^6\cap H^1(\T^3))\!-\!\mbox{\rm weak}\,,\\
      & \nabla_\varepsilon \cdot v^\varepsilon   \xrightharpoonup{\quad}\nabla_\perp \cdot v_\perp =0 \ \ \mbox{ \rm in } \  L_{\rm loc}^2(\R_+; L^2(\T^3))\!-\!\mbox{\rm weak}\,,\\
      & \P_\perp v_\perp^\varepsilon \xrightharpoonup{\quad} \P_\perp v_\perp \ \mbox{ \textrm{and} } \ \Q_\perp v_\perp^\varepsilon \xrightharpoonup{\quad} \Q_\perp v_\perp = 0 \ \ \mbox{ \rm in } \  L_{\rm loc}^2(\R_+; L^6\cap H^{1}(\T^3))\!-\!\mbox{\rm weak}\,, \\
      & \P_\perp v_\perp^\varepsilon \xrightarrow{\quad} \P_\perp v_\perp=  v_\perp \ \ \mbox{ \rm in } \  L_{\rm loc}^2(\R_+; L^p\cap H^{s}(\T^3))\!-\!\mbox{\rm strong}\,, \quad 1\leq p <6\,, \quad 0\leq s <1\,, \\
      & \rho^\varepsilon v^\varepsilon  \xrightharpoonup {\quad} v
      \ \ \mbox{ \rm in } \  L_{\rm loc}^2(\R_+; L^q\cap H^{-\upsigma}(\T^3))\!-\!\mbox{\rm weak}\,, \quad \forall\, \upsigma \geq \mathfrak s := \max\Bigl\{\frac{1}{2},\frac{3}{\gamma-1} \Bigr\}\,, \quad  q=\frac{6\gamma}{6+\gamma} \,,\\
      & \rho^\varepsilon v^\varepsilon-v^\varepsilon  \xrightarrow{\quad} 0
      \ \ \mbox{ \rm in } \  L_{\rm loc}^2(\R_+; L^q(\T^3))\!-\!\mbox{\rm strong}\,, \quad   q={6\gamma}/({6+\gamma})\,, \\
      & \P_\perp (\rho^\varepsilon v_\perp^\varepsilon)  \xrightarrow{\quad} \P_\perp v_\perp=  v_\perp
      \ \ \mbox{ \rm in } \  L_{\rm loc}^2(\R_+; L^q(\T^3))\!-\!\mbox{\rm strong}\,, \quad   q={6\gamma}/({6+\gamma})\,.
    \end{align*}
\end{lemma}
\begin{proof}
  We start with the first statement of Lemma~\ref{lem:cv-vperp}. From energy inequality \eqref{nrj-ineq-mhdeps}-\eqref{defPI-P} with the pressure term $\Pi_2$,  we obtain that  $v^\varepsilon$ is  bounded in $L_{\rm loc}^2(\R_+; \dot H^1 (\T^3))$. Let us show that $v^\varepsilon$ is  bounded in $L_{\rm loc}^2(\R_+; H^1 (\T^3))$. From  Poincar\'e--Wirtinger inequality and H\"older inequality, it is easy to show that
 \[
 \|\cdot \|_{H^1(\T^3)}^2 \sim | \overline{(\cdot)}  |^2 + \| \cdot \|_{\dot H^1(\T^3)}^2 \, , \quad \mbox{ with } \quad \overline{(\cdot)} \equiv |\T^3|^{-1}\int_{\T^3} dx\,  \cdot\, .
 \]
 There remains to control the mean value $ \overline{v^\varepsilon} $. Using H\"older inequality, the embedding $L^6(\T^3) \hookrightarrow L^{2\gamma/(\gamma-1)}(\T^3)$ for $\gamma>3/2$, the continuous embedding $ H_0^1(\T^3) \hookrightarrow L^6(\T^3)$ (with $H_0^1(\T^3)$ the set of zero-average functions in $\dot H^1(\T^3)$), we obtain for any $T>0$,
 \begin{equation}\label{fromestimatee}
   \begin{array}{rl} \displaystyle \int_0^Tdt \int_{\T^3} dx \, \rho^\varepsilon \big|v^\varepsilon - \overline{v^\varepsilon} \big|^2 \! \! \!
     & \displaystyle \leq    \int_0^Tdt \,\| \rho^\varepsilon\|_{L^\gamma(\T^3)}  \big \|v^\varepsilon - \overline{v^\varepsilon}  \big\|_{L^{2\gamma/(\gamma-1)}(\T^3)}^2
     \\
     & \displaystyle \leq       \int_0^Tdt \, \| \rho^\varepsilon\|_{L^\gamma(\T^3)}  \big \|v^\varepsilon - \overline{v^\varepsilon}  \big\|_{L^6(\T^3)}^2 \\
     % &\leq \int_0^Tdt \,\| \rho^\varepsilon\|_{L^\gamma(\T^3)} \| \nabla v^\varepsilon \|_{L^2(\T^3)}^2\\
     & \displaystyle \leq  \| \rho^\varepsilon\|_{L_{\rm loc}^\infty(\R_+;L^\gamma(\T^3))} \| \nabla v^\varepsilon \|_{L_{\rm loc}^2(\R_+;L^2(\T^3))}^2 \leq  { C} < \infty\,.
   \end{array}
 \end{equation}
 Using inequality $b^2/2 \leq (a-b)^2+ a^2$ and the mass conservation law, we obtain
 \[
 \int_{\T^3} dx \, \rho^\varepsilon (\overline{ v^\varepsilon})^2 = \| \rho_0^\varepsilon\|_{L^1(\T^3)} (\overline{ v^\varepsilon})^2 \leq 2 \int_{\T^3} dx \, \rho^\varepsilon (v^\varepsilon - \overline{ v^\varepsilon})^2 + 2 \int_{\T^3} dx \, \rho^\varepsilon (v^\varepsilon)^2 .
 \]
 With \eqref{fromestimatee}, this implies (uniformly in $ \varepsilon $)
 \begin{equation*}
   (\overline{ v^\varepsilon})^2 \leq 2 \, T^{-1} \, \| \rho_0^\varepsilon\|_{L^1(\T^3))}^{-1}
   \big( { C}  + T \, \|\sqrt{ \rho^\varepsilon}v^\varepsilon \|_{L^\infty([0,T];L^2(\T^3))}^2\big) \leq \widetilde C \, .
 \end{equation*}
 This information combined with the bound of  $v^\varepsilon$ in $L_{\rm loc}^2(\R_+; \dot H^1 (\T^3))$ indicates that $v^\varepsilon$ is  bounded in $L_{\rm loc}^2(\R_+; H^1 (\T^3))$.\\
 We continue with the second statement. To this end, we look at the mass conservation law \eqref{eq:rho-eps}. Since $\rho^\varepsilon \rightarrow 1$ in $L_{\rm loc}^\infty(\R_+; L^\gamma(\T^3))$--strong (Lemma~\ref{lem:cv-rho}),   $v^\varepsilon \rightharpoonup v $ in $L_{\rm loc}^2(\R_+; L^6(\T^3))$--weak and $\rho_0^\varepsilon \rightarrow 1$ in $L^\gamma(\T^3)$--strong (from \eqref{ineqC0-P} and by following the proof of Lemma~\ref{lem:cv-rho}), it is easy to pass to the limit in the distributional sense in the linear terms. For the nonlinear term, we writte $ \rho^\varepsilon \, v^\varepsilon =  (\rho^\varepsilon -1) \, v^\varepsilon +  v^\varepsilon $. Since $1/\gamma + 1/6 < 1$ (recall that $\gamma>3/2$), the first term vanishes strongly in $L^2(\R_+;L^{6\gamma/(6+\gamma)}(\T^3))$. At the limit, we recover for any test function $ \varphi $ that
    \[
    \int_\Omega dx\, \varphi(0) +
    \int_0^\infty dt \int_\Omega dx \,\big(\partial_t  \varphi+ v_\perp \cdot \nabla_\perp \varphi\big)
    = \int_0^\infty dt \int_\Omega dx \, v_\perp \cdot \nabla_\perp \varphi = 0 \,,
    \]
    which means that $\nabla_\perp \cdot v_\perp =0\,$ in $\,\mathcal D'(\R_+^*\times \T^3)$.\\
    The third assertion of Lemma~\ref{lem:cv-vperp} is a  consequence of the first and second statements of Lemma~\ref{lem:cv-vperp}, of the Helmhotz--Hodge decomposition $v_\perp^\varepsilon = \P_\perp v_\perp^\varepsilon + \Q_\perp v_\perp^\varepsilon$ and of the (weak) continuity properties of $ \P_\perp $ and $ \Q_\perp $.\\
    The fourth assertion exploits some Gagliardo--Nirenberg interpolation inequalities together with delicate equicontinuity properties in time that require to already control the product $\rho^\varepsilon \, v^\varepsilon $. The proof is postponed  to a later stage.\\
    We pursue with the proof of fifth statement of Lemma~\ref{lem:cv-vperp}. On the one hand, from the uniform bounds  $v^\varepsilon \in L_{\rm loc}^2(\R_+; L^6(\T^3))$ and $\rho^\varepsilon \in L_{\rm loc}^\infty(\R_+; L^\gamma(\T^3))$, and on the other hand, from H\"older inequality, we obtain $\rho^\varepsilon v^\varepsilon \in L_{\rm loc}^2(\R_+; L^{6\gamma/(6+\gamma)}(\T^3))$ uniformly with respect to $\varepsilon$, which gives,  by weak compactness, the weak convergence of this sequence in this space. Moreover, from the first assertion of Lemma~\ref{lem:cv-rho}~and~\ref{lem:cv-vperp}, the  product of $\rho^\varepsilon$ and  $v^\varepsilon$ weakly converges to the limit point $v$ in $L_{\rm loc}^2(\R_+; L^q(\T^3))$--weak with $1/q= 1/\gamma + 1/6$. The Sobolev embedding $ H^{\mathfrak s}(\T^3) \hookrightarrow L^{q'}(\T^3)$, with $1/q'=1-1/q=(5\gamma -6)/(6\gamma)$ and $\mathfrak s \geq \max\{0,3/\gamma-1\}$, implies by duality that $L^q (\T^3) \hookrightarrow H^{-\mathfrak s}(\T^3)$. Without loss of generality and in order to avoid further the multiplication of regularity indices we restrict $\mathfrak s$ such that  $\mathfrak s \geq \max\{1/2,3/\gamma-1\}$. We then have $\rho^\varepsilon v^\varepsilon \in L_{\rm loc}^{2}(\R_+; H^{- \mathfrak s}(\T^3))$. We conclude by using the Sobolev embedding $H^{\upsigma}(\T^3) \hookrightarrow H^{\mathfrak s}(\T^3) $ for $\upsigma \geq \mathfrak s$, and duality.\\    
From H\"older inequality, we have
\[
\|\rho^\varepsilon v^\varepsilon -v^\varepsilon \|_{L_{\rm loc}^2(\R_+;L^q(\T^3))} \leq
\|\rho^\varepsilon -1 \|_{L_{\rm loc}^\infty(\R_+;L^\gamma(\T^3))} \|v^\varepsilon\|_{L_{\rm loc}^2(\R_+;L^6(\T^3))}
\, , \qquad 1/q=1/\gamma +1/6 \, .
\]
Then, exploiting the first assertion of Lemmas~\ref{lem:cv-rho}~and~\ref{lem:cv-vperp}, we obtain the sixth statement of Lemma~\ref{lem:cv-vperp}.\\
We can now come back to  the $L_{\rm loc}^2 (\R_+ ;L^p (\T^3)) $--strong convergence in the fourth assertion of Lemma~\ref{lem:cv-vperp}. For this, we use Lemma~\ref{lem:compactness} of Appendix~\ref{s:tools}, for which we show below that its hypotheses, splitted in three points, are satisfied. 1)  From the fifth statement of Lemma~\ref{lem:cv-vperp}, the selfadjointness of $\P_\perp$ for the scalar product of $L^2(\Omega; \R^2)$, and the continuity of $\P_\perp$ in $L^\alpha(\Omega; \R^2)$, for $1<\alpha < \infty$,  we obtain $ \P_\perp(\rho^\varepsilon v_\perp^\varepsilon) \rightharpoonup \P_\perp v_\perp =v_\perp$ in $L_{\rm loc}^2(\R_+; L^{6\gamma/(6+\gamma)}(\T^3))$--weak. 2) The bound $\P_\perp v_\perp^\varepsilon \in  L_{\rm loc}^2(\R_+; L^6(\T^3))$ and Lemma~4.3 in \cite{Bre11} implies $\|\P_\perp v_\perp^\varepsilon (t,\cdot +h )-\P_\perp v_\perp^\varepsilon (t,\cdot )\|_{L_{\rm loc}^2(\R_+; L^6(\T^3))} \rightarrow 0$, as $|h|\rightarrow 0$, uniformly with respect to $\varepsilon$. 3) Applying the Leray projector $\P_\perp$ to equation \eqref{eq:vperp-eps} for $\rho^\varepsilon v_\perp^\varepsilon$, we obtain
    \begin{equation}
      \label{eqn:ptPpvp}
      \partial_t(\P_\perp(\rho^\varepsilon v_\perp^\varepsilon))=-\nabla_\varepsilon \cdot \P_\perp(\rho^\varepsilon v_\perp^\varepsilon \otimes  v^\varepsilon )
      + \partial_\myparallel \P_\perp B^\varepsilon - \nabla_\varepsilon \cdot \P_\perp(B_\perp^\varepsilon \otimes  B^\varepsilon )
      + \mu_\perp^\varepsilon \Delta_\perp \P_\perp v_\perp^\varepsilon + \mu_\myparallel^\varepsilon \Delta_\myparallel \P_\perp v_\perp^\varepsilon\,.
    \end{equation}
    Using the bounds $\rho^\varepsilon |v^\varepsilon|^2 \in L_{\rm loc}^\infty(\R_+; L^1(\T^3))$,  $v^\varepsilon\in L_{\rm loc}^2(\R_+; H^1(\T^3))$, and $B^\varepsilon \in L_{\rm loc}^\infty(\R_+ ;L^2(\T^3))$, and the following properties of the projector $\P_\perp$, $\P_\perp(H^\alpha(\T^3)) \hookrightarrow H^\alpha(\T^3)$, with $\alpha \geq 0$, and  $\P_\perp(L^1(\T^3))  \hookrightarrow W^{-\delta,1}(\T^3)$, with $\delta >0$, we obtain from \eqref{eqn:ptPpvp}, $ \partial_t(\P_\perp(\rho^\varepsilon v_\perp^\varepsilon)) \in L_{\rm loc}^2(\R_+; (W^{-\delta -1,1}+ H^{-1}+L^2)(\T^3)) \hookrightarrow  L_{\rm loc}^1(\R_+; W^{-\delta -1,1}(\T^3))$. Gathering points 1) to 3), we can apply  Lemma~\ref{lem:compactness} of Appendix~\ref{s:tools} with $g^\varepsilon =\rho^\varepsilon v_\perp^\varepsilon$,  $h^\varepsilon =\P_\perp v_\perp^\varepsilon$, $p_1=q_1=2$ ($1/p_1+1/q_1=1$), $p_2=6\gamma/(6+\gamma)$, and $q_2=6$ ($1/p_2+1/q_2=1/\gamma+1/3 < 1$, for $\gamma>3/2$), to deduce that
    \begin{equation}
      \label{eqn:scv-Ppvp}
      \P_\perp(\rho^\varepsilon v_\perp^\varepsilon) \cdot  \P_\perp v_\perp^\varepsilon \xrightharpoonup{\quad}
      |\P_\perp v_\perp|^2= | v_\perp|^2 \ \  \mbox{ \rm in } \  \ \mathcal D'(\R_+^*\times \T^3)\,.
    \end{equation}
    This limits leads to the strong convergence of $\P_\perp v_\perp^\varepsilon$ to $v_\perp$ in $L_{\rm loc}^2(\R_+; L^2(\T^3))$. Indeed, using \eqref{eqn:scv-Ppvp} and H\"older inequality, we obtain for any $T>0$,
    \begin{align*}
      \limsup_{\varepsilon \rightarrow 0} \| \P_\perp v_\perp^\varepsilon\|_{L_{\rm loc}^2(\R_+; L^2(\T^3))}^2 -
      \| \P_\perp v_\perp\|_{L_{\rm loc}^2(\R_+; L^2(\T^3))}^2 & \\
      & \hspace{-8.5cm}=
      \limsup_{\varepsilon \rightarrow 0} \int_0^Tdt \int_{\T^3} dx\, \big(| \P_\perp v_\perp^\varepsilon|^2-   \P_\perp(\rho^\varepsilon v_\perp^\varepsilon) \cdot \P_\perp v_\perp^\varepsilon \big)
      = \limsup_{\varepsilon \rightarrow 0} \int_0^Tdt \int_{\T^3} dx\,  \P_\perp v_\perp^\varepsilon \cdot  v_\perp^\varepsilon (1-\rho^\varepsilon) \\
      &\hspace{-8.55cm}\leq \limsup_{\varepsilon \rightarrow 0} \|\rho^\varepsilon-1\|_{L_{\rm loc}^\infty(\R_+; L^\gamma(\T^3))} \|v_\perp^\varepsilon\|_{L_{\rm loc}^2(\R_+; L^{2\theta}(\T^3))}^2\,,
    \end{align*}  
    with $1/\gamma +1/\theta=1$. Since $\gamma >3/2$, we obtain $2\theta <6$. Therefore, using the first statement of Lemmas~\ref{lem:cv-rho}~and~\ref{lem:cv-vperp}, the right-hand side of the previous inequality vanishes, which leads to $\limsup_{\varepsilon \rightarrow 0} \| \P_\perp v_\perp^\varepsilon\|_{L_{\rm loc}^2(\R_+; L^2(\T^3))}  \leq \| \P_\perp v_\perp\|_{L_{\rm loc}^2(\R_+; L^2(\T^3))} \geq \liminf_{\varepsilon \rightarrow 0} \| \P_\perp v_\perp^\varepsilon\|_{L_{\rm loc}^2(\R_+; L^2(\T^3))}$ and to $\P_\perp v_\perp^\varepsilon \rightarrow \P_\perp v_\perp=  v_\perp$ in  $L_{\rm loc}^2(\R_+; L^2(\T^3))$--strong, by Proposition~3.32 in \cite{Bre11}. This strong convergence in $L^2$ allows to prove the fourth statement of Lemma~\ref{lem:cv-vperp} by using some interpolation inequalities. Indeed, using Gagliardo--Nirenberg inequality in space  and Cauchy-Schwarz inequality in time, we obtain
    \begin{equation}
      \label{eqn:GNCSI-1}
      \|  \P_\perp v_\perp^\varepsilon -v_\perp\|_{L_{\rm loc}^2(\R_+; L^{p_1}(\T^3))}
      \lesssim  \|  \P_\perp v_\perp^\varepsilon -v_\perp\|_{L_{\rm loc}^2(\R_+; \dot H^1(\T^3))}^{1/2}
      \|  \P_\perp v_\perp^\varepsilon -v_\perp\|_{L_{\rm loc}^2(\R_+; L^{p_0}(\T^3))}^{1/2} \,, 
    \end{equation}
    with $1/p_1=1/12+1/(2p_0)$ and $\|  \P_\perp v_\perp^\varepsilon -v_\perp\|_{L_{\rm loc}^2(\R_+; \dot H^1(\T^3))} <\infty$. Iterating inequality \eqref{eqn:GNCSI-1}, we obtain a sequence of indices $p_n$ such that $1/p_{n+1}=1/12+ 1/(2p_n)$, with $p_0=2$, hence its limit $p_\infty=6$. This justifies the first part (strong convergence in $L^p$) of the fourth statement of Lemma~\ref{lem:cv-vperp}. For the second part of this statement, again, using Gagliardo--Nirenberg inequality in space   (with not necessarily integers \cite{BM18, BM19}) and  Cauchy-Schwarz inequality in time, we obtain
    \begin{equation}
      \label{eqn:GNCSI-2}
      \|  \P_\perp v_\perp^\varepsilon -v_\perp\|_{L_{\rm loc}^2(\R_+;  H^{s_1}(\T^3))}
      \lesssim  \|  \P_\perp v_\perp^\varepsilon -v_\perp\|_{L_{\rm loc}^2(\R_+;  H^1(\T^3))}^{1/2}
      \|  \P_\perp v_\perp^\varepsilon -v_\perp\|_{L_{\rm loc}^2(\R_+;  H^{s_0}(\T^3))}^{1/2}\,,
    \end{equation}
    with $s_1=1/2+ s_0/2$ and $\|  \P_\perp v_\perp^\varepsilon -v_\perp\|_{L_{\rm loc}^2(\R_+; H^1(\T^3))} <\infty$. Iterating inequality \eqref{eqn:GNCSI-2}, we obtain a sequence of indices $s_n$ such that $s_{n+1}=1/2+ s_n/2$, with $s_0=0$, hence its limit $s_\infty=1$. This justifies the second part (strong convergence in $H^s$) of the fourth statement of Lemma~\ref{lem:cv-vperp}.\\
    We finish with the proof the seventh statement. By triangular inequality and continuity of $\P_\perp$ in $L^\alpha(\Omega; \R^2)$, for $1<\alpha < \infty$, we obtain
    \[
    \|  \P_\perp (\rho^\varepsilon v_\perp^\varepsilon -v_\perp)\|_{L_{\rm loc}^2(\R_+; L^q(\T^3))} \leq
    \|  \P_\perp v_\perp^\varepsilon -v_\perp \|_{L_{\rm loc}^2(\R_+;L^q(\T^3))}+ \| \rho^\varepsilon v_\perp^\varepsilon -v_\perp^\varepsilon\|_{L_{\rm loc}^2(\R_+; L^q(\T^3))} .
    \]
    Using the fourth and  sixth assertions of Lemma~\ref{lem:cv-vperp}, and since $6/5<q<6$ (because $1/q=1/\gamma+1/6$, and $3/2<\gamma<\infty$) this inequality allows to conclude.
\end{proof}

%%%%%%%%%%%%%%%%%%%%%%%%%%%%%

\subsection{Passage to the limit in the equation for $B_\perp^\varepsilon $}
\label{ss:cv-Bperp-P}
Here, we justify the passage to the limit in \eqref{eq:Bperp-eps} for $B_\perp^\varepsilon$. Let us start with the initial condition term.
From the discussion about the properties of sequences of initial conditions in Section~\ref{ss:rT3} (in particular the uniform bound \eqref{ineqC0-P} and the resulting convergences),
we can pass to the limit, in the distributional sense, in the initial condition term of \eqref{eq:Bperp-eps} to obtain the limit initial condition $B_{0 \perp}$.
Next, using on the one hand the third statement of Lemma~\ref{lem:cv-Bperp}, and on the other hand the first and the second statements of Lemma~\ref{lem:cv-vperp}, we can pass to the limit, in the distributional sense, in all linear and nonlinear terms of equations  \eqref{eq:Bperp-eps} and \eqref{eq:zero-divB} to obtain the first equation of \eqref{rmhdincomp}
and equation \eqref{divperppourB} in the sense of distributions.

%%%%%%%%%%%%%%%%%%%%%%%%%%%%

\subsection{Passage to the limit in the equation for $\rho^\varepsilon v_\perp^\varepsilon $}
\label{ss:cv-rvperp-P}
Here, we justify the passage to the limit in equation \eqref{eq:vperp-eps} for $\rho^\varepsilon v_\perp^\varepsilon$, in several steps.
We start by recalling some basic tools. With $\Omega$ being either $\T^3$ or $\R^3$, we introduce the linear differential operator
\begin{equation}
  \label{eqn:LDOL}
  \mathcal L := -\left (\begin{array}{cc}
0 &   c^\varepsilon \, \nabla_\perp \cdot \\
{}^t \nabla_\perp & 0
  \end{array}
  \right),
\end{equation}
defined on $\mathcal D_0^\prime(\Omega;\R) \times \mathcal D^\prime(\Omega;\R^2) $, where $ \mathcal D_0^\prime(\Omega;\R)= \{\phi \in \mathcal D^\prime(\Omega;\R) \ | \ \int_{\Omega}dx\, \phi=0 \}$, and such that $\mathcal{L} U = -\,{}^t(c^\varepsilon \nabla_\perp \cdot  \Phi,\, {}^t\nabla_\perp \phi)$, with $ U\equiv {}^t(\phi,{}^t\Phi) \in \mathcal D_0^\prime(\Omega;\R) \times \mathcal D^\prime(\Omega;\R^2) $. Here $c^\varepsilon:=b^\varepsilon + 1/\overline{\rho^\varepsilon}$, with $b^\varepsilon:=b (\overline{\rho^\varepsilon})^{\gamma-1}$. Since for $\varepsilon$ small enough $\overline{\rho^\varepsilon} \in (1/2,3/2)$, there exist constants $0<\underline {c} \leq \bar {c} <\infty$, such that $\underline {c} \leq c^\varepsilon \leq \bar {c}$, and $c^\varepsilon \rightarrow c = 1+ b = 1+ a \gamma$.

We claim that $\mathcal L$ generates a one-parameter group of isometry $\{ \mathcal{S}(\tau):=\exp(\tau \mathcal L);\, \tau \in \R \}$, from $H^\alpha(\Omega;\R) \times H^\alpha(\Omega;\R^2)$ into itself with the norm $\vertiii{ U }_{H^\alpha(\Omega)}^2 := \|\phi \|_{H^\alpha(\Omega)}^2 +c^\varepsilon\| \Phi \|_{H^\alpha(\Omega)}^2$, for all $\alpha\in \R$. Indeed, this comes from the fact that the operator $\mathcal L$ is skew-adjoint for the scalar product $ \dsp{2.5}{ \cdot\, , \, \cdot  }$ of $L^2(\Omega;\R) \times L^2(\Omega;\R^2)$, defined by $ \dsp{2.5}{ U, \, V  } =\langle \phi, \psi\rangle_{L^2(\Omega)} + c^\varepsilon \langle \Phi,\Psi \rangle_{L^2(\Omega)}$, where $U\equiv {}^t(\phi,{}^t\Phi)$, $V\equiv {}^t(\psi,{}^t\Psi)$, and the notation $\langle \cdot, \cdot\rangle_{L^2(\Omega)}$ stands for the standard scalar product in $L^2(\Omega)$ for scalar or vector valued functions. This isometry group can also be directly verified from the $H^\alpha$-energy estimates of the solutions $U(\tau) = \mathcal{S}(\tau)\, U_0$, satisfying the equation $\partial_\tau \,U(\tau)=   \mathcal L \, U(\tau) $, i.e.,
\begin{equation}
\label{eqn:first-order-linear-wave-eq}
\partial_\tau \phi + c^\varepsilon \,\nabla_\perp\cdot \Phi = 0\,, \quad
\partial_\tau \Phi + \nabla_\perp \phi = 0\,.
\end{equation}
Indeed, we first set $\Lambda = ({\rm I}-\Delta)^{1/2}$ and recall that $H^\alpha (\Omega) = \Lambda^{-\alpha} L^2(\Omega)$. Applying $\Lambda^{\alpha}$ to equations \eqref{eqn:first-order-linear-wave-eq}, and taking the $L^2(\Omega)$ scalar product of the result with $\Lambda^{\alpha}\phi$ (resp.  $c^\varepsilon \Lambda^{\alpha}\Phi$) for the first (resp. second) equation of \eqref{eqn:first-order-linear-wave-eq}, we obtain
\begin{align*}
  \frac{d}{dt} \vertiii{ U }_{H^\alpha(\Omega)}^2 & =  \frac{d}{dt}  \dsp{2.5}{\Lambda^{\alpha} U , \Lambda^{\alpha} U}= \frac{d}{dt}\big(\|\Lambda^{\alpha}\phi\|_{L^2(\Omega)}^2
  + c^\varepsilon\|\Lambda^{\alpha}\Phi\|_{L^2(\Omega)}^2 \big) \\
  &= \int_\Omega dx\, c^\varepsilon 
  \big(
  \Lambda^{\alpha}\phi \nabla_\perp\cdot \Lambda^{\alpha}\Phi + \Lambda^{\alpha}\Phi\cdot  \nabla_\perp\Lambda^{\alpha}\phi 
  \big)=0\,.
\end{align*}
In order to understand some properties of the group $\mathcal S(\tau)$, we denote by $\mathcal S_1(\tau) \in \R $ and $\mathcal S_2(\tau) \in\R^2$ the components of  $\mathcal S(\tau)$.  Observe that $\int_{\Omega} dx\,\mathcal  S_1(\tau)\, U$ and $\P_\perp \mathcal S_2(\tau) \,U$ are independent of $\tau\in\R$. In particular $\mathcal S(\tau) \,U$ is independent of $\tau$, if $\P_\perp \Phi =0$ (i.e., $\nabla_\perp\cdot \Phi =0$) and if $\phi$ is constant. From \eqref{eqn:first-order-linear-wave-eq} the operator $\mathcal L$ is equivalent to the transverse wave operator. Indeed, the equation $\partial_\tau \, U = \mathcal{L}\, U$ is equivalent to the scalar wave equations $(\partial_\tau^2 -c^\varepsilon \Delta_\perp)\phi=0$, and $(\partial_\tau^2 -c^\varepsilon \Delta_\perp)\varphi=0$, where we have used the Helmholtz--Hodge decomposition $\Phi= \P_\perp \Phi + \nabla_\perp \varphi$, with $\int_{\Omega} dx\, \varphi =0$, $\forall \, \tau\in\R$, and observing that $\P_\perp \Phi$ is a constant determined by the initial conditions (since $\partial_\tau \P_\perp \Phi=0$ from the second equation of \eqref{eqn:first-order-linear-wave-eq}).\newline

The key to justify the limit of the equation for $\rho^\varepsilon v_\perp^\varepsilon $,
is to construct an approximate solution to the MHD equations \eqref{eq:rho-eps}-\eqref{eq:zero-divB}, which allows
us to pass to the limit in both singular terms and nonlinear terms. Such a construction is given by the following lemma.

\begin{lemma}
\label{lem:UG}
Let us define  $U^\varepsilon:={}^t(\phi^\varepsilon, {}^t\Phi^\varepsilon)$, where $\phi^\varepsilon := b^\varepsilon \varrho^\varepsilon + B_\myparallel^\varepsilon,\,$ and $\ \Phi^\varepsilon:=\Q_\perp(\rho^\varepsilon v_\perp^\varepsilon),\, $ with $\ \varrho^\varepsilon:= (\rho^\varepsilon-\overline{\rho^\varepsilon})/\varepsilon,\,$  and $\ b^\varepsilon:= b (\overline{\rho^\varepsilon})^{\gamma-1}$. Let $\mathfrak s$ be the same positive real number as in Lemma~\ref {lem:cv-vperp}, i.e., $\mathfrak s := \max\{1/2, 3/\gamma-1\} \in [1/2,1)$. Then, 
  \begin{itemize}
  \item[1.] There exist functions $\, \mathcal U={}^t(\psi, {}^t\Psi) \in L_{\rm loc}^2(\R_+;L^2(\T^3;\R^3))\,,$  and
    $\ \mathcal R^\varepsilon \in L_{\rm loc}^2(\R_+;H^{-\sigma}(\T^3,\R^3))$, with  $\ \mathfrak s <\sigma< (5/2)^+ $, such that
    \begin{equation}
      \label{eqn:GPR}
      U^\varepsilon= \mathcal S (t/\varepsilon)\, \mathcal U
      \, + \, \mathcal R^\varepsilon,
      \ \mbox{ with} \ \ \mathcal R^\varepsilon  \xrightarrow{\quad} 0 \ \
      \mbox{ \rm in } \  L_{\rm loc}^2(\R_+; H^{-\sigma}(\T^3))\!-\!\mbox{\rm strong}\,.
    \end{equation}
  \item[2.] There exists a function $\uppi_0 := \Delta_\perp^{-1} \nabla_\perp\cdot\, \Q_\perp u_{0 \perp} \in L^2(\T_\myparallel; H^1(\T_\perp^2))$, as well as a function $\uppi_1 \in  L_{\rm loc}^2(\R_+;H^{-r}(\T^3;\R))$, with $r> \big(\tfrac52 \big)^+$, such that
    \begin{equation}
      \label{eqn:existpi1}
            {\phi^\varepsilon}/{\varepsilon} \xrightharpoonup{\quad}   \delta_{0}(t) \otimes \uppi_0 \, + \,  \uppi_1 
            \ \ \mbox{ \rm in } \  H^{-1}(\R_+; H^{-r}(\T^3))\!-\!\mbox{\rm weak}\,.
    \end{equation}
  \item[3.] The limit point $(B_\myparallel, \varrho)$ satisfies the relation  $B_\myparallel + b\varrho =0$, for a.e. $(t,x) \in ]0,+\infty[ \times \T^3$.
  \end{itemize} 
\end{lemma}

\begin{proof}
  On the one hand, applying the projector $\Q_\perp$ to the perpendicular component of the second equation of \eqref{mhdeps} (the one for $\rho^\varepsilon v_\perp^\varepsilon$) to form an equation for $\Phi^\varepsilon:=\Q_\perp(\rho^\varepsilon v_\perp^\varepsilon)$ and on the other hand combining the first equation of \eqref{mhdeps} (the one for $\rho^\varepsilon$) and the parallel component of the second equation of \eqref{mhdeps} (the one for $B_\myparallel^\varepsilon$) to form an equation for $\phi^\varepsilon := b^\varepsilon \varrho^\varepsilon + B_\myparallel^\varepsilon$, we obtain the following equation for $U^\varepsilon:={}^t(\phi^\varepsilon, {}^t\Phi^\varepsilon)$,
  \begin{equation}
    \label{eqn:Ueps}
      \partial_t U^\varepsilon - \frac{1}{\varepsilon} \mathcal L  U^\varepsilon = F^\varepsilon\,,
  \end{equation}
  where
  \begin{equation}
    \label{eqn:Feps}
    F^\varepsilon:=
    \left (\begin{aligned}
     F_1^\varepsilon\\ F_2^\varepsilon
  \end{aligned}
    \right) =
    \left\{
  \begin{aligned}
    F_1^\varepsilon =& -b^\varepsilon\partial_\myparallel (\rho^\varepsilon v_\myparallel^\varepsilon) + (\overline{\rho^\varepsilon})^{-1}
    \nabla_\perp \cdot \Q_\perp(\varrho^\varepsilon v_\perp^\varepsilon) -B_\myparallel^\varepsilon  \nabla_\varepsilon \cdot v^\varepsilon \\
      & - (v^\varepsilon\cdot\nabla_\varepsilon)B_\myparallel^\varepsilon 
    +(B^\varepsilon\cdot\nabla_\varepsilon)v_\myparallel^\varepsilon + (\eta_\perp^\varepsilon\Delta_\perp+\, \eta_\myparallel^\varepsilon\Delta_\myparallel)B_\myparallel^\varepsilon\,,\\
     F_2^\varepsilon =& -\Q_\perp\nabla_\varepsilon \cdot (\rho^\varepsilon v_\perp^\varepsilon \otimes v^\varepsilon)
    -(\gamma-1)\nabla_\perp\Pi_2(\rho^\varepsilon) -\tfrac12 \nabla_\perp(|B^\varepsilon|^2) + \partial_\myparallel \Q_\perp B_\perp^\varepsilon \\
    & + \Q_\perp\nabla_\varepsilon \cdot (B_\perp^\varepsilon \otimes B^\varepsilon) + \mu_\perp^\varepsilon \nabla_\perp(\nabla_\perp \cdot v_\perp^\varepsilon)
    + \mu_\myparallel^\varepsilon \Delta_\myparallel \Q_\perp v_\perp^\varepsilon  +\lambda^\varepsilon\nabla_\perp(\nabla_\varepsilon \cdot v^\varepsilon)\,.
  \end{aligned}
  \right.
  \end{equation}

  We start with point $1$ of Lemma~\ref{lem:UG}. We aim at showing \eqref{eqn:GPR} in two steps. The first step concerns the existence of a filtered profile $\, \mathcal U={}^t(\psi, {}^t\Psi) $, while the second one establishes its regularity in $L^2$.
  
  \underline{Step 1}. Here, we show that the filtered solution ${}^t(\psi^\varepsilon, {}^t\Psi^\varepsilon) \equiv \, \mathcal U^\varepsilon:=\mathcal S(-t/\varepsilon)U^\varepsilon $ is relatively compact in $L_{\rm loc}^{2}(\R_+; H^{-\sigma}(\T^3))$ for $\sigma> \mathfrak s$. For this, we first show uniform bounds for  $U^\varepsilon$ in suitable functional spaces. Second, we use the fact that the group  $\mathcal S$ is an isometry in $H^{\alpha}$ ($\alpha\in \R$) in order to obtain similar bounds on $(\, \mathcal U^\varepsilon, \, \partial_t \, \mathcal U^\varepsilon)$. Next, we invoke an Aubin--Lions theorem to obtain compactness of the sequence $\, \mathcal U^\varepsilon$ and the existence of the filtered profile $\, \mathcal U$.  Finally, using  the isometry $\mathcal S$ and the averaged profile $\, \mathcal U$, we construct an approximation to $U^\varepsilon$, with an error estimate which converges strongly as indicated in \eqref{eqn:GPR}. In this way, we bypass the singularity in $1/\varepsilon$ in \eqref{eqn:Ueps}. \\
  Recall that $\kappa =\min\{2,\gamma\}$. Since $\varrho^\varepsilon \in L_{\rm loc}^{\infty}(\R_+; L^{\kappa}(\T^3))$ and $B_\myparallel^\varepsilon \in  L_{\rm loc}^{\infty}(\R_+; L^2(\T^3))$, we can assert that $\phi^\varepsilon \in  L_{\rm loc}^{\infty}(\R_+; (L^{\kappa}+L^2)(\T^3))$. From Sobolev embeddings and a duality argument, we obtain $L^{3/2}(\T^3) \hookrightarrow H^{-\upalpha}(\T^3)$, with $\upalpha \geq 1/2$. Since $\kappa >3/2$, then $L^{\kappa}(\T^3) \hookrightarrow H^{-\upalpha}(\T^3)$ and  $(L^{\kappa}+L^2)(\T^3) \hookrightarrow H^{-\upalpha}(\T^3)$. Since $\upalpha$ can be chosen such that $\upalpha \leq \mathfrak s$, then  $(L^{\kappa}+L^2)(\T^3) \hookrightarrow H^{-\mathfrak s}(\T^3)$ and $\phi^\varepsilon \in  L_{\rm loc}^{\infty}(\R_+; H^{-\mathfrak s}(\T^3)) \hookrightarrow  L_{\rm loc}^{2}(\R_+; H^{-\mathfrak s}(\T^3))$. From the fifth statement of Lemma~\ref{lem:cv-vperp}, $\rho^\varepsilon v_\perp^\varepsilon \in L_{\rm loc}^{2}(\R_+; H^{-\mathfrak s}(\T^3))$. Since $\Q_\perp$ (resp. $\mathcal S$) is a continuous map (resp. an isometry) in $H^{\alpha}$, with $\alpha\in \R$, then $\, \mathcal U^\varepsilon \in L_{\rm loc}^{2}(\R_+; H^{-\mathfrak s}(\T^3))$. \\
  We are now going to obtain a bound for $\partial_t \, \mathcal U^\varepsilon $. Using \eqref{eqn:Ueps}, a straightforward computation shows that $\partial_t \, \mathcal U^\varepsilon =\mathcal S(-t/\varepsilon)F^\varepsilon $. Let us show that $F^\varepsilon \in  L_{\rm loc}^{2}(\R_+; H^{-r}(\T^3))$ for $r>5/2+\delta$, and any $\delta >0$. We will then  obtain $ \partial_t \, \mathcal U^\varepsilon \in  L_{\rm loc}^{2}(\R_+; H^{-r}(\T^3))$, because $\mathcal S$ is an unitary group in $H^{\alpha}$, with $\alpha\in \R$. \\
  Let us start with $F_1^\varepsilon$. First, observe that
  $B_\myparallel^\varepsilon  \nabla_\varepsilon \cdot v^\varepsilon + (v^\varepsilon\cdot\nabla_\varepsilon)B_\myparallel^\varepsilon - (B^\varepsilon\cdot\nabla_\varepsilon) v_\myparallel^\varepsilon=[\nabla_\varepsilon \times (B_\myparallel^\varepsilon \times v^\varepsilon)]_{\myparallel} \in L_{\rm loc}^{2}(\R_+; W^{-1,3/2}(\T^3))$ by using H\"older inequality and the energy estimate. Obviously, from the energy estimate, the last two diffusive terms of $F_1^\varepsilon$ belong to $L_{\rm loc}^{2}(\R_+; H^{-1}(\T^3))$. Using $U^\varepsilon \in L_{\rm loc}^{2}(\R_+; H^{-\mathfrak s}(\T^3))$, the first term of $F_1^\varepsilon$ is in $L_{\rm loc}^{2}(\R_+; H^{-\mathfrak s-1}(\T^3))$. It remains to bound the second term of $F_1^\varepsilon$. Using H\"older inequality we obtain $\varrho^\varepsilon v^\varepsilon \in L_{\rm loc}^2(\R_+; L^{\mathfrak q}(\T^3))$ with $1/\mathfrak q =1/\kappa + 1/6$. Observe that $\mathfrak q \in (6/5, 3/2]$ since $\kappa \in (3/2,2]$. The Sobolev embedding $ H^{ \tilde{\mathfrak{s}}}(\T^3) \hookrightarrow L^{\mathfrak q'}(\T^3)$, with $1/\mathfrak q'=1-1/\mathfrak q=(5\kappa -6)/(6\kappa)$ and $\tilde{\mathfrak s} \geq 3/\kappa-1\in [1/2,1)$, implies by duality that $L^{\mathfrak q} (\T^3) \hookrightarrow H^{-\tilde{\mathfrak s}}(\T^3)$. Since  $  3/\gamma-1\geq 3/\kappa-1 $, we can choose $\tilde{\mathfrak s} = \mathfrak s$. Therefore, the second term of $F_1^\varepsilon$ is bounded in $L_{\rm loc}^{2}(\R_+; H^{-\mathfrak s-1}(\T^3))$, and we obtain $F_1^\varepsilon \in  L_{\rm loc}^{2}(\R_+; (H^{-\mathfrak s-1}+ W^{-1,3/2} + H^{-1})(\T^3)) \hookrightarrow  L_{\rm loc}^{2}(\R_+; H^{-\mathfrak s-1}(\T^3))$, where the previous injection results from Sobolev embedding.\\
    We continue with an  estimate for  $F_2^\varepsilon$. From the continuous embedding $\Q_\perp (L^1(\T^3)) \hookrightarrow W^{-\delta,1}(\T^3)$ which works for all $ \delta >0$, and the energy estimate, the first and fifth terms of $F_2^\varepsilon$ are uniformly bounded in $L_{\rm loc}^{2}(\R_+; W^{-1-\delta,1}(\T^3))$. From the energy estimate, the second and third terms of $F_2^\varepsilon$ are uniformly bounded in $L_{\rm loc}^{2}(\R_+; W^{-1,1}(\T^3))$. From the following continuous embedding,  $\forall \alpha\geq 0$, $\Q_\perp (H^{\alpha}(\T^3)) \hookrightarrow H^{\alpha}(\T^3)$, and the energy estimate, the fourth term of $F_2^\varepsilon$ is uniformly bounded in $L_{\rm loc}^{2}(\R_+;L^2(\T^3))$. Obviously, from the energy estimate, the last three diffusive terms of $F_2^\varepsilon$ belong to $L_{\rm loc}^{2}(\R_+; H^{-1}(\T^3))$. Therefore, we obtain $F_2^\varepsilon \in  L_{\rm loc}^{2}(\R_+; (W^{-1-\delta,1}+W^{-1,1}+H^{-1} + L^2)(\T^3)) \hookrightarrow  L_{\rm loc}^{2}(\R_+; H^{-r}(\T^3)) $, with $r>5/2+\delta $, by using Sobolev embeddings. \\
    Now, using Lemma~\ref{lem:ALL} of Appendix~\ref{s:tools}, with $\mathfrak B_0 = H^{- \mathfrak s}(\T^3)$,  $\mathfrak B = H^{-\sigma}(\T^3)$,  $\mathfrak B_1 = H^{-r}(\T^3)$,  $ \mathfrak s <\sigma < r $, and $p=q=2$, we obtain that $\, \mathcal U^\varepsilon$ is compact in $ L_{\rm loc}^{2}(\R_+; H^{-\sigma}(\T^3))$. We deduce that there exists $\, \mathcal U\equiv {}^t(\psi, {}^t\Psi) \in L_{\rm loc}^{2}(\R_+; H^{-\sigma}(\T^3))$ such that $\, \mathcal U^\varepsilon$ converges strongly to  $\, \mathcal U$ in $ L_{\rm loc}^{2}(\R_+; H^{-\sigma}(\T^3))$. Since $\P_\perp \Psi^\varepsilon =0$ (resp. $\int_{\T^3} dx\, \psi^\varepsilon=0$), $\forall \varepsilon \geq 0$, then $\P_\perp \Psi =0$ (resp. $\int_{\T^3} dx\, \psi=0$). Since  the group $\mathcal S$ is an isometry in $H^{\alpha}$ ($\alpha\in \R$), we finally obtain \eqref{eqn:GPR}.

 \underline{Step 2}. To show that $ \mathcal U \in  L_{\rm loc}^{2}(\R_+;  L^2(\T^3;\R^3))$, we use the auxiliary variable $\widetilde U^\varepsilon:={}^t(\phi^\varepsilon, {}^t\Q_\perp v_\perp^\varepsilon)$. We first establish two points which allow us to deal with a truncated version of $\widetilde U^\varepsilon$ instead of $\widetilde U^\varepsilon$ itself.\\
1) From the first statement of Lemma~\ref{lem:cv-rho} and estimates \eqref{eqn:varrhoeps-est} (which imply, via the De la Vall\'ee Poussin criterion \cite{FL07}, that   $\varrho^\varepsilon$ is spatially uniformly integrable in $L^{3/2}(\T^3)$, uniformly in time on any compact time interval), we obtain $\ \|\varrho^\varepsilon- \varrho^\varepsilon \mathbbm{1}_{\rho^\varepsilon\leq R}\|_{L_{\rm loc}^\infty(\R_+;L^\kappa(\T^3))} \rightarrow 0$, as $\varepsilon \rightarrow 0$, where $R=+\infty$ and $\kappa=2$, if $\gamma\geq 2$; and where $R\in(3/2,+\infty)$ with $R$ fixed, and $\kappa=\gamma$, if $\gamma <2$. Moreover, since $B^\varepsilon$ is $2$-uniformly integrable in space-time,  we obtain, from the first statement of Lemma~\ref{lem:cv-rho}, $\ \|  B_\myparallel^\varepsilon-  B_\myparallel^\varepsilon \mathbbm{1}_{\rho^\varepsilon\leq R}\|_{L_{\rm loc}^2(\R_+;L^2(\T^3))} \rightarrow 0$, as $\varepsilon \rightarrow 0$, for any $R\in (1,+\infty]$. Indeed,  the $2$-uniform integrability comes from the Gagliardo--Nirenberg interpolation inequality $\|B^\varepsilon \|_{L^{10/3}(\R_+ \times \T^3)} \leq \|B^\varepsilon \|_{L^{\infty}(\R_+,L^2(\T^3))}^{2/5}\|\nabla B^\varepsilon \|_{L^{2}(\R_+,L^2(\T^3))}^{3/5} < \infty$ (from uniform bounds of Lemma~\ref{lem:cv-Bperp}) and the De la Vall\'ee Poussin criterion.\\
2) From the sixth statement of Lemma~\ref{lem:cv-vperp}, we obtain $\|\rho^\varepsilon v_\perp^\varepsilon -v_\perp^\varepsilon \|_{L_{\rm loc}^2(\R_+;L^q(\T^3))} \rightarrow 0$, as $\varepsilon \rightarrow 0$, for $1/q=1/\gamma+1/6$. \\
We now set $\widetilde U_R^\varepsilon:={}^t(\phi_R^\varepsilon, {}^t\Q_\perp v_\perp^\varepsilon)$,   where $\phi_R^\varepsilon := (b^\varepsilon \varrho^\varepsilon + B_\myparallel^\varepsilon)\mathbbm{1}_{\rho^\varepsilon\leq R}$. Then, from Step 1, and the points 1) to 2) of Step 2, we obtain $\mathcal S(-t/\varepsilon) \widetilde U_R^\varepsilon \rightarrow \, \mathcal U$ in  $ L_{\rm loc}^{2}(\R_+; H^{-\sigma}(\T^3))$--strong. Indeed, we have
\begin{equation}
    \label{eqn:regU}
    S(-t/\varepsilon) \widetilde U_R^\varepsilon=\, \mathcal U + S(-t/\varepsilon) \,{}^t\big ( [\phi_R^\varepsilon -\phi^\varepsilon],
    {}^t\Q_\perp[(1-\rho^\varepsilon) v_\perp^\varepsilon]\big) + \mathcal S(-t/\varepsilon) \mathcal R^\varepsilon.
\end{equation}
Since  the group $\mathcal S$ is an isometry in $H^{\alpha}$ ($\alpha\in \R$), from the points 1) to 2) of Step 2, the second term of the right-hand side of \eqref{eqn:regU} vanishes as $\varepsilon \rightarrow 0$ in $L_{\rm loc}^{2}(\R_+; (L^\kappa + L^2)(\T^3;\R)\times L^q(\T^3;\R^2))$--strong, and thus in $L_{\rm loc}^{2}(\R_+; H^{-\sigma}(\T^3))$--strong since $ (L^\kappa + L^2)(\T^3) \hookrightarrow H^{-\upalpha}(\T^3)$ with $\kappa >3/2$, and $L^q(\T^3) \hookrightarrow H^{-\upsigma}(\T^3)$ with $1/q=1/\gamma+1/6$, and where $(\upalpha, \upsigma)$ can be chosen such that $1/2 \leq \upalpha \leq \mathfrak s \leq \upsigma < \sigma $. From \eqref{eqn:GPR}, the third term of the right-hand side of \eqref{eqn:regU} vanishes as $\varepsilon \rightarrow 0$ in  $ L_{\rm loc}^{2}(\R_+; H^{-\sigma}(\T^3))$--strong.\\
Finally, the first assertions of Lemma~\ref{lem:cv-Bperp} and  Lemma~\ref{lem:cv-vperp}, and inequality \eqref{eqn:varrhoeps-est} show that $ \widetilde U_R^\varepsilon$ is bounded in $L_{\rm loc}^{2}(\R_+;  L^2(\T^3;\R^3))$, uniformly with respect to $\varepsilon$. Since  the group $\mathcal S$ is an isometry in $H^{\alpha}$ ($\alpha\in \R$), we deduce that $\, \mathcal U \in  L_{\rm loc}^{2}(\R_+;  L^2(\T^3;\R^3))$, which ends the proof of the first point of Lemma~\ref{lem:UG}.

We now turn to the proof of the point 2 of Lemma~\ref{lem:UG}. Using the equation \eqref{eqn:Ueps} on the component $ \Phi^\varepsilon$, we can see that $\phi^\varepsilon/\varepsilon$ satisfies, $\forall \,\psi_\perp \in H_c^1(\R_+; H^r(\T^3;\R^2))$, with $r>0$ large enough (specified further),
\begin{multline}
  \label{eqn:SPT}
  \int_{\R_+} dt \int_{\T^3}dx \, \frac{\phi^\varepsilon}{\varepsilon} \nabla_\perp \cdot \psi_\perp
  =  -\int_{\T^3}dx \, \Q_\perp(\rho_0^\varepsilon v_{0\perp}^\varepsilon)\cdot \psi_\perp(0)
  -\int_{\R_+}  dt \int_{\T^3}dx \, (\rho^\varepsilon -1)v_\perp^\varepsilon \cdot \Q_\perp \partial_t\psi_\perp
   \\ 
   - \int_{\R_+}  dt \int_{\T^3}dx \, v_\perp^\varepsilon \cdot \Q_\perp \partial_t \psi_\perp
  - \int_{\R_+}  dt \int_{\T^3}dx \, F_2^\varepsilon \cdot  \psi_\perp 
  \ = \ T_0^\varepsilon + T_1^\varepsilon +T_2^\varepsilon +T_3^\varepsilon\,. 
\end{multline}
From properties of initial conditions (see Section~\ref{ss:rT3})  we have $\Q_\perp(\rho_0^\varepsilon v_{0\perp}^\varepsilon)  \rightharpoonup \Q_\perp u_{0\perp} = \Q_\perp v_{0\perp}$ in $ L^{2\gamma/(\gamma+1)}(\T^3)$--weak, and $u_{0\perp}=v_{0\perp} \in L^2(\T^3)$. Defining $\uppi_0 := \Delta_\perp^{-1} \nabla_\perp\cdot\, \Q_\perp u_{0 \perp} \in L^2(\T_\myparallel; H^1(\T_\perp^2))$, we then have $\nabla_\perp \uppi_0= \Q_\perp u_{0\perp} \in L^2(\T^3)$. From  H\"older inequality we obtain  $|T_0^\varepsilon| \leq \| \rho_0^\varepsilon v_{0\perp}^\varepsilon\|_{L^{2\gamma/(\gamma+1)}(\T^3)} \!\! $ $ \|  \psi_\perp(0)\|_{L^{2\gamma/(\gamma-1)}(\T^3)}$. This and continuous Sobolev embeddings, imply that there exists a constant $\mathcal C_0$ (uniform in $\varepsilon$) such that $|T_0^\varepsilon| \leq \mathcal C_0 \|\psi_\perp\|_{H_c^1(\R_+; H^r(\T^3))}$, with $r>(5/2)^+$. From H\"older inequality we obtain $|T_1^\varepsilon|\leq \|\rho^\varepsilon -1 \|_{L_{\rm loc}^\infty(\R_+;L^\gamma(\T^3))} \|v_\perp^\varepsilon\|_{L_{\rm loc}^2(\R_+;L^6(\T^3))} \| \partial_t\Q_\perp \psi_\perp\|_{L^2(\R_+;L^{q'}(\T^3))}$, with $1/q'=1-1/\gamma -1/6$. Then first, from Lemmas~\ref{lem:cv-rho}~and~\ref{lem:cv-vperp}, we obtain $T_1^\varepsilon\rightarrow 0$, as $\varepsilon \rightarrow 0$. Second, using  continuous Sobolev embeddings, there exists a constant $\mathcal C_1$ (uniform in $\varepsilon$) such that $|T_1^\varepsilon| \leq \mathcal C_1 \|\psi_\perp\|_{H_c^1(\R_+; H^r(\T^3))}$, with $r>(5/2)^+$.  Using Lemma~\ref{lem:cv-vperp} (in particular $\Q_\perp v_\perp=0$), we obtain $T_2^\varepsilon\rightarrow 0$, as $\varepsilon \rightarrow 0$. From H\"older inequality, we obtain  $ |T_2^\varepsilon|\leq \|v_\perp^\varepsilon\|_{L_{\rm loc}^2(\R_+;L^2(\T^3))} \| \partial_t \psi_\perp\|_{L^2(\R_+;L^{2}(\T^3))}$,  which implies, together with  Lemma~\ref{lem:cv-vperp} and  continuous Sobolev embeddings, that there exists a constant $\mathcal C_2$ (uniform in $\varepsilon$) such that $|T_2^\varepsilon| \leq \mathcal C_2 \|\psi_\perp\|_{H_c^1(\R_+; H^r(\T^3))}$, with $r>(5/2)^+$. From the Step 1 of this proof, we know that $F_2^\varepsilon \in L_{\rm loc}^{2}(\R_+; H^{-r}(\T^3))$, for $r>5/2+\delta $, and any $\delta >0$. Since $\nabla_\perp \times F_2^\varepsilon=0$ in $\mathcal D'(\R_+^*\times \T^3)$, there exists $f_2^\varepsilon = \Delta_\perp^{-1}\nabla_\perp \cdot F_2^\varepsilon \in  L_{\rm loc}^{2}(\R_+; H^{-r}(\T^3)) $   (uniformly in  $\varepsilon$) such that $F_2^\varepsilon=\nabla_\perp f_2^\varepsilon$. Therefore, there exists a function $\uppi_1 \in  L_{\rm loc}^{2}(\R_+; H^{-r}(\T^3)) $ such that $f_2^\varepsilon \rightharpoonup \uppi_1 $ in $ L_{\rm loc}^{2}(\R_+; H^{-r}(\T^3))$--weak. Moreover, we deduce that there exists a constant $\mathcal C_3$ (uniform in $\varepsilon$) such that $|T_3^\varepsilon| \leq \mathcal C_3 \|\psi_\perp\|_{H_c^1(\R_+; H^r(\T^3))}$, with $r>(5/2)^+$. In summary, we have shown that there exists a constant $\mathcal C:=\sum_{i=0,\ldots,3}\mathcal C_i$, uniform in $\varepsilon$, such that for $r>(5/2)^+$,
\[
\bigg|\int_{\R_+}  dt \int_{\T^3}dx \, \nabla_\perp \Big(\frac{\phi^\varepsilon}{\varepsilon}\Big) \cdot \psi_\perp \bigg|
=\bigg|\int_{\R_+}  dt \int_{\T^3}dx \, \frac{\phi^\varepsilon}{\varepsilon} \nabla_\perp \cdot \psi_\perp \bigg|
\leq  \mathcal C \|\psi_\perp\|_{H_c^1(\R_+; H^r(\T^3))} \,,
\]
which means that $\nabla_\perp (\phi^\varepsilon/\varepsilon)$ and a fortiori $\phi^\varepsilon/\varepsilon$ belong to $H^{-1}(\R_+; H^{-r}(\T^3))$, uniformly in $\varepsilon$. Moreover, we have proved that
\[
\int_{\R_+}  dt \int_{\T^3}dx \, \frac{\phi^\varepsilon}{\varepsilon} \nabla_\perp \cdot \psi_\perp
\xrightarrow{\quad}  \int_{\R_+}  dt \int_{\T^3}dx \,  \uppi_1\nabla_\perp \cdot \psi_\perp\, +
\int_{\T^3}dx \,  \uppi_0\nabla_\perp \cdot \psi_\perp(0)\,.
\]
These two properties establish the point 2 of Lemma~\ref{lem:UG}. Using \eqref{eqn:existpi1}, and Lemmas~\ref{lem:cv-rho}~and~\ref{lem:cv-Bperp},  we obtain that $\phi^\varepsilon \rightharpoonup \phi=b \varrho + B_\myparallel =0$ in  $\mathcal D'(\R_+^* \times \T^3)$, and $b \varrho + B_\myparallel =0$ in $L_{\rm loc}^\infty(\R_+;L^\kappa(\T^3))$, with $\kappa >3/2$, which justify the point 3 of Lemma~\ref{lem:UG}.
\end{proof}  

We are now able to justify the passage to the limit in equation \eqref{eq:vperp-eps} for $\rho^\varepsilon v_\perp^\varepsilon$. Let us start with the initial condition term. From the discussion about the properties of sequences of initial conditions in Section~\ref{ss:rT3} (in particular the uniform bound \eqref{ineqC0-P} and the resulting convergences), we can pass to the limit, in the distributional sense, in the initial condition term of \eqref{eq:vperp-eps} to obtain the limit point  $u_{0 \perp } =v_{0 \perp }=\P_\perp u_{0 \perp } + \Q_\perp u_{0 \perp }$. Using \eqref{eqn:existpi1} in equation \eqref{eq:vperp-eps}, we observe that the term $-\nabla_\perp \uppi_0 = - \Q_\perp u_{0 \perp }$ (coming from the weak limit of $\phi^\varepsilon/\varepsilon$) cancels the irrotational par $\Q_\perp u_{0 \perp }$ of the previous limit point  $u_{0 \perp }$, so that the limit initial condition is simply $\P_\perp u_{0 \perp }=\P_\perp v_{0 \perp }$. This is consistent with the fact that in the limit equation the test function $\psi_\perp$ can be chosen divergence-free, i.e., $\P_\perp \psi_\perp =\psi_\perp$. Moreover, according to \eqref{mhdepssing}, the two conditions $ \nabla_\perp \cdot v_{0\perp} = 0 $ (or  $\nabla_\perp \cdot u_{0\perp} = 0 $, since $u_{0\perp} = v_{0\perp} $) and $ B_{0\myparallel} + b \varrho_0 = 0 $ are related to a preparation of the initial data to avoid fast time oscillations. Since the limit initial condition  $\P_\perp u_{0 \perp }$ comes naturally without any preparation, then in our framework, we can deal with general data satisfying $ \nabla_\perp \cdot u_{0 \perp} =\nabla_\perp \cdot v_{0 \perp} \not \equiv  0$.

We next deal with the linear terms of \eqref{eq:vperp-eps}. Using weak convergence of $v_\perp^\varepsilon$ (resp. $B^\varepsilon$) yielded by the first statement of Lemma~\ref{lem:cv-vperp} (resp. Lemma~\ref{lem:cv-Bperp}) we can pass to the limit, in the distributional sense, in all linear diffusive terms (resp. the linear advection term $B_\perp^\varepsilon \cdot \partial_\myparallel \psi_\perp$) of \eqref{eq:vperp-eps}. Using weak convergence for $B^\varepsilon$ and strong convergence for $B_\perp^\varepsilon$, which are supplied by Lemma~\ref{lem:cv-Bperp}, we can pass to the limit in the quadratic nonlinear term $B_\perp^\varepsilon \otimes B^\varepsilon :D_\varepsilon \psi_\perp$ to obtain the term $B_\perp \otimes B_\perp :D_\perp \psi_\perp$. Using the identity $\I_\perp=\P_\perp+\Q_\perp$, for the term involving time derivative in \eqref{eq:vperp-eps}, we obtain, $\forall \,\psi_\perp \in \mathscr{C}_c^\infty(\R_+\times \T^3;\R^2)$, 
\begin{equation}
  \label{eqn:ptrv}
\int_{\R_+}  \int_{\T^3}dx \, \rho^\varepsilon v_\perp^\varepsilon \cdot\partial_t \psi_\perp
=-T_1^\varepsilon -T_2^\varepsilon + \int_{\R_+}  dt \int_{\T^3}dx \,\P_\perp v_\perp^\varepsilon \cdot \partial_t \psi_\perp \,,
\end{equation}
where the terms $T_1^\varepsilon$ and $T_2^\varepsilon$ are the same as in \eqref{eqn:SPT}. For the same reasons than the ones invoked in the proof of the point 2 of  Lemma~\ref{lem:UG}, $T_1^\varepsilon$ and $T_2^\varepsilon$ vanish as $\varepsilon \rightarrow 0$. Therefore, using \eqref{eqn:ptrv} and Lemma~\ref{lem:cv-Bperp} (in particular $\P_\perp v_\perp = v_\perp$), we obtain that $\partial_t(\rho^\varepsilon v_\perp^\varepsilon) \rightharpoonup \partial_t v_\perp$ in $\mathcal D'(\R_+^*\times \T^3)$. Now in equation \eqref{eq:vperp-eps}, we simultaneously deal with the magnetic pressure term $|B^\varepsilon|^2/2$, the singular fluid pressure term $ p^\varepsilon /\varepsilon^2$ (with $p^\varepsilon=a (\rho^\varepsilon)^\gamma$),  and the singular magnetic term $B_\myparallel^\varepsilon/\varepsilon$. Setting $\mathfrak p^\varepsilon = p^\varepsilon /\varepsilon^2 + B_\myparallel^\varepsilon/\varepsilon + |B^\varepsilon|^2/2$, this term can be rewritten as $ \mathfrak p^\varepsilon =\phi^\varepsilon/\varepsilon  + a (\overline{\rho^\varepsilon})^\gamma/\varepsilon^2 + \uppi_2^\varepsilon\,,$ with $\, \uppi_2^\varepsilon = (\gamma-1)\Pi_2(\rho^\varepsilon) + |B^\varepsilon|^2/2$. In $\mathfrak p^\varepsilon $, the constant term $ a (\overline{\rho^\varepsilon})^\gamma/\varepsilon^2$ is irrelevant because it disappears by spatial integration in \eqref{eq:vperp-eps}. From the point 2 of  Lemma~\ref{lem:UG}, we obtain  $\phi^\varepsilon/\varepsilon \rightharpoonup \uppi_1$ in $\mathcal D'(\R_+^* \times \T^3)$. In fact, from \eqref{eqn:existpi1} we have $\phi^\varepsilon/\varepsilon \rightharpoonup  \delta_0(t) \otimes \uppi_0  + \uppi_1 $ in $H^{-1}(\R_+, H^{-r}( \T^3))$, but as already mentioned above, the term $ \delta_0(t) \otimes \uppi_0 $ cancels the irrotational part  $\Q_\perp u_{0 \perp }$ of the limit term $u_{0 \perp }$, so that  the limit initial condition is $\P_\perp u_{0 \perp }=\P_\perp v_{0 \perp }$. From energy inequality \eqref{nrj-ineq-mhdeps}-\eqref{defPI-P} with the pressure term $\Pi_2$, we obtain  $\uppi_2^\varepsilon \in L_{\rm loc}^\infty(\R_+;L^1(\T^3))$, uniformly with respect to $\varepsilon$. Then, by weak compactness, there exists a function $\uppi_2 \in  L_{\rm loc}^\infty(\R_+;L^1(\T^3))$ such that $\uppi_2^\varepsilon \rightharpoonup \uppi_2$ in $ L_{\rm loc}^\infty(\R_+;L^1(\T^3))$--weak$-*$. Therefore, we obtain $ \mathfrak p^\varepsilon \rightharpoonup (\uppi_1 + \uppi_2)$ in $\mathcal D'(\R_+^* \times \T^3)$.

\begin{remark}
  Even if we have the strong convergence of $B_\perp^\varepsilon$ in $L_{\rm loc}^2(\R_+; L^2(\T^3)))$ and a uniform bound in $L_{\rm loc }^2(\R_+;H^1(\T^3))$ for $B^\varepsilon$, we do not have $|B^\varepsilon|^2\rightharpoonup |B|^2$ in $\mathcal D'(\R_+^*\times\T^3)$. The reason of this lack of convergence comes from  the fact that $\partial_t B_\myparallel $ is not bounded, uniformly with respect to $\varepsilon$, in some suitable functional spaces. Indeed, equation \eqref{eq:Bparal-eps} for $B_\myparallel^\varepsilon$ contains a singular term in $1/\varepsilon$, which prevents such a boundedness. In other words, this is  fast time oscillations in the parallel direction that prevent time compactness and thus, such a convergence. 
\end{remark}

It remains to pass to the limit in the nonlinear term $\rho^\varepsilon v_\perp^\varepsilon \otimes v^\varepsilon $ in \eqref{eq:vperp-eps}. This is the purpose of the following lemma.
\begin{lemma}
  \label{lem:RST} There exists a distribution $\uppi_3 \in \mathcal D'(\R_+^* \times \T^3)$, such that
  \begin{equation*}
    \nabla_\varepsilon \cdot (\rho^\varepsilon v_\perp^\varepsilon \otimes v^\varepsilon) \xrightharpoonup{\quad}
    \nabla_\perp \cdot ( v_\perp \otimes v_\perp)  \ + \ \nabla_\perp \uppi_3 \ \ \mbox{ in } \
    \mathcal{D}'(\R_+^*\times \T^3)\,.
   \end{equation*}
\end{lemma}

Using Lemma~\ref{lem:RST}, we can complete the justification of the passage to the limit, in the distributional sense, in equations  \eqref{eq:vperp-eps} to obtain the second equation of \eqref{rmhd} in the sense of distributions.

\begin{remark}
  In \eqref{rmhdincomp} the pressure term $\uppi$, which can be seen as a Lagrange multiplier ensuring the constraint $\nabla_\perp \cdot v_\perp=0$, comes from the following three contributions, $\uppi_1$, $\uppi_2$ and  $\uppi_3$. The pressure  term $\uppi_1$ comes from the singular fluid pressure term and the singular  magnetic term. The pressure $\uppi_2$ comes from the non-singular fluid and magnetic pressure terms. The pressure term $\uppi_3$ comes from the Reynolds stress tensor.
\end{remark}

\namedproof{\textit{of Lemma~\ref{lem:RST}}}.
Observe first the following decomposition, $\forall \,\psi_\perp \in \mathscr{C}_c^\infty(\R_+\times \T^3;\R^2)$, 
\begin{equation}
  \label{eq:RST-1}
  \begin{aligned}
  \int_{\R_+}  dt \int_{\T^3}dx \, \rho^\varepsilon v_\perp^\varepsilon \otimes v^\varepsilon : D_\varepsilon\psi_\perp
  &= \int_{\R_+}  dt \int_{\T^3}dx \, \rho^\varepsilon v_\perp^\varepsilon \otimes v_\perp^\varepsilon : D_\perp\psi_\perp \\
  &\quad + \varepsilon\int_{\R_+}  dt \int_{\T^3}dx \, \rho^\varepsilon v_\myparallel^\varepsilon v_\perp^\varepsilon \cdot \partial_\myparallel\psi_\perp 
  \ =  \bar{\Gamma}_{1}^\varepsilon + \bar \Gamma_{2}^\varepsilon\,. 
  \end{aligned}
\end{equation}  
Using H\"older  inequality, we obtain $ |\bar \Gamma_{2}^\varepsilon| \leq \varepsilon \|\rho^\varepsilon |v^\varepsilon|^2 \|_{L_{\rm loc}^\infty(\R_+;L^1(\T^3))} \|\partial_\myparallel \psi_\perp \|_{L^1(\R_+;L^\infty(\T^3))}$. Therefore, exploiting the energy estimate, we have $\bar \Gamma_{2}^\varepsilon \rightarrow 0$, as $\varepsilon \rightarrow 0$. \\
To deal with the term $\bar \Gamma_{1}^\varepsilon$, we follow the spirit of the proof of the convergence result in the part III of \cite{LM98}. For this we consider the following decomposition
\begin{equation}
  \label{eq:RST-2}
  \begin{aligned}
    \Gamma_{1}^\varepsilon & :=\nabla _\perp \cdot (\rho^\varepsilon v_\perp^\varepsilon \otimes v_\perp^\varepsilon ) =  \sum_{i=1}^5 \Gamma_{1i}^\varepsilon
     = \nabla _\perp \cdot \big (\,\rho^\varepsilon v_\perp^\varepsilon \otimes\P_\perp v_\perp^\varepsilon
     \, +\,  \P_\perp(\rho^\varepsilon v_\perp^\varepsilon) \otimes\Q_\perp v_\perp^\varepsilon \\
     &\quad \, + \, 
    (\Q_\perp[\rho^\varepsilon v_\perp^\varepsilon]-\Upphi^\varepsilon) \otimes \Q_\perp v_\perp^\varepsilon 
    \, +\,  \Upphi^\varepsilon\otimes (\Q_\perp v_\perp^\varepsilon-\Upphi^\varepsilon) + \Upphi^\varepsilon\otimes\Upphi^\varepsilon\, \bigr) \,,
\end{aligned}
\end{equation}
where we set $\Upphi^\varepsilon:= \mathcal S_2(t/\varepsilon) \,\mathcal U$. We successively deal with the terms $\Gamma_{1i}^\varepsilon $,  for $i\in\{1, \ldots, 5  \}$. We start with  $\Gamma_{11}^\varepsilon $. Using the fourth and fifth assertions of Lemma~\ref{lem:cv-vperp} with $s=\upsigma \in [\mathfrak s , 1)$, we obtain $\Gamma_{11}^\varepsilon \rightharpoonup \nabla_\perp \cdot (v_\perp \otimes v_\perp)$ in $\mathcal D'(\R_+^* \times \T^3)$, which gives the first part of the limit in Lemma~\ref{lem:RST}. We continue with the term $\Gamma_{12}^\varepsilon$.  Using the third and seventh assertions of Lemma~\ref{lem:cv-vperp}, since $1/q+1/6< 1$,  we obtain $\Gamma_{12}^\varepsilon \rightharpoonup 0 $ in $\mathcal D'(\R_+^* \times \T^3)$. For the term $\Gamma_{13}^\varepsilon$, we observe that $\Q_\perp[\rho^\varepsilon v_\perp^\varepsilon]-\Upphi^\varepsilon=\mathcal R_2^\varepsilon$, where the term $\mathcal R_2^\varepsilon \in \R^2$ is the second component of the error term $\mathcal R^\varepsilon ={}^t(\mathcal R_1^\varepsilon,{}^t\mathcal R_2^\varepsilon)$ involved in equation \eqref{eqn:GPR} of Lemma~\ref{lem:UG}. Using \eqref{eqn:GPR} with $\sigma=1$ and the first or the third statement of Lemma~\ref{lem:cv-vperp}, $\Gamma_{13}^\varepsilon \rightharpoonup 0 $ in $\mathcal D'(\R_+^* \times \T^3)$. We pursue with the term $\Gamma_{14}^\varepsilon$. Let $\Upphi_\eta^\varepsilon$ be a regularization of $\Upphi^\varepsilon$ obtained by the following way. Using the fact that $\mathscr{C}_c^\infty$ is dense in $L^2$, we define $\Upphi_\eta^\varepsilon:=\mathcal S_2(t/\varepsilon)\, \mathcal U_\eta$, where $ \, \mathcal U_\eta \in \mathscr{C}_c^\infty(\R_+\times \T^3)$ is such that $\|\, \mathcal U_\eta -\, \mathcal U \|_{L_{\rm loc}^2(\R_+;L^2(\T^3))} \leq \eta$, with $0\leq \eta \ll 1$. We consider the decomposition
  \begin{equation}
    \label{eqn:RST-3}
    \Upphi^\varepsilon \otimes (\Q_\perp v_\perp^\varepsilon-\Upphi^\varepsilon) = (\Upphi_\eta^\varepsilon-\Upphi^\varepsilon)
     \otimes (\Q_\perp v_\perp^\varepsilon-\Upphi^\varepsilon)  + \Upphi_\eta^\varepsilon\otimes  (\Q_\perp v_\perp^\varepsilon-\Upphi^\varepsilon) =: R_{1\eta}^\varepsilon + R_{2\eta}^\varepsilon\,.  
  \end{equation}
  For the term $ R_{1\eta}^\varepsilon$, using the isometry property of $\mathcal S$, there exists a constant $ C=  C(\| \, \mathcal U\|_{L_{\rm loc}^2(\R_+;L^2(\T^3))})$ such that $\| R_{1\eta}^\varepsilon\|_{L_{\rm loc}^1(\R_+;L^1(\T^3))}  \leq  C \eta $. For the term $ R_{2\eta}^\varepsilon$,  using the isometry of $\mathcal S$, we first observe that $\Upphi_\eta^\varepsilon \in L_{\rm loc}^2(\R_+;H^m(\T^3)) $, with $m\geq 0$, and for all $\eta >0$. Second, we claim that $\Q_\perp v_\perp^\varepsilon-\Upphi^\varepsilon \rightarrow 0$ in $L_{\rm loc}^2(\R_+;H^{-\sigma}(\T^3))$--strong, for $\mathfrak s <\sigma < (5/2)^+$. Indeed,  $\Q_\perp v_\perp^\varepsilon-\Upphi^\varepsilon= \mathcal R_2^\varepsilon + \Q_\perp ((1-\rho^\varepsilon) v_\perp^\varepsilon)$, where, using \eqref{eqn:GPR}, we have $\mathcal R_2^\varepsilon \rightarrow 0$ in $L_{\rm loc}^2(\R_+;H^{-\sigma}(\T^3))$--strong. From the sixth statement  of Lemma~\ref{lem:cv-vperp}, the continuity of $\Q_\perp$ in $L^q$ with $1/q=1/\gamma+1/6$, and the embedding $L^q(\T^3) \hookrightarrow H^{-\upsigma}(\T^3)$ for $\upsigma \geq \mathfrak s$ (cf. proof of  Lemma~\ref{lem:cv-vperp}), we obtain $\Q_\perp ((1-\rho^\varepsilon) v_\perp^\varepsilon)\rightarrow 0$ in $L_{\rm loc}^2(\R_+;H^{-\sigma}(\T^3))$--strong, for all $\sigma\geq \upsigma$. Therefore, in the right-hand side of \eqref {eqn:RST-3}, taking first the limit $\varepsilon \rightarrow 0$ and then the limit $\eta \rightarrow 0$, we obtain $\Gamma_{14}^\varepsilon \rightharpoonup 0 $ in $\mathcal D'(\R_+^* \times \T^3)$. It remains to show that  $\Gamma_{15}^\varepsilon \rightharpoonup \uppi_3 \,$ in $\mathcal D'(\R_+^*\times \T^3)$. For this, using Fourier series, we can  compute explicitly the term $\Gamma_{15}^\varepsilon$. We set $\mathfrak U^\varepsilon \equiv {}^t(\upphi^\varepsilon, {}^t \Upphi^\varepsilon) := {}^t(\mathcal S_1(t/\varepsilon) \, \mathcal U,{}^t\mathcal S_2(t/\varepsilon) \, \mathcal U) =  \mathcal S(t/\varepsilon) \, \mathcal U $, with $\, \mathcal U={}^t(\psi,{}^t \Psi)\in L_{\rm loc}^2(\R_+; L^2(\T^3))$, and such that $\int_{\T^3} dx\,\psi=0$, and  $\P_\perp\Psi=0 $. Since $\P_\perp\Psi=0$, then $\Psi =\nabla_\perp \uppsi$, with $\uppsi = \Delta_\perp^{-1}\nabla_\perp \cdot \Psi \in L_{\rm loc}^2(\R_+; L^2(\T_\myparallel ;H^1(\T_\perp^2)))$. This regularity is deduced from the $L^2$ regularity of $\, \mathcal U$. Similarly, since $\mathcal S$ and $\P_\perp$ commute, we obtain $\int_{\T^3} dx\,\upphi^\varepsilon =0 $, and  $\P_\perp\Upphi^\varepsilon=0 $, so that $\Upphi^\varepsilon =\nabla_\perp \upvarphi^\varepsilon$, with $\upvarphi^\varepsilon = \Delta_\perp^{-1}\nabla_\perp \cdot \Upphi^\varepsilon$. We introduce the Fourier series
  \[
  \psi=\sum_{k\in\Z^3} \psi_k(t) e^{{\rm i}k\cdot x}\,, \quad  \Psi={\rm i}\sum_{k\in\Z^3} k_\perp\uppsi_k(t) e^{{\rm i}k\cdot x}\,, 
  \]
  with $\psi_0(t)=0$, and
  \begin{equation}
    \label{eqn:N0}
    \| \{ \psi_k\}_k \|_{L_{\rm loc}^2(\R_+; \ell^2(\Z^3))} + \| \{\uppsi_k |k_\perp|\}_k \|_{L_{\rm loc}^2(\R_+; \ell^2(\Z^3))} =:\mathcal N_0 <\infty\,,
  \end{equation} 
  where the last estimate comes from the $L^2$ regularity of $\, \mathcal U$ stated in Lemma~\ref{lem:UG}.  We  denote by $\mathfrak U_k^\varepsilon={}^t(\upphi_k^\varepsilon, {}^t \Upphi_k^\varepsilon = {\rm i}\,{}^t  k_\perp\upvarphi_k^\varepsilon)$ the Fourier coefficients of $\mathfrak U^\varepsilon \equiv {}^t(\upphi^\varepsilon, {}^t \Upphi^\varepsilon=\nabla_\perp \upvarphi^\varepsilon)$. Inserting  the Fourier series  of $\mathfrak U^\varepsilon $ in the linear equation $\partial_t\,\mathfrak U^\varepsilon = \mathcal L\, \mathfrak U^\varepsilon/\varepsilon $, we are  led to solve linear second-order ODEs in time for the Fourier coefficients $\upphi_k^\varepsilon(t)$ and  $\upvarphi_k^\varepsilon(t)$, with the inital conditions $\mathfrak U_k ^\varepsilon(0)= \, \mathcal U_k(t)$  and  $\partial_t\,\mathfrak U_k^\varepsilon(0)= \mathcal L_k \,\mathcal U_k(t)/\varepsilon $, where $\mathcal L_k = {\rm i}\, {}^t(-c^\varepsilon  k_\perp \cdot\, ,\, {}^t k_\perp)$. Solving these linear ODEs, we obtain for $\Upphi^\varepsilon$,
  \begin{equation*}
    \label{Upphieps}
    \Upphi^\varepsilon= \nabla_\perp \upvarphi^\varepsilon = {\rm i}
    \sum_{k\in\Z^3} e^{{\rm i}k\cdot x} \, k_\perp \, \Big \{
    \uppsi_k(t) \cos\big(\sqrt{c^\varepsilon}|k_\perp|\tfrac t \varepsilon\big) -
    \tfrac{1}{\sqrt{c^\varepsilon}|k_\perp|} \psi_k(t)\sin\big(\sqrt{c^\varepsilon}|k_\perp|\tfrac t \varepsilon\big)
    \Big \}\,.
  \end{equation*}  
  We then obtain
  \begin{equation}
    \label{eqn:upup}
  \Upphi^\varepsilon\otimes \Upphi^\varepsilon
  = -\sum_{k,l \in \Z^3} e^{{\rm i}(k+l)\cdot x} \, \theta_k^\varepsilon(t)  \theta_l^\varepsilon(t) ( k_\perp \otimes l_\perp)  = -\big(S_1^\varepsilon(t,x) + S_2^\varepsilon(t,x)\big)\,, 
  \end{equation}
  with $\theta_k^\varepsilon(t)=\uppsi_k(t) \cos(\sqrt{c^\varepsilon}|k_\perp| t /\varepsilon) - \bigl( {\psi_k(t)}/({\sqrt{c^\varepsilon}|k_\perp|}) \bigr) \sin (\sqrt{c^\varepsilon}|k_\perp| t /\varepsilon) \in L_{\rm loc}^2(\R_+)$, $\forall \, k\in \Z^3$, and where $S_1^\varepsilon$ (resp. $S_2^\varepsilon$) is the sum of $(k,l)$ on the subset $\Lambda_1=\{(k,l)\in \Z^3\times \Z^3 \ ; \ |k_\perp| \neq |l_\perp|\}$ (resp.  $\Lambda_2=\{(k,l) \in \Z^3\times \Z^3 \ ; \ |k_\perp| = |l_\perp|\}$). We next show that $S_1^\varepsilon \rightharpoonup 0$ in $\mathcal D'(\R_+^* \times \T^3)$, and that $\nabla_\perp \cdot S_2^\varepsilon $ is a perpendicular gradient. We first deal with  $S_1^\varepsilon$. Let $\varphi(t,x) = \chi(t) \uplambda(x) \in \mathscr{C}_c^\infty(\R_+\times\T^3)$. Then we obtain,
  \begin{equation}
    \label{eqn:S1eps}
    \langle S_1^\varepsilon,\varphi \rangle 
    \, =\, |\T^3|\!\!\! \sum_{(k,l) \in \Lambda_1} \!\!\!\! \uplambda_{k+l}\, (k_\perp \otimes \,l_\perp) \!\! \int_{\R_+} \!\!\! \! dt \,  \theta_k^\varepsilon(t)  \theta_l^\varepsilon(t) \chi(t) 
    \, =\!\!\!\!\sum_{(k,l) \in \Lambda_1}\!\!\!\! \mathcal H_{kl}^\varepsilon
    \,, 
  \end{equation}
  where $\uplambda_{k}$ is the Fourier coefficients of $\uplambda$. Using Cauchy--Scwharz inequality in time, we obtain $| \mathcal H _{kl}^\varepsilon| \leq |\T^3|\|\chi\|_{L^\infty(\R_+)} d_k d_l |\uplambda_{k+l}|$, with $d_k:=(|k_\perp|\| \uppsi_k\|_{L_{\rm loc}^2(\R_+)} + \| \psi_k\|_{L_{\rm loc}^2(\R_+)} ) \max\{1,1/\sqrt{\underline{c}}\} \geq |k_\perp| \,\| \theta_k^\varepsilon\|_{L_{\rm loc}^2(\R_+)}$. Using this estimate, bound \eqref{eqn:N0}, and Cauchy--Schwarz inequality for one of the infinite sums in \eqref{eqn:S1eps}, we obtain $|\langle S_1^\varepsilon,\varphi \rangle| \leq 2 \mathcal N_0^2 |\T^3|  \,\|\chi\|_{L^\infty(\R_+)}   \max\{1,1/\sqrt{\underline{c}}\}^2 \sum_{k\in \Z^3}|\uplambda_{k}|$. This last sum converges because, from the regularity of $\uplambda$, the Fourier coefficients $\uplambda_k$ decrease enough with respect to $k$. Then $S_1^\varepsilon$ is summable in $\mathcal D'(\R_+^* \times \T^3)$. Now, using the Riemann--Lebesgue theorem and the condition $ |k_\perp| \neq |l_\perp| $ for $ (k,l) \in \Lambda_1 $, recombining the oscillating products implying $ \cos $ and $ \sin$, we easily show that $\theta_k^\varepsilon\theta_l^\varepsilon \rightharpoonup 0$ in $L_{\rm loc}^1(\R_+)$--weak. Using this vanishing limit and the summability of $S_1^\varepsilon$, we obtain  $S_1^\varepsilon \rightharpoonup 0$, in $\mathcal D'(\R_+^* \times \T^3)$. We next deal with $S_2^\varepsilon $. With the same arguments as those used for $S_1^\varepsilon $, $S_2^\varepsilon$ is summable in $\mathcal D'(\R_+^* \times \T^3)$. It remains to show that $\nabla_\perp \cdot S_2^\varepsilon$ is a perpendicular gradient. Symmetrizing the sum $S_2^\varepsilon$ in $(k,l)$ (such that the expression of the general term remains invariant by exchanging $k$ and $l$), using the change of variable $l=n-k$ in $S_2^\varepsilon$, and applying the divergence operator $\nabla_\perp \cdot\,$ to $S_2^\varepsilon$, we obtain
  \[
  \nabla_\perp \cdot S_2^\varepsilon = \frac{\rm i}{2}\sum_{n\in\Z^3} e^{{\rm i}n\cdot x} \!\! \sum_{|k_\perp|=|n_\perp - \,k_\perp|}
  \{(k_\perp \otimes [n_\perp-k_\perp])+( [n_\perp-k_\perp] \otimes k_\perp)\} \, n_\perp \theta_n^\varepsilon \theta_{n-k}^\varepsilon\,.
  \]
  To simplify this expression, we first observe that $|k_\perp|=|n_\perp -k_\perp|$ implies $|n_\perp|^2=2n_\perp\cdot k_\perp$. Using this identity, we  obtain  $(k_\perp \otimes [n_\perp-k_\perp])+( [n_\perp-k_\perp] \otimes k_\perp)= n_\perp(n_\perp\cdot k_\perp)=n_\perp|n_\perp|^2/2$ so that
  \[
  \nabla_\perp \cdot S_2^\varepsilon =\nabla_\perp \Big(\tfrac 14\sum_{n\in\Z^3}
  |n_\perp|^2 e^{{\rm i}n\cdot x}  \sum_{k\in\Z^3} \theta_n^\varepsilon \theta_{n-k}^\varepsilon
  \Big) =: -\nabla_\perp \uppi_3^\varepsilon \xrightharpoonup{\quad}  -\nabla_\perp \uppi_3 \ \ \  \mbox{ in} \ \ \mathcal D'(\R_+^* \times \T^3)\,.
  \]
  This ends the proof of Lemma~\ref{lem:RST}.
\endnamedproof

%%%%%%%%%%%%%%%%%%%%%%

\subsection{Passage to the limit in the equation for $\rho^\varepsilon v_{\raisebox{0.5pt}{$\scriptscriptstyle\parallel$}}^\varepsilon $}
\label{ss:cv-rvpara-P}
Here,  we justify the passage to the limit in equation \eqref{eq:vparal-eps} for $\rho^\varepsilon v_\myparallel^\varepsilon$. Let us start with the initial condition term. From the discussion about the properties of sequences of initial conditions in Section~\ref{ss:rT3} (in particular the uniform bound \eqref{ineqC0-P} and the resulting convergences), we can pass to the limit, in the distributional sense, in the initial condition term of \eqref{eq:vparal-eps} to obtain the limit initial condition $u_{0 \myparallel}=v_{0 \myparallel}$ (since $u_0=v_0$). Next, using weak convergence of $v_\myparallel^\varepsilon$  yielded by the first statement of Lemma~\ref{lem:cv-vperp},  we can pass to the limit, in the distributional sense, in all linear diffusive terms of \eqref{eq:vparal-eps}. Using the same arguments as those used to deal with \eqref{eqn:ptrv} and show that $\partial_t(\rho^\varepsilon v_\perp^\varepsilon) \rightharpoonup \partial_t v_\perp$ in $\mathcal D'(\R_+^* \times \T^3)$, we obtain $\partial_t(\rho^\varepsilon v_\myparallel^\varepsilon) \rightharpoonup \partial_t v_\myparallel$ in $\mathcal D'(\R_+^*\times \T^3)$. Regarding the terms $\varepsilon \lambda^\varepsilon v^\varepsilon \cdot \nabla_\varepsilon \partial_\myparallel \psi_\myparallel $ and $\varepsilon (|B^\varepsilon|/2) \partial_\myparallel \psi_\myparallel$, uniform bounds in $L_{\rm loc}^2(\R_+;L^2(\T^3))$ for $v^\varepsilon $ and $B^\varepsilon$, given by the energy estimate, show that the terms $\varepsilon\lambda^\varepsilon\partial_\myparallel \nabla_\varepsilon\cdot v^\varepsilon $ and $\varepsilon \partial_\myparallel (|B^\varepsilon|/2)$ vanish  in  $\mathcal D'(\R_+^* \times \T^3)$, as $\varepsilon \rightarrow 0$. For the term $B_\myparallel^\varepsilon \otimes B^\varepsilon : D_\varepsilon\psi_\myparallel$, we consider the following decomposition, $\forall \,\psi_\myparallel \in \mathscr{C}_c^\infty(\R_+\times \T^3;\R)$, 
\begin{equation}
  \label{eq:rvp-1}
  \int_{\R_+}  dt \int_{\T^3}dx \, B_\myparallel^\varepsilon  \otimes B^\varepsilon : D_\varepsilon\psi_\myparallel
  = \int_{\R_+}  dt \int_{\T^3}dx \, B_\myparallel^\varepsilon  B_\perp^\varepsilon \cdot\nabla_\perp \psi_\myparallel
  \ + \  \varepsilon\int_{\R_+}  dt\,  |B_\myparallel^\varepsilon|^2 \partial_\myparallel\psi_\myparallel\,.
\end{equation}  
Using the uniform bound in $L_{\rm loc}^2(\R_+;L^2(\T^3))$ for  $B^\varepsilon$, given by the energy estimate, the second term of \eqref{eq:rvp-1} vanishes as $\varepsilon \rightarrow 0$. Using weak convergence of  $B_\myparallel^\varepsilon$ and strong convergence of $B_\perp^\varepsilon $, given respectively by the first and third statements of Lemma~\ref{lem:cv-Bperp}, we obtain $\nabla_\perp \cdot (B_\myparallel^\varepsilon B_\perp^\varepsilon) \rightharpoonup \nabla_\perp \cdot (B_\myparallel B_\perp)$ in  $\mathcal D'(\R_+^* \times \T^3)$, which ends the treatment of the first term of \eqref{eq:rvp-1}. Therefore, we obtain $\nabla_\varepsilon \cdot (B_\myparallel^\varepsilon \otimes B^\varepsilon) \rightharpoonup \nabla_\perp \cdot (B_\myparallel B_\perp)$ in  $\mathcal D'(\R_+^* \times \T^3)$. For the singular fluid pressure term  $  p^\varepsilon /\varepsilon$, we rewrite this term as $ p^\varepsilon /\varepsilon= \varepsilon (\gamma-1)\Pi_2(\rho^\varepsilon) + b^\varepsilon\varrho^\varepsilon + a (\overline{\rho^\varepsilon})^\gamma/\varepsilon$. The constant term $ a (\overline{\rho^\varepsilon})^\gamma/\varepsilon$ is irrelevant because it disappears by spatial integration in \eqref{eq:vparal-eps}. From energy inequality \eqref{nrj-ineq-mhdeps}-\eqref{defPI-P} with the pressure term $\Pi_2$, we obtain $\varepsilon (\gamma-1)\Pi_2(\rho^\varepsilon) \rightarrow 0$ in  $L_{\rm loc}^\infty(\R_+;L^1(\T^3))$--strong. From the second statement of Lemma~\ref{lem:cv-rho}, we obtain $ b^\varepsilon\varrho^\varepsilon \rightharpoonup b \varrho$ in $L_{\rm loc}^\infty(\R_+;L^\kappa(\T^3))$--weak$-*$. Therefore, we obtain $\partial_\myparallel ( p^\varepsilon /\varepsilon) \rightharpoonup b \partial_\myparallel\varrho=-\partial_\myparallel B_\myparallel$ in $\mathcal D'(\R_+^* \times \T^3)$, where for the last equality we have used the point 3 of Lemma~\ref{lem:UG}. It remains to deal with the term $\rho^\varepsilon v_\myparallel^\varepsilon \otimes v^\varepsilon : D_\varepsilon\psi_\myparallel$, for which we consider the following decomposition, $\forall \,\psi_\myparallel \in \mathscr{C}_c^\infty(\R_+\times \T^3;\R)$, 
\begin{equation}
  \label{eq:rvp-2}
    \int_{\R_+}  dt \int_{\T^3}dx \, \rho^\varepsilon v_\myparallel^\varepsilon \otimes v^\varepsilon : D_\varepsilon\psi_\myparallel
  = \int_{\R_+}  dt \int_{\T^3}dx \, \rho^\varepsilon v_\myparallel^\varepsilon  v_\perp^\varepsilon \cdot \nabla_\perp\psi_\myparallel 
  \ +  \ \varepsilon\int_{\R_+}  dt \int_{\T^3}dx \, \rho^\varepsilon |v_\myparallel^\varepsilon|^2 \partial_\myparallel\psi_\myparallel\,. 
\end{equation}  
Since from the energy estimate, $\|\rho^\varepsilon |v^\varepsilon|^2 \|_{L_{\rm loc}^\infty(\R_+;L^1(\T^3))}$ is bounded uniformly with respect to $\varepsilon$, the second term of \eqref{eq:rvp-2} vanishes as $\varepsilon \rightarrow 0$. Finally, it  remains to deal with the first term of \eqref{eq:rvp-2}. This term is the most delicate, because we have only weak compactness for $\rho^\varepsilon v_\myparallel^\varepsilon$  and $v_\perp^\varepsilon$. Indeed, even if $\P_\perp v_\perp^\varepsilon$ converge strongly,  $\Q_\perp v_\perp^\varepsilon$ converge weakly (to zero, see Lemma~\ref{lem:cv-vperp}). Therefore, to pass to the limit in this term, we will use  Lemma~\ref{lem:compactness} of Appendix~\ref{s:tools}, for which we show below that its hypotheses, splitted in three points, are satisfied. 1) From the fifth statement of Lemma~\ref{lem:cv-vperp}, we obtain $ \rho^\varepsilon v_\myparallel^\varepsilon \rightharpoonup  v_\myparallel $ in $L_{\rm loc}^2(\R_+; L^{6\gamma/(6+\gamma)}(\T^3))$--weak. 2) The uniform bound $v_\perp^\varepsilon \in  L_{\rm loc}^2(\R_+; L^6(\T^3))$ and Lemma~4.3 in \cite{Bre11} imply $\|v_\perp^\varepsilon (t,\cdot +h )- v_\perp^\varepsilon (t,\cdot )\|_{L_{\rm loc}^2(\R_+; L^6(\T^3))} \rightarrow 0$, as $|h|\rightarrow 0$, uniformly with respect to $\varepsilon$. 3) From equation \eqref{eq:vparal-eps} for $\rho^\varepsilon v_\myparallel^\varepsilon$, we obtain in $\mathcal D'(\R_+^*\times \T^3)$,
\begin{equation}
  \label{eqn:ptvpa}
  \partial_t(\rho^\varepsilon v_\myparallel^\varepsilon)
  =-\nabla_\varepsilon \cdot (\rho^\varepsilon v_\myparallel^\varepsilon \otimes  v^\varepsilon )
  + \partial_\myparallel \big(\tfrac 1 \varepsilon p^\varepsilon + \tfrac{\varepsilon}2 |B^\varepsilon|^2\big)+
  \nabla_\varepsilon \cdot(B_\myparallel^\varepsilon \otimes  B^\varepsilon )
  + \mu_\perp^\varepsilon \Delta_\perp  v_\myparallel^\varepsilon + \mu_\myparallel^\varepsilon \Delta_\myparallel v_\myparallel^\varepsilon
  +\varepsilon \lambda^\varepsilon \partial_\myparallel\nabla_\varepsilon \cdot v^\varepsilon
  \,.
\end{equation}
Using the energy estimate and the preceding bound for the pressure term $p^\varepsilon/\varepsilon$ (already used in this section), we obtain $ \partial_t(\rho^\varepsilon v_\myparallel^\varepsilon) \in  L_{\rm loc}^{\infty}(\R_+; (W^{-1,1} +W^{-1,\kappa}) (\T^3)) +  L_{\rm loc}^{2}(\R_+; H^{-1}(\T^3)) \hookrightarrow  L_{\rm loc}^{1}(\R_+; W^{-1,1}(\T^3))$. Gathering points 1) to 3), we can apply  Lemma~\ref{lem:compactness} of Appendix~\ref{s:tools} with $g^\varepsilon =\rho^\varepsilon v_\myparallel^\varepsilon$,  $h^\varepsilon = v_\perp^\varepsilon$, $p_1=q_1=2$ ($1/p_1+1/q_1=1$), $p_2=6\gamma/(6+\gamma)$, and $q_2=6$ ($1/p_2+1/q_2=1/\gamma+1/3 < 1$, for $\gamma>3/2$), to obtain $ \nabla_\perp \cdot (\rho^\varepsilon v_\myparallel^\varepsilon v_\perp^\varepsilon)  \rightharpoonup \nabla_\perp \cdot ( v_\myparallel v_\perp)$ in $\mathcal D'(\R_+^* \times \T^3)$. In conclusion, we have shown that the weak formulation \eqref{eq:vparal-eps} converges to the second equation of \eqref{rmhd} in the sense of distributions.

%%%%%%%%%%%%%%%%%%%%%%%%%

\subsection{Passage to the limit in a combination of the equations for $ \varrho^\varepsilon $ and  $B_{\raisebox{0.5pt}{$\scriptscriptstyle\parallel$}}^\varepsilon $}
\label{ss:cv-Bpara-P}
The passage to the limit in the equation of $B_\myparallel^\varepsilon$ is more delicate, because, unlike what is done to treat the asymptotic limit in the perpendicular direction, here, we can not use the unitary group method to deal with the singular term in $1/\varepsilon$ in equation \eqref{eq:Bparal-eps} for $B_\myparallel^\varepsilon$. From the study of the asymptotic limit in the perpendicular direction, more precisely the point 3 of the Lemma~\ref{lem:UG}, we observe a relation between $B_\myparallel$ and $\varrho$, namely, $B_\myparallel + b \varrho = B_\myparallel + \mathbbmssit{p} =0$, where we define $ \mathbbmssit{p} =a \gamma \varrho$. From this relation, the idea is to cancel the   $1/\varepsilon$-singularity in equation \eqref{eq:Bparal-eps} for $B_\myparallel^\varepsilon$ with the  $1/\varepsilon$-singularity coming from the equation of $\varrho^\varepsilon$, this latter equation being obtained from equation \eqref{eq:rho-eps} for $\rho^\varepsilon$. Indeed, from equation \eqref{eq:rho-eps}, we construct the following equation for $\varrho^\varepsilon/\overline{\varrho^\varepsilon} = (\rho^\varepsilon/\overline{\varrho^\varepsilon}-1)/\varepsilon$, $\forall \varphi \in \mathscr{C}_c^\infty(\R_+\times \T^3; \R)$,
\begin{equation}
  \label{eq:varrho-eps}
  \int_\Omega dx \,\frac{\varrho_0^\varepsilon} {\overline{\rho_0^\varepsilon}}\varphi(0) +
  \int_0^\infty dt \int_\Omega dx \, \bigg (\frac{\varrho^\varepsilon}{\overline{\rho^\varepsilon}} \big(
  \partial_t  + v_\perp^\varepsilon \cdot \nabla_\perp \big) \varphi + \frac{1}{\varepsilon}v_\perp^\varepsilon \cdot \nabla_\perp \varphi
  + \frac{\rho^\varepsilon}{\overline{\rho^\varepsilon}}v_\myparallel^\varepsilon \partial_\myparallel \varphi \,\bigg) 
  =0\,.
\end{equation}
Let us define the auxilliary component
\[
\mathbbmsl{B}_\myparallel^\varepsilon := \frac{1}{\mathbbl{c}} \ \Bigl(B_\myparallel^\varepsilon -\frac{\varrho^\varepsilon}{\overline{\rho^\varepsilon}}\Bigr) \, , \qquad \mathbbmsl{ B}_{0\myparallel}^\varepsilon := \frac{1}{\mathbbl{c}} \ \Bigl(B_{0\myparallel}^\varepsilon -\frac{\varrho^\varepsilon}{\overline{\rho^\varepsilon}}\Bigr) \, , \qquad \mathbbl{c} = 1 + \frac{1}{b} \, .
\]
Taking $ \varphi = \psi_\myparallel $ in \eqref{eq:varrho-eps}, where $ \psi_\myparallel $ is the same test function as the one used in equation \eqref{eq:Bparal-eps}, and substracting to \eqref{eq:Bparal-eps}, we obtain
\begin{multline}
\label{eq:q-eps} \mathbbl{c} 
\int_\Omega dx \, \mathbbmsl{B}_{0\myparallel}^\varepsilon \,  \psi_\myparallel(0) + \mathbbl{c} 
\int_0^\infty dt \int_\Omega dx \, \big(
\mathbbmsl{B}_\myparallel^\varepsilon \,  (\partial_t  + v_\perp^\varepsilon \cdot \nabla_\perp ) \psi_\myparallel
\\
- v_\myparallel^\varepsilon \,  B_\perp^\varepsilon \cdot \nabla_\perp \psi_\myparallel
- \frac{\rho^\varepsilon}{\overline{\rho^\varepsilon}}\, v_\myparallel^\varepsilon \, \partial_\myparallel  \psi_\myparallel
+ \eta_\perp B_{\myparallel }^\varepsilon \, \Delta_\perp \psi_\myparallel + \eta_\myparallel B_{\myparallel }^\varepsilon \, \Delta_\myparallel \psi_\myparallel
\big) 
=0\, .
\end{multline}
In the sense of distributions, this reveals  slow dynamics on $ \mathbbmsl{B}_\myparallel^\varepsilon $. Since $ B_\myparallel + b \varrho=0$, the weak limit of $ \mathbbmsl{B}_\myparallel^\varepsilon $ is $ \mathbbmsl{B}_\myparallel = (B_\myparallel - \varrho) / \mathbbl{c} =  B_\myparallel $. This means that, after a boundary layer near $ t = 0 $, the asymptotic behavior  of $ \mathbbmsl{B}_\myparallel^\varepsilon $ is similar to the one of $ B_\myparallel $. And, because $ B_\myparallel + b \varrho=0$, the time evolution of $ B_\myparallel $ provides information on the first-order pressure $ \mathbbmssit{p} = b \varrho$ or on the first-order density $\varrho $. Now, to exhibit the equation satisfied by $ B_\myparallel $, we pass to the limit in \eqref{eq:q-eps}. Let us start with the initial condition term. From the assumptions and the  discussion about the properties of sequences of initial conditions in Section~\ref{ss:rT3}, we have $\varrho_0^\varepsilon \rightharpoonup \varrho_0$ in $L^\kappa(\T^3)$--weak with $\kappa=\min\{2,\gamma\}$, $ \overline{\rho^\varepsilon} \rightarrow 1 $ and $ B_{0\myparallel}^\varepsilon \rightharpoonup B_{0\myparallel} $ in $L^2(\T^3)$--weak. It follows that $ \mathbbmsl{B}_{0\myparallel}^\varepsilon \rightharpoonup \mathbbmsl{B}_{0\myparallel} := (B_{0\myparallel} - \varrho_0 )/\mathbbl{c} $. Observe that $ \mathbbmsl{B}_{0\myparallel} =  B_{0\myparallel} $ if and only if $ B_{0\myparallel} + b \varrho_0 = 0 $. In view of \eqref{mhdepssing}, the two conditions $ \nabla_\perp \cdot v_{0\perp} = 0 $ and $ B_{0\myparallel} + b \varrho_0 = 0 $ are related to a preparation of the initial data to avoid fast time oscillations. Still, in our framework, we can incorporate general data satisfying $ \mathbbmsl{B}_{0\myparallel} \not \equiv  B_{0\myparallel} $ (and also  $ \nabla_\perp \cdot v_{0\perp} \not \equiv 0$, see Section~\ref{ss:cv-rvperp-P}).

It is obvious that $ \partial_t \mathbbmsl{ B}_\myparallel^\varepsilon  \rightharpoonup  \partial_t \mathbbmsl{B}_\myparallel = \partial_t B_\myparallel $ in $\mathcal D'(\R_+^* \times \T^3)$. We next deal with the linear terms of \eqref{eq:q-eps}. Using weak convergence of $B_\myparallel^\varepsilon$, yielded by the first statement of Lemma~\ref{lem:cv-Bperp}, we can pass to the limit, in the distributional sense, in all linear diffusive terms of \eqref{eq:q-eps}. Using the fifth statement of  Lemma~\ref{lem:cv-vperp}, we obtain $\partial_\myparallel (\rho^\varepsilon v_\myparallel^\varepsilon) \rightharpoonup \partial_\myparallel v_\myparallel$ in $\mathcal D'(\R_+^* \times \T^3)$. Using weak convergence of $ v_\myparallel^\varepsilon$ and the strong convergence of $B_\perp^\varepsilon$, given  by respectively the first statement of Lemma~\ref{lem:cv-vperp} and the third statement of  Lemma~\ref{lem:cv-Bperp}, we obtain $\nabla_\perp \cdot(v_\myparallel^\varepsilon B_\perp^\varepsilon) \rightharpoonup \nabla_\perp \cdot(v_\myparallel B_\perp)$ in  $\mathcal D'(\R_+^* \times \T^3)$. Finally it remains to pass in the limit in the term $\nabla_\perp \cdot(\mathbbmsl{B}_\myparallel^\varepsilon  v_\perp^\varepsilon)$. On the one hand, we have weak compactness for $\mathbbmsl{B}_\myparallel^\varepsilon $, and on the other hand, we have only weak compactness for $v_\perp^\varepsilon$, because, despite the strong compactness of solenoidal part $\P_\perp v_\perp^\varepsilon$, the irrotational part $\Q_\perp v_\perp^\varepsilon$ converges only weakly (see Lemma~\ref{lem:cv-vperp}). Therefore, to pass to the limit in this term, we will use  Lemma~\ref{lem:compactness} of Appendix~\ref{s:tools}, for which we show below that its hypotheses, splitted in three points, are satisfied. 1) We first recall that $ \mathbbl{c}\mathbbmsl{B}_\myparallel^\varepsilon = B_\myparallel^\varepsilon - ( \varrho^\varepsilon/\overline{\rho^\varepsilon} ) $. Using the second statement of Lemma~\ref{lem:cv-rho} together with the first statement of Lemma~\ref{lem:cv-Bperp}, for $\kappa=\min\{2,\gamma\}$, we obtain that $ \mathbbmsl{B}_\myparallel^\varepsilon \in L_{\rm loc}^\infty(\R_+; (L^\kappa+L^2)(\T^3)) \hookrightarrow L_{\rm loc}^\infty(\R_+;L^\kappa(\T^3))$. Therefore, by weak compactness, we obtain $ \mathbbl{c}  \mathbbmsl{B}_\myparallel^\varepsilon \rightharpoonup  \mathbbl{c} \mathbbmsl{B}_\myparallel = B_\myparallel  - \varrho $
in $L_{\rm loc}^\infty(\R_+;L^\kappa(\T^3))$--weak$-*$. 2) The uniform bound $v_\perp^\varepsilon \in  L_{\rm loc}^2(\R_+; L^6(\T^3))$ and Lemma~4.3 in \cite{Bre11} imply $\|v_\perp^\varepsilon (t,\cdot +h )- v_\perp^\varepsilon (t,\cdot )\|_{L_{\rm loc}^2(\R_+; L^6(\T^3))} \rightarrow 0$, as $|h|\rightarrow 0$, uniformly with respect to $\varepsilon$. 3) Using the uniform $L_{\rm loc}^\infty(\R_+;L^\kappa(\T^3))$--bound for $ \mathbbmsl{B}_\myparallel^\varepsilon $,  the uniform $L_{\rm loc}^2(\R_+;L^6(\T^3))$--bound for  $v^\varepsilon$, the uniform $L_{\rm loc}^2(\R_+;L^6\cap H^1(\T^3))$--bound for $B_\perp^\varepsilon$, and the uniform  $L_{\rm loc}^2(\R_+;L^q(\T^3))$--bound for $\rho^\varepsilon v_\myparallel^\varepsilon$, with  $q=6\gamma/(6+\gamma)$, we obtain from equation \eqref{eq:q-eps} and H\"older inequality
\[
\begin{array}{rl}
  \partial_t \mathbbmsl{B}_\myparallel^\varepsilon \in  L_{\rm loc}^{2}(\R_+; W^{-1,6\kappa/(6+\kappa)}(\T^3)) \! \! \! & + \, L_{\rm loc}^1(\R_+; W^{-1,3}(\T^3)) \smallskip \\
  \ & + \, L_{\rm loc}^{2}(\R_+; (W^{-1,q}+H^{-1})(\T^3)) \hookrightarrow  L_{\rm loc}^{1}(\R_+; W^{-1,1}(\T^3)) \, ,
\end{array}
\]
where the last continuous injection comes from Sobolev embeddings. Gathering points 1) to 3), we can apply  Lemma~\ref{lem:compactness} of Appendix~\ref{s:tools} with $g^\varepsilon = \mathbbmsl{ B}_\myparallel^\varepsilon$,  $h^\varepsilon = v_\perp^\varepsilon$, $p_1=\infty$, $q_1=2$ ($1/p_1+1/q_1 < 1$), $p_2=\kappa$, and $q_2=6$ ($1/p_2+1/q_2\leq 2/3+1/6 < 1$, for $\gamma\,, \kappa>3/2$), to obtain $ \nabla_\perp \cdot ( \mathbbmsl{B}_\myparallel^\varepsilon  v_\perp^\varepsilon )  \rightharpoonup \nabla_\perp \cdot (  \mathbbmsl{B}_\myparallel v_\perp)$ in $\mathcal D'(\R_+^* \times \T^3)$. In conclusion, the weak formulation \eqref{eq:q-eps} converges to
\begin{multline*}
 \mathbbl{c}
\int_\Omega dx \, \mathbbmsl{B}_{0\myparallel} \,  \psi_\myparallel(0) +  \mathbbl{c}
\int_0^\infty dt \int_\Omega dx \, \big( B_\myparallel \,  (\partial_t  + v_\perp \cdot \nabla_\perp ) \psi_\myparallel
\\
- v_\myparallel \,  B_\perp \cdot \nabla_\perp \psi_\myparallel
- v_\myparallel \, \partial_\myparallel  \psi_\myparallel
+ \eta_\perp B_{\myparallel} \, \Delta_\perp \psi_\myparallel + \eta_\myparallel B_{\myparallel } \, \Delta_\myparallel \psi_\myparallel
\big)
=0\, .
\end{multline*}
Knowing (by passing to the weak limit) that $ \nabla_\perp \cdot B_\perp = 0 $, we recover the first equation of \eqref{rmhd} with  the initial data prescribed in \eqref{inidatparal}.

%%%%%%%%%%%%%%%%%%%%%%%%%%%%%

\section{Asymptotic analysis in the whole space}
\label{s:AAWS}

This section is devoted to the proof of Theorem~\ref{TH-CV-R}. As in the periodic case, we first obtain some weak and strong compactness properties for the same sequences. Since in the whole case the density is only locally integrable in space, some of these compactness results need different proofs. 

%%%%%%%%%%%%%%%%%%%%%%%%%%

\subsection{Compactness of $\rho^\varepsilon $ and  $\varrho^\varepsilon $ }
Here, we aim at proving the following lemma.
\begin{lemma}
  \label{lem:cv-rho-R}
  There exists a generic constant $C>0$, which may depend on $C_0$,  $a$, and $\gamma$ such that the sequences ${\rho^\varepsilon}$ and  $\varrho^\varepsilon:=(\rho^\varepsilon -1)/\varepsilon$ satisfy the following properties.
  \begin{align*}
    &\| \rho^\varepsilon \|_{ L_{\rm loc}^\infty(\R_+; L_{\rm loc }^\gamma(\R^3))} \leq C\,,
    \quad {\rm and} \quad (\rho^\varepsilon -1) \in L_{\rm loc}^\infty(\R_+; L_2^\gamma \cap H^{-\upalpha}(\R^3))\,, \quad \forall \,\gamma>1 \,, \quad \upalpha \geq 1/2\,,\\
    &\| \rho^\varepsilon -1\|_{ L_{\rm loc}^\infty(\R_+; L^\gamma(\R^3))} \leq  C \varepsilon^{2/\gamma}\,,
    \quad {\rm and} \quad  \| \rho^\varepsilon -1\|_{ L_{\rm loc}^\infty(\R_+; L^2(\R^3))} \leq  C \varepsilon\,,
     \quad \forall\, \gamma \geq 2\,, \\
    & \| \rho^\varepsilon -1\|_{ L_{\rm loc}^\infty(\R_+; L_{\rm loc}^\gamma(\R^3))} \leq  C \varepsilon\,, \quad {\rm for} \ {\rm all } \ \  1<\gamma <2\,,\\
    & \rho^\varepsilon \xrightarrow{\quad} 1 \ \ \mbox{ \rm in } \  L_{\rm loc}^\infty(\R_+; L_2^\gamma \cap  L_{\rm loc}^\gamma(\R^3))\!-\!\mbox{\rm strong}\,, \quad  \forall\, \gamma >1\,,\\
    &  \| \varrho^\varepsilon \|_{ L_{\rm loc}^\infty(\R_+; L_2^\kappa\cap L_{\rm loc }^\kappa\cap H^{-\upalpha}(\R^3))} \leq C\,, 
    \quad \kappa=\min\{2,\gamma\} \,, \quad \upalpha\geq 1/2\\ 
    & \varrho^\varepsilon \xrightharpoonup{\quad} \varrho\ \  \mbox{ \rm in } \  L_{\rm loc}^\infty(\R_+; L_{\rm loc }^\kappa \cap H^{-\upalpha}(\R^3))\!-\!\mbox{\rm weak--}\!\ast\,,  \quad \kappa=\min\{2,\gamma\} \,, \quad \upalpha\geq 1/2\,.
  \end{align*}
\end{lemma}

\begin{proof}
  We start with the first bound of the first assertion of Lemma~\ref{lem:cv-rho-R}. We first claim that $\rho^\varepsilon \in L_{\rm loc}^\infty(\R_+; L_{\rm loc}^\gamma(\R^3))$ if  $\rho^\varepsilon \in L_{\rm loc}^\infty(\R_+; L_{\rm loc}^1(\R^3))$. Indeed, using the convexity of the power function $\R^+\ni x\mapsto x^\gamma$ (with $\gamma >1$), and energy inequality \eqref{nrj-ineq-mhdeps}-\eqref{dissip-D} with the pressure term $\Pi_3$, there exists a constant $ c_0$, such that for any compact set $K\subset \R^3$, $0\leq \int_{K}dx \,\{ (\rho^\varepsilon)^\gamma - \gamma \rho^\varepsilon +\gamma -1 \}\leq  c_0$. Then, $ \int_{K}dx \,(\rho^\varepsilon)^\gamma \leq  c_0 + (\gamma-1)|K|+\gamma \int_{K}dx \,\rho^\varepsilon <\infty$, if  $\rho^\varepsilon \in L_{\rm loc}^\infty(\R_+; L_{\rm loc}^1(\R^3))$. It remains to prove this last property. For this, we consider a test function $\varphi \in \mathscr{C}_c^\infty(\R^3)$, such that $\varphi \geq 0$, and $\varphi \equiv 1$ on a subset $ K $ of its support $ S $. From the mass conservation law and the energy estimate, we obtain
\begin{align*}
  \frac{d}{dt} \int_{S} dx \, \rho^\varepsilon  \varphi  =
  \int_S dx\, \rho^\varepsilon v^\varepsilon \cdot \nabla \varphi 
   \leq \| 2 \nabla \sqrt{\varphi}\|_{L^\infty(\R^3)}\bigg(\int_{\R^3}dx\, \rho^\varepsilon |v^\varepsilon|^2\bigg)
  \bigg(\int_{S}dx\, \rho^\varepsilon \varphi \bigg) \leq \mathfrak c_0
  \bigg(\int_{S} dx\,\rho^\varepsilon \varphi\bigg) \,,
\end{align*}
where the constant $\mathfrak c_0$ depends on $C_0$ and $\varphi$. This inequality leads to $ \int_{K} dx\, |\rho^\varepsilon| \leq \int_{S} dx\, \rho^\varepsilon \varphi \leq \exp(\mathfrak c_0 t)\int_{S} dx\, \rho_0^\varepsilon \varphi $, which shows that $\rho^\varepsilon \in L_{\rm loc}^\infty(\R_+; L_{\rm loc}^1(\R^3))$. It can also be shown from the second and third statements of  Lemma~\ref{lem:cv-rho-R}. Now, from the uniform bound \eqref{ineqC0-R} and energy inequality \eqref{nrj-ineq-mhdeps}-\eqref{dissip-D} with the pressure term $\Pi_3$ defined by \eqref{defPI-R}, we obtain $\Pi_3(\rho^\varepsilon) \in L_{\rm loc}^\infty(\R_+;L^1(\R^3))$ uniformly with respect to $\varepsilon$. From this bound and Lemma~\ref{lem:ORC} (with $f=\rho^\varepsilon$ and $\bar f=1$), we obtain, for any $T\in (0,+\infty)$,
\begin{align}
  \label{eqn:bdze}
 &\sup_{t\in[0,T]}  \int_{\R^3}dx \,\{ |\rho^\varepsilon-1|^2
\mathbbm{1}_{\{|\rho^\varepsilon-1|\leq \delta\}} + |\rho^\varepsilon-1|^\gamma
\mathbbm{1}_{\{|\rho^\varepsilon-1|>\delta\}} \} 
= \sup_{t \in [0,T]}\int_{\R^3} dx\, \mathfrak Z_{2,\delta}^{\gamma,1}(\rho^\varepsilon) \nonumber \\
& \hspace{2.5cm} \leq \sup_{t \in [0,T]} \frac{1}{\kappa_1}\int_{\R^3} dx\, \Uppi_{1,\gamma}(\rho^\varepsilon)
= \sup_{t \in [0,T]} \frac{(\gamma-1)\varepsilon^2}{\kappa_1  a} \int_{\R^3} dx\, \Pi_3(\rho^\varepsilon)
\leq\frac{ C_0(\gamma-1)\varepsilon^2}{\kappa_1  a}\,.
\end{align}
This inequality implies the bound in $L_{\rm loc}^\infty(\R_+; L_2^\gamma(\R^3))$ in the second part of the first  statement of Lemma~\ref{lem:cv-rho-R}, and also the strong convergence of $\rho^\varepsilon$ to $1$ in $L_{\rm loc}^\infty(\R_+, L_2^\gamma(\R^3))$ in the fourth statement of Lemma~\ref{lem:cv-rho-R}. To complete the proof of the second part of the first  statement of Lemma~\ref{lem:cv-rho-R}, we have to show that $L_2^\gamma(\R^3) \hookrightarrow H^{-\upalpha}(\R^3)$, with $\upalpha \geq 1/2$. This embedding is obvious for $\gamma=2$. For $\gamma \neq 2$, from the definition of the space  $L_2^\gamma(\R^3)$, the density $\rho^\varepsilon(t)-1$ can be splitted into parts $d_1(t)^\varepsilon:=(\rho^\varepsilon(t) -1) \mathbbm{1}_{\{|{\rho^\varepsilon(t)-1}|\leq \delta\}} \in L^2(\R^3)$  and $d_2^\varepsilon(t):=(\rho^\varepsilon(t) -1) \mathbbm{1}_{\{|{\rho^\varepsilon(t)-1}|> \delta\}} \in L^\gamma(\R^3)$. The part $d_1^\varepsilon(t)$ is obviously in $H^{-\upalpha}(\R^3)$ with $\upalpha \geq 1/2$.  For the part $d_2^\varepsilon(t)$ we proceed as follows. By Sobolev embeddings we have $H^\upalpha(\R^3)  \hookrightarrow L^{\gamma'}(\R^3)$ with $1/\gamma +1/\gamma'=1$, and $\upalpha>1/2\, $ since $\gamma>3/2$. By duality we then obtain $d_2^\varepsilon(t) \in L^\gamma(\R^3)  \hookrightarrow H^{-\upalpha}(\R^3)$. Therefore $(\rho^\varepsilon(t) -1) \in  H^{-\upalpha}(\R^3)$, with $\upalpha \geq 1/2$. We continue with the second assertion of  Lemma~\ref{lem:cv-rho-R}. Estimate \eqref{eqn:bdze} and the inequality $|\rho^\varepsilon -1|^2 \geq |\rho^\varepsilon-1|^\gamma$, for $\gamma\geq 2$ and $|\rho^\varepsilon-1| \leq \delta <1$, imply the first part of the second statement of Lemma~\ref{lem:cv-rho-R}. The second part of this second statement appeals to Lemma~\ref{lem:convex}. Indeed, from this lemma, for $\gamma \geq 2$, we obtain $|\rho^\varepsilon -1|^2 \leq \varepsilon^2(\gamma-1)/(\nu_1  a)\Pi_3(\rho^\varepsilon)$, which gives $\| \rho^\varepsilon -1\|_{ L_{\rm loc}^\infty(\R_+; L^2(\R^3))} \leq  \varepsilon \sqrt{C_0 (\gamma-1)/(\nu_1  a)}$. We continue with the third assertion of Lemma~\ref{lem:cv-rho-R}. Using Cauchy--Schwarz inequality and estimate \eqref{eqn:bdze}, we obtain, for any compact set $K\subset \R^3$, and any $T\in (0,+\infty)$,
\begin{align*}
 \sup_{t\in [0,T]} \|\rho^\varepsilon(t) -1\|_{L^\gamma(K)}^\gamma
  & \leq |K|^{1-\gamma/2}  \sup_{t\in [0,T]} \bigg(\int_K dx\, |\rho^\varepsilon -1|^2
  \mathbbm{1}_{\{|{\rho^\varepsilon-1}|\leq \delta\}} \bigg)^{\gamma/2} \\
  & \quad + \sup_{t\in [0,T]} \int_K dx\, |\rho^\varepsilon -1|^\gamma
\mathbbm{1}_{\{|{\rho^\varepsilon-1}|> \delta\}} \\
& \leq C(|K|, C_0,  a,  \gamma) (\varepsilon^\gamma + \varepsilon^2) \,,
\end{align*}
which justifies the third assertion of Lemma~\ref{lem:cv-rho-R}. Then, the strong convergence of $\rho^\varepsilon$ to $1$ in $L_{\rm loc}^\infty(\R_+, L_{\rm loc}^\gamma(\R^3))$ in the fourth statement of Lemma~\ref{lem:cv-rho-R}, is obtained from the first part of the second statement of Lemma~\ref{lem:cv-rho-R}, and the third statement of Lemma~\ref{lem:cv-rho-R}. We continue with the fifth statement of Lemma~\ref{lem:cv-rho-R}. The uniform bound in $L_{\rm loc}^\infty(\R_+; L_{\rm loc }^\kappa(\R^3))$ for $\varrho^\varepsilon$ comes from the second part of the second statement of Lemma~\ref{lem:cv-rho-R}, and the third statement of Lemma~\ref{lem:cv-rho-R}. For the uniform  bound of  $\varrho^\varepsilon$ in $L_{\rm loc}^\infty(\R_+; L_2^\kappa(\R^3))$, we distinguish two cases according to the value of $\gamma$. For $\gamma \geq 2$, since $\kappa=2$, the second part of the second statement of Lemma~\ref{lem:cv-rho-R} implies the fifth one. For $1<\gamma <2$, estimate \eqref{eqn:bdze} leads to
\[
\sup_{t\in[0,T]}  \int_{\R^3}dx \, \Big\{ \Big|\frac{\rho^\varepsilon-1}{\varepsilon}\Big|^2
\mathbbm{1}_{\{|\frac{\rho^\varepsilon-1}{\varepsilon}|\leq \frac{ \delta}{\varepsilon}\}} + 
\varepsilon^{\gamma-2}\Big|\frac{\rho^\varepsilon-1}{\varepsilon}\Big|^\gamma
\mathbbm{1}_{\{|\frac{\rho^\varepsilon-1}{\varepsilon}|> \frac{ \delta}{\varepsilon}\}}
\Big\} \leq \frac{C_0(\gamma-1)}{\kappa_1 a}\,,
\]
for any $T\in (0,+\infty)$. This last  estimate and  inequality $\varepsilon^{\gamma-2}>1$, imply the bound of $\varrho^\varepsilon$ in $L_{\rm loc}^\infty(\R_+; L_2^\kappa(\R^3))$. To complete the fifth assertion of Lemma~\ref{lem:cv-rho-R}, we observe, as above, that we have the embedding $L_2^\kappa(\R^3) \hookrightarrow H^{-\upalpha}(\R^3)$, with $\upalpha \geq 1/2$ and $\kappa=\min\{2,\gamma\}$. Finally, the sixth statement of Lemma~\ref{lem:cv-rho-R} is obtained from the fifth assertion and weak compactness properties. This ends the proof of Lemma~\ref{lem:cv-rho-R}.
\end{proof}  

%%%%%%%%%%%%%%%%%%%%%%%%

\subsection{Compactness of $B^\varepsilon $ }
Here, we aim at proving the following lemma.
\begin{lemma}
  \label{lem:cv-Bperp-R}
  The sequence ${B^\varepsilon}$ satisfies the following properties.
  \begin{align*}
    & B^\varepsilon \xrightharpoonup{\quad} B \ \ \mbox{ \rm in } \  L_{\rm loc}^2(\R_+;L^6\cap H^1(\R^3))\!-\!\mbox{\rm weak} \cap  L_{\rm loc}^\infty(\R_+ ;L^2(\R^3))\!-\!\mbox{\rm weak--}\ast\,,\\
    & \nabla_\varepsilon \cdot B^\varepsilon   \xrightharpoonup{\quad}\nabla_\perp \cdot B_\perp =0 \ \ \mbox{ \rm in } \  L_{\rm loc}^2(\R_+; L^2(\R^3))\!-\!\mbox{\rm weak}\,,\\
    &  B_\perp^\varepsilon \xrightarrow{\quad} B_\perp \ \ \mbox{ \rm in } \  L_{\rm loc}^r(\R_+; L_{\rm loc}^2(\R^3))\!-\!\mbox{\rm strong}\,, \quad 1\leq r <\infty\,. 
  \end{align*}
\end{lemma}

\begin{proof}
The proof of Lemma~\ref{lem:cv-Bperp-R} is similar to the proof of Lemma~\ref{lem:cv-Bperp}.
\end{proof}
  
  %%%%%%%%%%%%%%%%%%%%%%

\subsection{Compactness of $v^\varepsilon $ and $\rho^\varepsilon v^\varepsilon $}
Here, we aim at proving the following Lemma.

\begin{lemma}
  \label{lem:cv-vperp-R} Assume $\gamma > 3/2$. Let $\mathfrak s:=\max\{1/2,3/\gamma-1\} \in [1/2,1)$. The sequences ${v^\varepsilon}$ and $\rho^\varepsilon v^\varepsilon $ satisfy the following properties.
  \begin{align*}
    & v^\varepsilon \xrightharpoonup{\quad} v \ \ \mbox{ \rm in } \  L_{\rm loc}^2(\R_+; L^6\cap H^1(\R^3))\!-\!\mbox{\rm weak}\,,\\
    & \nabla_\varepsilon \cdot v^\varepsilon   \xrightharpoonup{\quad}\nabla_\perp \cdot v_\perp =0 \ \ \mbox{ \rm in } \  L_{\rm loc}^2(\R_+; L^2(\R^3))\!-\!\mbox{\rm weak}\,,\\
    & \Q_\perp v_\perp^\varepsilon \xrightharpoonup{\quad} 0 \ \ \mbox{ \rm in } \  L_{\rm loc}^2(\R_+; L^6\cap H^{1}(\R^3))\!-\!\mbox{\rm weak}\,, \\
    & \P_\perp v_\perp^\varepsilon \xrightarrow{\quad} \P_\perp v_\perp=  v_\perp \ \ \mbox{ \rm in } \  L_{\rm loc}^2(\R_+; L_{\rm loc}^p\cap H_{\rm loc}^{s}(\R^3))\!-\!\mbox{\rm strong}\,, \quad 1\leq p <6\,, \quad 0\leq s <1\,, \\
    & \rho^\varepsilon v^\varepsilon  \xrightharpoonup {\quad} v
    \ \ \mbox{ \rm in } \  L_{\rm loc}^2(\R_+; L_{\rm loc}^q \cap H^{-\upsigma}(\R^3))\!-\!\mbox{\rm weak}\,, \quad \forall\, \upsigma \geq \mathfrak s\,, \quad  q={6\gamma}/({6+\gamma})\,,\\
    & \rho^\varepsilon v^\varepsilon-v^\varepsilon  \xrightarrow{\quad} 0
    \ \ \mbox{ \rm in } \  L_{\rm loc}^2(\R_+; L_{\rm loc}^q(\R^3))\!-\!\mbox{\rm strong}\,, \quad   q={6\gamma}/({6+\gamma})\,, \\
    & \P_\perp (\rho^\varepsilon v_\perp^\varepsilon)  \xrightarrow{\quad} \P_\perp v_\perp=  v_\perp
    \ \ \mbox{ \rm in } \  L_{\rm loc}^2(\R_+; L_{\rm loc}^q(\R^3))\!-\!\mbox{\rm strong}\,, \quad   q={6\gamma}/({6+\gamma})\,.
  \end{align*}
\end{lemma}

\begin{proof}
  We start with the first statement of Lemma~\ref{lem:cv-vperp-R}. The proof of the first statement of  Lemma~\ref{lem:cv-vperp-R} is similar to the one of the first statement of  Lemma~\ref{lem:cv-vperp} except for the bound $ L_{\rm loc}^2(\R_+; L^2(\R^3))$. We consider the decomposition $v^\varepsilon=v_1^\varepsilon + v_2^\varepsilon$, where $v_1^\varepsilon := v^\varepsilon \mathbbm{1}_{\{|\rho^\varepsilon -1| \leq \delta \}}$, and $v_2^\varepsilon := v^\varepsilon \mathbbm{1}_{\{|\rho^\varepsilon -1| > \delta \}}$. For $v_1^\varepsilon$, since $|\rho^\varepsilon -1| \leq \delta$, we obtain $1-\delta \leq |\rho^\varepsilon|$ and thus
\begin{equation}
\|v_1^\varepsilon\|_{L_{\rm loc}^\infty(\R^+;L^2(\R^3))} \leq (1-\delta)^{-1/2}\|\rho^\varepsilon |v^\varepsilon|^2\|_{L_{\rm loc}^\infty(\R^+;L^1(\R^3))}^{1/2}
\leq \sqrt{C_0/(1-\delta)}\,,
\label{eqn:HGN-1}
\end{equation}
where we have used energy inequality \eqref{nrj-ineq-mhdeps}-\eqref{dissip-D} with the pressure term $\Pi_3$ defined by \eqref{defPI-R}.
For $v_2^\varepsilon$, using in order, H\"older inequality ($1/\gamma + 1/\gamma'=1$), estimate \eqref{eqn:bdze}, Gagliardo--Nirenberg interpolation inequality ($1/(2\gamma') = \uptheta/2 +  (1-\uptheta)(1/2-1/3) $, i.e.,  $\uptheta = 1-3/(2\gamma) \in (0,1)$, since $3/2 <\gamma <\infty$), and Young inequality ($ab \leq \uptheta a^{1/\uptheta} + (1-\uptheta) b^{1/(1-\uptheta)}$), we obtain
\begin{align}
  \int_{\R^3}dx \, |v_2^\varepsilon|^2 & \leq \delta^{-1} \int_{\R^3}dx \, |v^\varepsilon|^2
  |\rho^\varepsilon -1|\mathbbm{1}_{\{|\rho^\varepsilon -1| > \delta \}} \nonumber \\
  &\leq \delta^{-1}\|(\rho^\varepsilon -1)\mathbbm{1}_{\{|\rho^\varepsilon -1| > \delta \}} \|_{L^{\gamma}(\R^3)}\| v^\varepsilon\|_{L^{2\gamma'}(\R^3)}^2\nonumber \\
  & \leq \delta^{-1}\bigg(\frac{C_0(\gamma-1)}{\kappa_1  a}\bigg)^{1/\gamma}\varepsilon^{2/\gamma}\|v^\varepsilon\|_{L^{2\gamma'}(\R^3)}^2  \nonumber \\
  & \leq \delta^{-1}\bigg(\frac{C_0(\gamma-1)}{\kappa_1  a}\bigg)^{1/\gamma}\varepsilon^{2/\gamma} \|v^\varepsilon\|_{L^{2}(\R^3)}^{2\uptheta}
  \|\nabla v^\varepsilon\|_{L^{2}(\R^3)}^{2(1-\uptheta)} \nonumber \\
  & \leq \delta^{-1}\bigg(\frac{C_0(\gamma-1)}{\kappa_1  a}\bigg)^{1/\gamma}\varepsilon^{2/\gamma} \big( \uptheta\|v^\varepsilon\|_{L^{2}(\R^3)}^{2}
  +(1-\uptheta)\|\nabla v^\varepsilon\|_{L^{2}(\R^3)}^{2}\big)\,.
  \label{eqn:HGN-2}
\end{align}
From energy inequality \eqref{nrj-ineq-mhdeps}-\eqref{dissip-D} with the pressure term $\Pi_3$, we obtain $\nabla v^\varepsilon \in L_{\rm loc}^2(\R_+;L^2(\R^3))$,
uniformly with respect to $\varepsilon$. Using this last bound and \eqref{eqn:HGN-1}, a local time integration of estimate \eqref{eqn:HGN-2} implies that there exist two constants $K_0$ and $K_1$ (depending on $C_0$, $\gamma$, $\kappa_1$, $a$ and $\delta$) such that
\begin{equation}
\|v^\varepsilon\|_{L_{\rm loc}^2(\R^+;L^2(\R^3))}^2 \leq K_0 + K_1\varepsilon^{2/\gamma}  \|v^\varepsilon\|_{L_{\rm loc}^2(\R^+;L^2(\R^3))}^2\,. 
\label{eqn:HGN-3}
\end{equation}
For $\varepsilon$ small enough, this last inequality implies $v^\varepsilon \in L_{\rm loc}^2(\R_+;L^2(\R^3))$ and thus $v^\varepsilon \in L_{\rm loc}^2(\R_+;L^6\cap H^1(\R^3))$. We note that estimates \eqref{eqn:HGN-2}-\eqref{eqn:HGN-3} also imply $\varepsilon^{-\upeta} v_2^\varepsilon \in L_{\rm loc}^2(\R_+;L^2(\R^3))$, uniformly with respect to $\varepsilon$, for  $0\leq \upeta \leq 2/3$ (since $\gamma >3/2$). We continue with the other statements of Lemma~\ref{lem:cv-vperp-R}. Using Lemma~\ref{lem:cv-rho-R} and energy inequality \eqref{nrj-ineq-mhdeps}-\eqref{dissip-D} with the pressure term $\Pi_3$ defined by \eqref{defPI-R}, the proof of statements two to seven of Lemma~\ref{lem:cv-vperp-R} is similar to the proof of their counterparts of Lemma~\ref{lem:cv-vperp} for periodic domains.
\end{proof}

We continue with an auxiliary lemma, which will be useful to pass to the limit in the term $ \rho^\varepsilon v_\perp^\varepsilon \otimes v_\perp^\varepsilon$. Contrary to the periodic case, we are not able to get the strong convergence \eqref{eqn:GPR} for $U^\varepsilon$. Indeed, we can still prove locally  strong compactness in space-time for $\, \mathcal U^\varepsilon$ (using an Aubin--Lions theorem and estimates \eqref{eqlem:AUXR3}-\eqref{eqlem:AUXR4} below), but returning to $U^\varepsilon$, via the group of isometry  $\mathcal S$, such local strong compactness seems unaccessible since $ \mathcal S$ is non local. To recover some strong compactness for $U^\varepsilon$ we consider a truncated version of $U^\varepsilon$ and use the uniform integrability in space of such truncated sequences as well as the energy inequality and the strong convergence of $\rho^\varepsilon$.
\begin{lemma}
  \label{lem:auxilR}
  Assume $\gamma >3/2$. Let us define  $U^\varepsilon:={}^t(\phi^\varepsilon, {}^t\Phi^\varepsilon)$, where $\phi^\varepsilon := b \varrho^\varepsilon + B_\myparallel^\varepsilon,\,$ and $\ \Phi^\varepsilon:=\Q_\perp(\rho^\varepsilon v_\perp^\varepsilon),\,$  with $\ \varrho^\varepsilon:= (\rho^\varepsilon-1)/\varepsilon\, $. We also define $\, \mathcal U^\varepsilon := \mathcal S (-t/\varepsilon)  U^\varepsilon$, where the group of isometry $\{ \mathcal S(\tau)\,; \tau \in \R \}$ is the same as the one defined in Section~\ref{ss:cv-rvperp-P} with now $\Omega=\R^3$, except that we substitute $c$ for  $c^\varepsilon$. Finally, we recall that $\kappa = \min\{2,\gamma\}$, and we set $\upkappa= \max\{1/2,3/(2\gamma)\} \in [1/2,1)$ and $\mathfrak s:=\max\{1/2,3/\gamma -1 \} \in [1/2,1)$. Then,
\begin{itemize}
  \item[1.] We have the following uniform (with respect to $\varepsilon$) bounds
    \begin{align}
      & \phi^\varepsilon \in  L_{\rm loc}^\infty(\R_+; H^{-\upalpha}(\R^3))\,, \quad \upalpha \geq \tfrac12\,, \label{eqlem:AUXR1} \\
      & \rho^\varepsilon v^\varepsilon  \in  L_{\rm loc}^\infty\big(\R_+; \big(L^2+L^{2\gamma/(\gamma+1)}\big)\cap L_{\rm loc}^{2\gamma/(\gamma +1)} \cap H^{-\upbeta}(\R^3)\big)\,,
      \quad \upbeta \geq \tfrac{3}{2\gamma}\,.   \label{eqlem:AUXR2} \\
       & \varrho^\varepsilon v^\varepsilon  \in  L_{\rm loc}^2\big(\R_+; \big(L^{3/2}+L^{6\kappa/(6+\kappa)}\big)\cap L_{\rm loc}^{6\kappa/(6+\kappa)} \cap H^{-\varkappa}(\R^3)\big)\,,
      \quad \varkappa \geq \mathfrak s \,.   \label{eqlem:AUXR2bis}
    \end{align}
  \item[2.] There exists a constant $C$, independent of $\varepsilon$, such that
  \begin{align}
    & \| \, \mathcal U^\varepsilon\|_{L_{\rm loc}^\infty(\R_+; H^{-\sigma}(\R^3))} \leq C\,, \quad  \sigma \geq \upkappa\,, \label{eqlem:AUXR3} \\
      &  \| \partial_t\, \mathcal U^\varepsilon\|_{L_{\rm loc}^2(\R_+; H^{-r}(\R^3))} \leq C\,, \quad r> \big(\tfrac{5}{2}\big)^+\,,   \label{eqlem:AUXR4} 
  \end{align}
\item[3.] There exists a function $\,\widetilde{\mathcal U}^\varepsilon:={}^t(\widetilde{\psi}^\varepsilon, {}^t\widetilde{\Psi}^\varepsilon) \in L_{\rm loc}^\infty(\R_+;L^2(\R^3;\R^3))$,
  a constant  $C$, independent of $\varepsilon$, and a function $\upomega : \R_+ \mapsto \R_+$, continuous in the neighborhood of zero, with $\, \upomega(\varepsilon) \rightarrow 0$, as $\varepsilon \rightarrow 0_+$, such that,  for $\sigma \geq \upkappa$,  
 \begin{align}
   &  \| \,\widetilde{\mathcal U}^\varepsilon\|_{L_{\rm loc}^\infty(\R_+;L^2(\R^3)} \leq C\,, \label{eqlem:AUXR5} \\
   & \| \, \mathcal U^\varepsilon - \,\widetilde{\mathcal U}^\varepsilon\|_{L_{\rm loc}^\infty(\R_+;H^{-\sigma}(\R^3))} \lesssim \upomega(\varepsilon)\,,   \label{eqlem:AUXR6} \\
   &  \| U^\varepsilon - \mathcal S (t/\varepsilon)\,\widetilde{\mathcal U}^\varepsilon \|_{L_{\rm loc}^\infty(\R_+;H^{-\sigma}(\R^3))} \lesssim \upomega(\varepsilon)\,,  \label{eqlem:AUXR7}\\
   & \| \Q_\perp v_\perp^\varepsilon - \mathcal S_2 (t/\varepsilon)\,\widetilde{\mathcal U}^\varepsilon \|_{L_{\rm loc}^2(\R_+;H^{-\sigma}(\R^3))} \lesssim \upomega(\varepsilon) + \varepsilon\,.  \label{eqlem:AUXR8}
  \end{align}
\end{itemize}
\end{lemma}

\begin{proof}
  We start with the point 1 of Lemma~\ref{lem:auxilR}, beginning with \eqref{eqlem:AUXR1}. Using the fifth statement of Lemma~\ref{lem:cv-rho-R} for an estimate of $\varrho^\varepsilon$,  the first statement of Lemma~\ref{lem:cv-Bperp-R} for an estimate of $B_\myparallel^\varepsilon $, the embedding $(L^\kappa + L^2)(\R^2) \hookrightarrow H^{-\upalpha}$, with $\alpha \geq 1/2$ and $\kappa >3/2$, and the definition $\phi^\varepsilon := b \varrho^\varepsilon + B_\myparallel^\varepsilon $,  we obtain \eqref{eqlem:AUXR1}. We continue with \eqref{eqlem:AUXR2}, by recasting $\rho^\varepsilon v^\varepsilon$ as
\begin{equation}
 \label{eqn:AUXR1} 
 \rho^\varepsilon v^\varepsilon =(\sqrt{\rho^\varepsilon} v^\varepsilon) \sqrt{\rho^\varepsilon}
 \mathbbm{1}_{\{|\rho^\varepsilon-1|\leq \delta\}} + (\sqrt{\rho^\varepsilon} v^\varepsilon)
 \frac{\sqrt{\rho^\varepsilon}}{\sqrt{|\rho^\varepsilon-1|}}\sqrt{|\rho^\varepsilon-1|}
 \mathbbm{1}_{\{|\rho^\varepsilon-1|> \delta\}}. 
\end{equation}
The first term of the right-hand side of \eqref{eqn:AUXR1} is the product of the function $\sqrt{\rho^\varepsilon} v^\varepsilon \in L_{\rm loc}^\infty(\R_+;L^2(\R^3))$  and the  function $\rho^\varepsilon\mathbbm{1}_{\{|\rho^\varepsilon-1|\leq \delta\}} \in L_{\rm loc}^\infty(\R_+;L^\infty(\R^3))$, thus using H\"older inequality this product is in $ L_{\rm loc}^\infty(\R_+;L^2(\R^3))$. The second term of the right-hand side of \eqref{eqn:AUXR1} is the triple product of the function $\sqrt{\rho^\varepsilon }v^\varepsilon \in L_{\rm loc}^\infty(\R_+;L^2(\R^3))$,  the  function $ \sqrt{\rho^\varepsilon}/\sqrt{|\rho^\varepsilon-1|}\mathbbm{1}_{\{|\rho^\varepsilon-1|> \delta\}}  \in L_{\rm loc}^\infty(\R_+;L^\infty(\R^3))$, and the function $\sqrt{|\rho^\varepsilon-1|} \mathbbm{1}_{\{|\rho^\varepsilon-1|> \delta\}} \in  L_{\rm loc}^\infty(\R_+;L^{2\gamma}(\R^3))$, thus, using H\"older inequality, this triple product is in $L_{\rm loc}^\infty(\R_+;L^{2\gamma/(\gamma +1)}(\R^3))$. Using the Sobolev embeddings $\dot H^\upbeta(\R^3) \hookrightarrow \{L^{2\gamma/(\gamma-1)}(\R^3)\,,\ L^2(\R^3)\}$ (with $\gamma>3/2$), for $\upbeta\geq 3/(2\gamma)$, by duality we obtain $\{ L^2(\R^3)\,,  L^{2\gamma/(\gamma+1)}(\R^3) \} \hookrightarrow H^{-\upbeta}(\R^3)$; hence \eqref{eqlem:AUXR2}. We continue with the proof of  \eqref{eqlem:AUXR2bis}. We first consider the decomposition $\varrho^\varepsilon v^\varepsilon=\varrho^\varepsilon v_1^\varepsilon +\varrho^\varepsilon v_2^\varepsilon$, with $ v_1^\varepsilon$ and  $v_2^\varepsilon$ defined as in the proof of Lemma~\ref{lem:cv-vperp-R}. From  $\varrho^\varepsilon \mathbbm{1}_{\{|\rho^\varepsilon-1|\leq \delta\}} \in L_{\rm loc}^\infty(\R_+;L^2(\R^3))$ (fifth statement of Lemma~\ref{lem:cv-rho-R}) and $v^\varepsilon \in L_{\rm loc}^2(\R_+;L^6(\R^3))$, using H\"older inequality, we obtain $\varrho^\varepsilon v_1^\varepsilon \in L_{\rm loc}^2(\R_+;L^{3/2}(\R^3))$. From  $\varrho^\varepsilon \mathbbm{1}_{\{|\rho^\varepsilon-1|> \delta\}} \in L_{\rm loc}^\infty(\R_+;L^\kappa(\R^3))$ (fifth statement of Lemma~\ref{lem:cv-rho-R}) and $v^\varepsilon \in L_{\rm loc}^2(\R_+;L^{6}(\R^3))$, using H\"older inequality, we obtain $\varrho^\varepsilon v_2^\varepsilon \in L_{\rm loc}^2(\R_+;L^{\mathfrak q}(\R^3))$, with $\mathfrak q=6\kappa/(6+\kappa)$. Moreover, using the first statement of Lemma~\ref{lem:cv-rho-R} and H\"older inequality, we obtain $\varrho^\varepsilon v^\varepsilon \in L_{\rm loc}^2(\R_+;L_{\rm loc}^{\mathfrak q}(\R^3))$. It remains to show the $H^{-\varkappa}$-bound in \eqref{eqlem:AUXR2bis}. Since $\varrho^\varepsilon v_1^\varepsilon \in L_{\rm loc}^2(\R_+;L^{3/2}(\R^3))$, using the Sobolev embedding $L^{3/2}(\R^3)\hookrightarrow H^{-\upalpha}(\R^3)$, with $\upalpha\geq 1/2$, we obtain $\varrho^\varepsilon v_1^\varepsilon \in L_{\rm loc}^2(\R_+;H^{-\upalpha}(\R^3))$. Since $\varrho^\varepsilon v_2^\varepsilon \in L_{\rm loc}^2(\R_+;L^{\mathfrak q}(\R^3))$, using the Sobolev embedding $L^{\mathfrak q}(\R^3)\hookrightarrow H^{- \tilde{\mathfrak s}}$ with $\tilde{\mathfrak s} \geq (3/\kappa)-1 $ while $ (3/\kappa)-1 \leq \mathfrak s $, we get $\varrho^\varepsilon v_2^\varepsilon \in L_{\rm loc}^2(\R_+;H^{-\tilde{\mathfrak s}}(\R^3))$. Combining this two last results, we obtain  $\varrho^\varepsilon v^\varepsilon \in L_{\rm loc}^2(\R_+;H^{-\varkappa}(\R^3))$, with $\varkappa \geq \mathfrak s$; hence \eqref{eqlem:AUXR2bis}.\\
We continue with the point 2 of Lemma~\ref{lem:auxilR}, starting with \eqref{eqlem:AUXR3}. Using \eqref{eqlem:AUXR1}-\eqref{eqlem:AUXR2} and since the group $\mathcal S $ (resp. the operator $\Q_\perp$) is an isometry (resp. continuous) in $H^\alpha(\R^3)$ with $\alpha\in \R$, we obtain \eqref{eqlem:AUXR3}. For \eqref{eqlem:AUXR4}, observe that $\partial_t \, \mathcal U^\varepsilon = \mathcal S(-t/\varepsilon) F^\varepsilon$, with $F^\varepsilon$ defined by
\begin{equation}
    \label{eqn:Feps-R}
    F^\varepsilon:=
    \left (\begin{aligned}
     F_1^\varepsilon\\ F_2^\varepsilon
  \end{aligned}
    \right) =
    \left\{
  \begin{aligned}
    F_1^\varepsilon =& -b\partial_\myparallel (\rho^\varepsilon v_\myparallel^\varepsilon) +
    \nabla_\perp \cdot \Q_\perp(\varrho^\varepsilon v_\perp^\varepsilon) -B_\myparallel^\varepsilon  \nabla_\varepsilon \cdot v^\varepsilon \\
      & - (v^\varepsilon\cdot\nabla_\varepsilon)B_\myparallel^\varepsilon 
    +(B^\varepsilon\cdot\nabla_\varepsilon)v_\myparallel^\varepsilon + (\eta_\perp^\varepsilon\Delta_\perp+\, \eta_\myparallel^\varepsilon\Delta_\myparallel)B_\myparallel^\varepsilon\,,\\
     F_2^\varepsilon =& -\Q_\perp\nabla_\varepsilon \cdot (\rho^\varepsilon v_\perp^\varepsilon \otimes v^\varepsilon)
    -(\gamma-1)\nabla_\perp\Pi_3(\rho^\varepsilon) -\tfrac12 \nabla_\perp(|B^\varepsilon|^2) + \partial_\myparallel \Q_\perp B_\perp^\varepsilon \\
    & + \Q_\perp\nabla_\varepsilon \cdot (B_\perp^\varepsilon \otimes B^\varepsilon) + \mu_\perp^\varepsilon \nabla_\perp(\nabla_\perp \cdot v_\perp^\varepsilon)
    + \mu_\myparallel^\varepsilon \Delta_\myparallel \Q_\perp v_\perp^\varepsilon  +\lambda^\varepsilon\nabla_\perp(\nabla_\varepsilon \cdot v^\varepsilon)\,.
  \end{aligned}
  \right.
  \end{equation}
Using estimates of point 1 in Lemma~\ref{lem:auxilR} and following Step 1 in the proof  Lemma~\ref{lem:UG}, we obtain from \eqref{eqn:Feps-R} and $\upkappa > \mathfrak s$ ($\gamma >3/2$), $F_1^\varepsilon \in  L_{\rm loc}^{2}(\R_+; (H^{-\upkappa-1}+ W^{-1,3/2} + H^{-1})(\R^3)) \hookrightarrow  L_{\rm loc}^{2}(\R_+; H^{-\upkappa-1}(\R^3))$, and  $F_2^\varepsilon \in  L_{\rm loc}^{2}(\R_+; (W^{-1-\delta,1}+W^{-1,1}+H^{-1}+L^2 )(\R^3))  \hookrightarrow  L_{\rm loc}^{2}(\R_+; H^{-r}(\R^3)) $, with $r>5/2+\delta $, for any $\delta >0$. Therefore, using the isometry $\mathcal S $ in $H^\alpha(\R^3)$ for any $ \alpha\in \R$, we obtain \eqref{eqlem:AUXR4}.\\
We continue with the point 3 of Lemma~\ref{lem:auxilR}, starting with \eqref{eqlem:AUXR5}. For any $\delta >0$, we define $\widetilde{\psi}^\varepsilon := \mathcal S_1(-t/\varepsilon)[\phi^\varepsilon \mathbbm{1}_{\{\rho^\varepsilon\leq 1+\delta\}}]$ and $\widetilde{\Psi}^\varepsilon := \mathcal S_2(-t/\varepsilon)\Q_\perp (\rho^\varepsilon v_\perp^\varepsilon \mathbbm{1}_{\{\rho^\varepsilon\leq 1+\delta\}})$. Clearly, from the uniform bounds above and the isometry $\mathcal S$ in  $H^\alpha(\R^3)$ for any $ \alpha\in \R$, we obtain $ \widetilde{\psi}^\varepsilon$, $\widetilde{\Psi}^\varepsilon\in   L_{\rm loc}^\infty(\R_+;L^2(\R^3))$. We continue with the proof of \eqref{eqlem:AUXR6}. Since $\varrho^\varepsilon \in L_{\rm loc }^\infty(\R_+; L_2^\kappa(\R^3))$ and $ B_\myparallel^\varepsilon \in L_{\rm loc }^\infty(\R_+; L^2(\R^3))$, using the De la Vall\'ee Poussin criterion, we obtain that $\phi^\varepsilon$ is spatially uniformly integrable in $L^{3/2}(\R^3)$, uniformly in time on any compact time interval. Then, from the fourth statement of Lemma~\ref{lem:cv-rho-R}, we obtain $\| \phi^\varepsilon\mathbbm{1}_{\{\rho^\varepsilon>1+\delta\}}\|_{L_{\rm loc }^\infty(\R_+; L^{3/2}(\R^3))} \rightarrow 0$, as $\varepsilon \rightarrow 0$. Therefore, using the isometry $\mathcal S$, we obtain the first part of \eqref{eqlem:AUXR6}, that is $\| \psi^\varepsilon - \widetilde{\psi }^\varepsilon\|_{L_{\rm loc}^\infty(\R_+;H^{-\sigma}(\R^3))} \rightarrow 0$, as $\varepsilon \rightarrow 0$. Since $\rho^\varepsilon v_\perp^\varepsilon \mathbbm{1}_{\{\rho^\varepsilon> 1+\delta\}} = \sqrt{\rho^\varepsilon} v_\perp^\varepsilon  \sqrt{\rho^\varepsilon} \mathbbm{1}_{\{\rho^\varepsilon> 1+\delta\}}$, using $\rho^\varepsilon \rightarrow 1$ in $L_{\rm loc}^\infty(\R_+; L_2^\gamma(\R^3))\!-\!\mbox{\rm strong}$ (fourth statement of Lemma~\ref{lem:cv-rho-R}), the uniform bound $\|\rho^\varepsilon |v^\varepsilon|^2 \|_{L_{\rm loc }^\infty(\R_+; L^1(\R^3))} \leq C<\infty$, and H\"older inequality, we obtain $\rho^\varepsilon v_\perp^\varepsilon \mathbbm{1}_{\{\rho^\varepsilon > 1+\delta\}} \rightarrow 0$ in $L_{\rm loc }^\infty(\R_+; L^{2\gamma/(\gamma+1)}(\R^3))$, as $\varepsilon \rightarrow 0$. Since the group $\mathcal S $ (resp. the operator $\Q_\perp$) is an isometry (resp. continuous) in $H^\alpha(\R^3)$ for any $\alpha\in \R$, using the embedding $ L^{2\gamma/(\gamma+1)}(\R^3) \hookrightarrow H^{-\sigma}(\R^3)$ (see above in the proof), we obtain the second part of \eqref{eqlem:AUXR6}, that is $\| \Psi^\varepsilon - \widetilde{\Psi }^\varepsilon\|_{L_{\rm loc}^\infty(\R_+;H^{-\sigma}(\R^3))} \rightarrow 0$, as $\varepsilon \rightarrow 0$. Still using the isometry  $\mathcal S $, we deduce  estimate \eqref{eqlem:AUXR7} from \eqref{eqlem:AUXR6}. It remains to prove \eqref{eqlem:AUXR8}. Using the continuity of $\Q_\perp$ and \eqref{eqlem:AUXR7}, we obtain $\| \Q_\perp v_\perp^\varepsilon - \mathcal S_2 (t/\varepsilon)\,\widetilde{\mathcal U}^\varepsilon \|_{L_{\rm loc}^2(\R_+;H^{-\sigma}(\R^3))} \leq \| (\rho^\varepsilon -1) v_\perp^\varepsilon \|_{L_{\rm loc}^2(\R_+;H^{-\sigma}(\R^3))} + \upomega(\varepsilon)$. We split $(\rho^\varepsilon -1) v_\perp^\varepsilon$ into the part $d_1^\varepsilon:=(\rho^\varepsilon -1)v_\perp^\varepsilon \mathbbm{1}_{\{|{\rho^\varepsilon-1}|\leq \delta\}}$ and the part  $d_2^\varepsilon:=(\rho^\varepsilon -1)v_\perp^\varepsilon \mathbbm{1}_{\{|{\rho^\varepsilon-1}|> \delta\}}$. Since $L^{3/2}(\R^3) \hookrightarrow H^{-1/2}(\R^3)$, using H\"older inequality and estimate \eqref{eqn:bdze}, we obtain $\| d_1^\varepsilon\|_{L_{\rm loc}^2(\R_+;H^{-\sigma}(\R^3))} \leq \| d_1^\varepsilon\|_{L_{\rm loc}^2(\R_+;H^{-1/2}(\R^3))} \leq \| (\rho^\varepsilon -1)  \mathbbm{1}_{\{|\rho^\varepsilon-1| \leq \delta\}}\|_{L_{\rm loc}^\infty(\R_+;L^2(\R^3))}\|  v_\perp^\varepsilon \|_{L_{\rm loc}^2(\R_+;L^6(\R^3))} \lesssim \varepsilon$. For the part $d_2^\varepsilon$, we distinguish two cases according to the value of $\gamma$. For $\gamma \geq 2$, following the same proof as for $d_1^\varepsilon$, we obtain $\| d_2^\varepsilon\|_{L_{\rm loc}^2(\R_+;H^{-\sigma}(\R^3))} \lesssim \varepsilon$. For $3/2<\gamma < 2$, we have $\upkappa=3/(2\gamma) > \mathfrak s =3/\gamma-1$ (since $\gamma >3/2$). Then, using the Sobolev embeddings $ L^{6\gamma/(6+\gamma)}(\R^3)  \hookrightarrow H^{-\mathfrak s}(\R^3)$, and estimate \eqref{eqn:bdze}, we obtain $\| d_2^\varepsilon\|_{L_{\rm loc}^2(\R_+;H^{-\sigma}(\R^3))} \leq \| d_2^\varepsilon\|_{L_{\rm loc}^2(\R_+;H^{-\upkappa}(\R^3))} \leq \| d_2^\varepsilon\|_{L_{\rm loc}^2(\R_+;H^{-\mathfrak s}(\R^3))} \leq \| (\rho^\varepsilon -1)  \mathbbm{1}_{\{|\rho^\varepsilon-1| > \delta\}}\|_{L_{\rm loc}^\infty(\R_+;L^\gamma(\R^3))} \|  v_\perp^\varepsilon \|_{L_{\rm loc}^2(\R_+;L^6(\R^3))} \lesssim \varepsilon^{2/\gamma} \leq \varepsilon$. This ends the proof of Lemma~\ref{lem:auxilR}.
\end{proof}

%%%%%%%%%%%%%%%%%%

\subsection{Passage to the limit in the compressible MHD equations}
Here, we justify the passage to the limit in the weak formulation \eqref{initial-data-ws}-\eqref{eq:zero-divB} of the MHD equations for the whole space. Using Lemmas~\ref{lem:cv-Bperp-R}~and~\ref{lem:cv-vperp-R}, the passage to the limit in equation \eqref{eq:Bperp-eps} for $B_\perp^\varepsilon$ follows the same proof as the one of the periodic case described in Section~\ref{ss:cv-Bperp-P}. Using Lemmas~\ref{lem:cv-rho-R},~\ref{lem:cv-Bperp-R}~and~\ref{lem:cv-vperp-R}, and following the sames lines as the ones of Section~\ref{ss:cv-rvperp-P} for the periodic case, we can pass to the limit in almost all terms of equation \eqref{eq:vperp-eps} for $\rho^\varepsilon v_\perp^\varepsilon$. Indeed, the only difference with the proof of Section~\ref{ss:cv-rvperp-P} is the justification of the limit $\varepsilon\rightarrow 0$ for the term $ \rho^\varepsilon v_\perp^\varepsilon \otimes v_\perp^\varepsilon$ that we detail below in Lemma~\ref{lem:cv-rvpvp-R}. In particular, the pressure term $\mathfrak p^\varepsilon = p^\varepsilon /\varepsilon^2 + B_\myparallel^\varepsilon/\varepsilon + |B^\varepsilon|^2/2$, rewritten as $\mathfrak p^\varepsilon =\phi^\varepsilon/\varepsilon +  a /\varepsilon^2+ \uppi_2^\varepsilon\,$, with $\, \uppi_2^\varepsilon = (\gamma-1)\Pi_3(\rho^\varepsilon) + |B^\varepsilon|^2/2$, converges weakly to  $\uppi_1 + \uppi_2$ in  $\mathcal D'(\R_+^*\times \R^3)$. More precisely,  $\phi^\varepsilon/\varepsilon \rightharpoonup  \delta_0(t) \otimes\uppi_0  +\uppi_1$ in $H^{-1}(\R_+, H^{-r}(\R^3))$--weak, with $r>(5/2)^+$, and $\uppi_2^\varepsilon \rightharpoonup \uppi_2$ in $ L_{\rm loc}^\infty(\R_+;L^1(\R^3))$--weak$-*$, while the constant term $ba /\varepsilon^2$  disappears by spatial integration in \eqref{eq:vperp-eps}. As in the periodic case, the term $\delta_0(t)\otimes \uppi_0 $ cancels the irrotational part  $\Q_\perp u_{0 \perp }$ of the limit term $u_{0 \perp }$, so that  the limit initial condition is $\P_\perp u_{0 \perp }=\P_\perp v_{0 \perp }$. We also obtain $\phi^\varepsilon \rightharpoonup \phi=b \varrho + B_\myparallel =0$ in $H^{-1}(\R_+, H^{-r}(\R^3))$--weak, and the relation $b \varrho + B_\myparallel =0$ holds for a.e. $(t,x) \in ]0,+\infty[\times \R^3$, since $b \varrho + B_\myparallel \in L_{\rm loc}^\infty(\R_+;(L_{\rm loc}^\kappa+L^2)(\R^3))$. Using Lemmas~\ref{lem:cv-rho-R},~\ref{lem:cv-Bperp-R}~and~\ref{lem:cv-vperp-R}, the passage to limit in equations \eqref{eq:vparal-eps} and  \eqref{eq:q-eps}, for respectively  $\rho^\varepsilon v_\myparallel^\varepsilon$ and $\mathbbmsl{B}^\varepsilon$, is justified in a similar way as the one described in Sections~\ref{ss:cv-rvpara-P}~and~\ref{ss:cv-Bpara-P} for the periodic case. Finally, to conclude the proof in the case of the whole space, we use the following lemma, which is the counterpart of Lemma~\ref{lem:RST} with a different proof since we do not have the strong convergence \eqref{eqn:GPR}.

\begin{lemma}
\label{lem:cv-rvpvp-R} There exists a distribution $\uppi_3 \in \mathcal D'(\R_+^* \times \R^3)$, such that
  \begin{equation*}
    \nabla_\perp \cdot (\rho^\varepsilon v_\perp^\varepsilon \otimes v_\perp^\varepsilon) \xrightharpoonup{\quad}
    \nabla_\perp \cdot ( v_\perp \otimes v_\perp)  \ + \ \nabla_\perp \uppi_3 \ \ \mbox{ in } \
    \mathcal{D}'(\R_+^*\times \R^3)\,.
   \end{equation*}
\end{lemma}

\begin{proof}
  First, we observe the following decomposition
  \begin{equation}
    \label{eqn:RTT-R:1}
    \nabla_\perp \cdot (\rho^\varepsilon v_\perp^\varepsilon \otimes v_\perp^\varepsilon)
    = \nabla_\perp \cdot (\rho^\varepsilon v_\perp^\varepsilon \otimes \P_\perp v_\perp^\varepsilon)
    +  \nabla_\perp \cdot (\P_\perp(\rho^\varepsilon v_\perp^\varepsilon) \otimes \Q_\perp v_\perp^\varepsilon)
    + \nabla_\perp \cdot (\Q_\perp(\rho^\varepsilon v_\perp^\varepsilon) \otimes \Q_\perp v_\perp^\varepsilon)\,.
  \end{equation}
  Using the fourth and fifth statements of  Lemma~\ref{lem:cv-vperp-R},  for the first term of the right-hand side of \eqref{eqn:RTT-R:1}, we obtain $\nabla_\perp \cdot (\rho^\varepsilon v_\perp^\varepsilon \otimes v_\perp^\varepsilon) \rightharpoonup \nabla_\perp \cdot ( v_\perp \otimes v_\perp)$ in $\mathcal D'(\R_+^* \times \R^3)$. Using the third and seventh statements of  Lemma~\ref{lem:cv-vperp-R}, for the second term of the right-hand side of \eqref{eqn:RTT-R:1}, we obtain $\nabla_\perp \cdot (\P_\perp(\rho^\varepsilon v_\perp^\varepsilon) \otimes \Q_\perp v_\perp^\varepsilon) \rightharpoonup 0 $ in $\mathcal D'(\R_+^* \times \R^3)$. It remains to show that we have $\nabla_\perp \cdot (\Q_\perp(\rho^\varepsilon v_\perp^\varepsilon) \otimes \Q_\perp v_\perp^\varepsilon) \rightharpoonup \nabla_\perp \pi_3 $ in  $\mathcal D'(\R_+^* \times \R^3)$, or equivalently, that, for any $\psi_\perp \in \mathscr{C}_c^\infty(\R_+\times \R^3)$ such that $\nabla_\perp\cdot  \psi_\perp =0$, we have
  \begin{equation}
    \label{eqn:RTT-R:2}
    \lim_{\varepsilon \rightarrow 0} \Gamma^\varepsilon := \lim_{\varepsilon \rightarrow 0}  \int_{\R_+} dt \int_{\R^3}dx\, {}^t(\Q_\perp(\rho^\varepsilon v_\perp^\varepsilon))
        [D_\perp \psi_\perp] \Q_\perp v_\perp^\varepsilon =0\,.
  \end{equation}
  Since $\Q_\perp v_\perp^\varepsilon$ is bounded in $L_{\rm loc}^2(\R_+;H^1(\R^3))$  uniformly with respect to $\varepsilon$, Using \eqref{eqlem:AUXR7} with $\sigma\in [\upkappa,1]$, we obtain
  \begin{equation}
    \label{eqn:RTT-R:3}
    |\Gamma^\varepsilon -\Gamma_1^\varepsilon| \lesssim \upomega(\varepsilon)\,, 
    \quad \mbox{ with } \quad
    \Gamma_1^\varepsilon:=  \int_{\R_+} dt \int_{\R^3}dx\,{}^t(\mathcal S_2(t/\varepsilon)\, \widetilde{\mathcal U}^\varepsilon) 
          [D_\perp \psi_\perp]  \Q_\perp v_\perp^\varepsilon\,.
  \end{equation}   
  Using the isometry $\mathcal S $ in $H^\alpha(\R^3)$ for any $ \alpha\in \R$, estimate \eqref{eq:MolA}, bound \eqref{eqlem:AUXR5}, and the fact that the group $\mathcal S$ and the mollification operator $\mathcal J_{3,\updelta}$ commute, we obtain, for any $\upmu \in [0,1]$ and $\updelta \in (0,1)$, 
  \begin{align}
    \| \mathcal S_2(t/\varepsilon) \mathcal J_{3,\updelta}^2 \,\widetilde{\mathcal U}^\varepsilon -
    \mathcal S_2(t/\varepsilon) \,\widetilde{\mathcal U}^\varepsilon \|_{L_{\rm loc}^\infty(\R_+;H^{-\upmu}(\R^3))} &\leq
    \| \mathcal  J_{3,\updelta} \mathcal S_2(t/\varepsilon)  \,\widetilde{\mathcal U}^\varepsilon -
    \mathcal S_2(t/\varepsilon) \,\widetilde{\mathcal U}^\varepsilon \|_{L_{\rm loc}^\infty(\R_+;H^{-\upmu}(\R^3))}  \nonumber \\
    & \hspace{-2.cm }\quad + \quad \|  \mathcal J_{3,\updelta} \mathcal S_2(t/\varepsilon) \mathcal J_{3,\updelta} \,\widetilde{\mathcal U}^\varepsilon -
    \mathcal S_2(t/\varepsilon) \mathcal J_{3,\updelta} \,\widetilde{\mathcal U}^\varepsilon \|_{L_{\rm loc}^\infty(\R_+;H^{-\upmu}(\R^3))} 
    \nonumber \\
    &\hspace{-5cm } \lesssim  \updelta^{\upmu}(1+\| \chi_\updelta\|_{L^1(\R^3)})\| \,\widetilde{\mathcal U}^\varepsilon \|_{L_{\rm loc}^\infty(\R_+;L^2(\R^3))}
    \lesssim  \updelta^{\upmu}\,. \label{eqn:RTT-R:4}
  \end{align}
  Since $\Q_\perp v_\perp^\varepsilon$ is bounded in $L_{\rm loc}^2(\R_+;H^1(\R^3))$  uniformly with respect to $\varepsilon$,  using \eqref{eqn:RTT-R:4}, we obtain, for any $\upmu \in (0,1]$,
  \begin{equation}
    \label{eqn:RTT-R:5}
    |\Gamma_1^\varepsilon- \Gamma_2^\varepsilon| \lesssim \updelta^\upmu\,, \quad \mbox{ with } \quad
    \Gamma_2^\varepsilon:=  \int_{\R_+} dt \int_{\R^3}dx\, {}^t(\mathcal S_2(t/\varepsilon) \mathcal J_{3,\updelta}^2 \,\widetilde{\mathcal U}^\varepsilon)[ D_\perp \psi_\perp]  \Q_\perp v_\perp^\varepsilon\,.
  \end{equation}
  Since for any $\updelta >0$, and any $\upnu \geq 0$, $\mathcal S_2(t/\varepsilon) \mathcal J_{3,\updelta}^2 \,\widetilde{\mathcal U}^\varepsilon \in L_{\rm loc}^\infty(\R_+;H^{\upnu}(\R^3))$, there exists a constant   $C_\updelta$  such that $\| \mathcal S_2(t/\varepsilon) \mathcal J_{3,\updelta}^2 \,\widetilde{\mathcal U}^\varepsilon\|_{ L_{\rm loc}^\infty(\R_+;H^{\sigma}(\R^3))} \leq C_\updelta$, where $C_\updelta$ explodes as $\updelta \rightarrow 0$. Then, using \eqref{eqlem:AUXR8}, we obtain
  \begin{equation}
    \label{eqn:RTT-R:7}
    |\Gamma_2^\varepsilon- \Gamma_3^\varepsilon| \lesssim C_\updelta (\upomega(\varepsilon)+\varepsilon)\,, \quad \mbox{ with } \quad
    \Gamma_3^\varepsilon:=  \int_{\R_+} dt \int_{\R^3}dx\, {}^t(\mathcal S_2(t/\varepsilon) \mathcal J_{3,\updelta}^2 \,\widetilde{\mathcal U}^\varepsilon)[ D_\perp \psi_\perp] \mathcal S_2(t/\varepsilon) \,\widetilde{\mathcal U}^\varepsilon \,.
    \end{equation}
  We now claim that
  \begin{equation}
    \label{eqn:RTT-R:8}
    |\Gamma_3^\varepsilon- \Gamma_4^\varepsilon| \lesssim \updelta\,, \quad \mbox{ with } \quad
    \Gamma_4^\varepsilon:=  \int_{\R_+} dt \int_{\R^3}dx\, {}^t(\mathcal S_2(t/\varepsilon) \mathcal J_{3,\updelta} \,\widetilde{\mathcal U}^\varepsilon)[D_\perp \psi_\perp] \mathcal S_2(t/\varepsilon) \mathcal J_{3,\updelta} \,\widetilde{\mathcal U}^\varepsilon \,.
    \end{equation}
  Indeed, using in order $[\mathcal S\, ,\, \mathcal J_{3,\updelta}]=0$ (where $[\,\cdot\, ,\, \cdot\,]$ is the commutator), $L^1\ast L^2 \subset L^2$, Cauchy--Scwharz inequality, $\| x \chi(x)\|_{L^1(\R^3)}\leq C <\infty$,  estimate \eqref{eq:MolA} (with $\sigma=s=0$), the isometry $\mathcal S $ in $H^\alpha(\R^3)$ for any $ \alpha\in \R$, and bound \eqref{eqlem:AUXR5}, we obtain
  \begin{align*}
    \Gamma_3^\varepsilon- \Gamma_4^\varepsilon &=
    \int_{\R_+} dt \int_{\R^3}dx\, {}^t(\mathcal S_2(t/\varepsilon) \mathcal J_{3,\updelta} \,\widetilde{\mathcal U}^\varepsilon) 
        [\mathcal J_{3,\updelta} \,,\, D_\perp \psi_\perp] \mathcal S_2(t/\varepsilon)\,\widetilde{\mathcal U}^\varepsilon \\
        &= \int_{\R_+} dt \int_{\R^3}dx\, {}^t(\mathcal S_2(t/\varepsilon) \mathcal J_{3,\updelta} \,\widetilde{\mathcal U}^\varepsilon)(x) 
        \int_{\R^3}dy\, \big(D_\perp \psi_\perp(y)-D_\perp \psi_\perp(x)\big)\chi_\updelta(x-y) \mathcal S_2(t/\varepsilon)\,\widetilde{\mathcal U}^\varepsilon(y)\\
        &\leq \updelta \|D_\perp^2 \psi_\perp\|_{L^\infty(\R_+\times \R^3)}
        \| \mathcal S_2(t/\varepsilon) \mathcal J_{3,\updelta} \,\widetilde{\mathcal U}^\varepsilon \|_{L_{\rm loc}^\infty(\R_+; L^2(\R^3))}
        \| |\tfrac{x}{\updelta}\chi_\updelta| \ast (\mathcal S_2(t/\varepsilon)\,\widetilde{\mathcal U}^\varepsilon) \|_{L_{\rm loc}^\infty(\R_+; L^2(\R^3))}\\
        &\lesssim   \updelta
        \| \mathcal S_2(t/\varepsilon) \,\widetilde{\mathcal U}^\varepsilon \|_{L_{\rm loc}^\infty(\R_+; L^2(\R^3))}^2
        \lesssim \updelta \| \,\widetilde{\mathcal U}^\varepsilon \|_{L_{\rm loc}^\infty(\R_+; L^2(\R^3))}^2 \lesssim \updelta \,.
  \end{align*}
  We now claim that there exists a constant $\widetilde{C}_\updelta$, which explodes as $\updelta \rightarrow 0$, such that for all $\upnu \geq 0$, 
  \begin{equation}
    \label{eqn:RTT-R:9}
    \Upgamma^\varepsilon:= \|  \mathcal J_{3,\updelta} \,\widetilde{\mathcal U}^\varepsilon (t_1) -
    \mathcal J_{3,\updelta} \,\widetilde{\mathcal U}^\varepsilon (t_2)
    \|_{H^\upnu(\R^3)} \leq \widetilde{C}_\updelta \big( |t_1-t_2| + \upomega(\varepsilon) \big)
  \end{equation}
  Indeed, using \eqref{eqlem:AUXR4} and \eqref{eqlem:AUXR6}, we obtain
  \begin{align*}
    \Upgamma^\varepsilon & \leq \| \chi_\updelta \|_{H^{\upnu+r}}
    \| {\mathcal U}^\varepsilon (t_1)  - {\mathcal U}^\varepsilon (t_2)  \|_{L_{\rm loc}^\infty(\R_+; H^{-r}(\R^3))}
    + 2  \| \chi_\updelta \|_{H^{\upnu+\sigma}}  \| \widetilde{\mathcal U}^\varepsilon (t_1)  - {\mathcal U}^\varepsilon (t_2)  \|_{L_{\rm loc}^\infty(\R_+; H^{-\sigma}(\R^3))} \\
    & \leq \widetilde{C}_\updelta \bigg(\int_{t_1}^{t_2} d\tau\, \|\partial_t {\mathcal U}^\varepsilon (\tau) \|_{ H^{-r}(\R^3)} + \upomega(\varepsilon) \bigg)
    \leq \widetilde{C}_\updelta \big( |t_1-t_2| + \upomega(\varepsilon) \big)\,.
  \end{align*}
  Time continuity estimate \eqref{eqn:RTT-R:9} allows us to replace the term $\mathcal S (t/\varepsilon)  \mathcal J_{3,\updelta} \,\widetilde{\mathcal U}^\varepsilon $ by its time regularization $ \mathcal J_{1,\upeta}  \mathcal S_2(t/\varepsilon)  \mathcal J_{3,\updelta} \,\widetilde{\mathcal U}^\varepsilon $ (with $\upeta >0$) since the error is controlled as
  \begin{equation}
    \label{eqn:RTT-R:10}
    \| \mathcal S(t/\varepsilon) \mathcal J_{1,\upeta} \mathcal J_{3,\updelta} \,\widetilde{\mathcal U}^\varepsilon  -
    \mathcal S(t/\varepsilon) \mathcal J_{3,\updelta} \,\widetilde{\mathcal U}^\varepsilon
    \|_{L_{\rm loc}^\infty(\R_+; H^{\upnu}(\R^3))} \leq \widetilde{C}_\updelta \big( \upeta + \upomega(\varepsilon) \big)\,,
  \end{equation}
  for all $\upnu \geq 0$. Using \eqref{eqn:RTT-R:10}, we then obtain
  \begin{equation}
    \label{eqn:RTT-R:11}
    |\Gamma_4^\varepsilon -\Gamma_5^\varepsilon| \lesssim \widetilde{C}_\updelta \big( \upeta + \upomega(\varepsilon) \big)\,, 
    \ \mbox{ with } \
    \Gamma_5^\varepsilon:=   \int_{\R_+} dt \int_{\R^3}dx\, {}^t(\mathcal S_2(t/\varepsilon) \mathcal J_{1,\upeta} \mathcal J_{3,\updelta} \,\widetilde{\mathcal U}^\varepsilon)[D_\perp \psi_\perp] \mathcal S_2(t/\varepsilon) \mathcal J_{1,\upeta} \mathcal J_{3,\updelta} \,\widetilde{\mathcal U}^\varepsilon\,. 
  \end{equation}   
  Therefore, by gathering estimates \eqref{eqn:RTT-R:5}, \eqref{eqn:RTT-R:7}, \eqref{eqn:RTT-R:8} and \eqref{eqn:RTT-R:11}, and by first taking the limit $\varepsilon \rightarrow 0$, then the limit $\upeta \rightarrow 0$, and finally the limit $\updelta \rightarrow 0$, we observe that the proof of Lemma~\ref{lem:cv-rvpvp-R} is complete, if we  prove $\Gamma_5^\varepsilon \rightarrow 0$, as $\varepsilon \rightarrow 0$. This is the matter of the rest of the proof.

In other words, we just have to show $\lim_{\varepsilon \rightarrow 0}\Gamma_5^\varepsilon  =0$, when $ \mathcal J_{1,\upeta} \mathcal J_{3,\updelta} \,\widetilde{\mathcal U}^\varepsilon$ is replaced by a smooth version of $\widetilde{\mathcal U}^\varepsilon$, that we denote by ${\mathcal V}^\varepsilon$ (to simplify the notation) with $ {\mathcal V}^\varepsilon \in \mathscr{C}^m(\R_+; H^\upnu(\R^3))$ for any $m\geq 0$ and any $\upnu \geq 0$ (not uniformly with respect to $\upeta$ and $\updelta$). As in the proof of Lemma~\ref{lem:RST}, following the spirit of the proof of the convergence result in the part III of \cite{LM98}, we introduce $\mathfrak U^\varepsilon \equiv {}^t(\upphi^\varepsilon, {}^t \Upphi^\varepsilon) := {}^t(\mathcal S_1(t/\varepsilon) \, \mathcal V^\varepsilon,{}^t\mathcal S_2(t/\varepsilon) \, \mathcal V^\varepsilon) = \mathcal S(t/\varepsilon) \, \mathcal V^\varepsilon$, with $\, \mathcal V^\varepsilon = {}^t(\psi^\varepsilon,{}^t \Psi^\varepsilon)\in \mathscr{C}^m(\R_+; H^\upnu(\R^3))$, and we compute explicity $\nabla_\perp \cdot( \Upphi^\varepsilon \otimes  \Upphi^\varepsilon)$ via Fourier transform. Since $\P_\perp\Psi^\varepsilon=0$, then $\Psi^\varepsilon =\nabla_\perp \uppsi^\varepsilon$, with $\uppsi^\varepsilon = \Delta_\perp^{-1}\nabla_\perp \cdot \Psi^\varepsilon$. This and commutation between $\mathcal S$ and $\P_\perp$ imply $\P_\perp\Upphi^\varepsilon=0 $, so that $\Upphi^\varepsilon =\nabla_\perp \upvarphi^\varepsilon$, with $\upvarphi^\varepsilon = \Delta_\perp^{-1}\nabla_\perp \cdot \Upphi^\varepsilon$.  
We introduce the Fourier decompositions
\[
\psi^\varepsilon=\frac{1}{(2\pi)^3}\int_{\R^3}d\xi\, e^{{\rm i}\xi\cdot x} \,\hat \psi^\varepsilon(t,\xi) \,, \quad
\Psi^\varepsilon=\frac{{\rm i}}{(2\pi)^3}\int_{\R^3}d\xi\,  e^{{\rm i}\xi \cdot x}\, \hat \uppsi^\varepsilon(t,\xi) \xi_\perp \,, 
\]
and denote by $\hat{\mathfrak U}^\varepsilon={}^t(\hat \upphi^\varepsilon, {}^t \hat \Upphi^\varepsilon = {\rm i} \,{}^t  \xi_\perp \hat \upvarphi^\varepsilon)$ the Fourier transform of ${\mathfrak U}^\varepsilon \equiv {}^t(\upphi^\varepsilon, {}^t \Upphi^\varepsilon=\nabla_\perp \upvarphi^\varepsilon)$. Inserting  the Fourier decomposition  of $\mathfrak U^\varepsilon $ in the linear equation $\partial_t\,\mathfrak U^\varepsilon = \mathcal L\, \mathfrak U^\varepsilon/\varepsilon $, we are  led to solve linear second-order ODEs in time for the Fourier coefficients $\hat\upphi^\varepsilon(t)$ and  $\hat \upvarphi^\varepsilon(t)$, with the inital conditions $\hat{\mathfrak U} ^\varepsilon(0)= \, \hat{\mathcal V}^\varepsilon(t)$  and $\partial_t\,\hat {\mathfrak U}^\varepsilon(0)= \hat{\mathcal L} \, \hat{\mathcal V}^\varepsilon(t)/\varepsilon $, where $\hat{\mathcal L} = {\rm i}\, {}^t(-c \, \xi_\perp \cdot\, ,\, {}^t \xi_\perp)$. Solving these linear ODEs, we obtain for $\Upphi^\varepsilon$,
  \begin{equation*}
    \label{Upphieps-R}
    \Upphi^\varepsilon= \nabla_\perp \upvarphi^\varepsilon = \frac{1}{(2\pi)^3}
    \int_{\R^3}d\xi\, e^{{\rm i}\xi\cdot x} \, \frac{{\rm i}\xi_\perp}{|\xi_\perp|} \, \Big \{
   \hat{\mathfrak m}^\varepsilon(t,\xi)\cos\big(\sqrt{c}|\xi_\perp|\tfrac t \varepsilon\big) -
    \tfrac{1}{\sqrt{c}} \hat\psi(t,\xi)\sin\big(\sqrt{c}|\xi_\perp|\tfrac t \varepsilon\big)
    \Big \}\,,
  \end{equation*}  
  where we have introduced $ \hat{\mathfrak m}^\varepsilon=\hat\uppsi^\varepsilon |\xi_\perp| $ (with
  ${\mathfrak m}^\varepsilon \in \mathscr{C}^m(\R_+; H^\upnu(\R^3))$, $\forall \,(m,\upnu)\geq 0$) to symmetrize the expressions.
  We then obtain
  \begin{equation}
    \label{eqn:upup-R}
    \Upphi^\varepsilon\otimes \Upphi^\varepsilon
    = -\frac{1}{2(2\pi)^6} \int_{\R^3}d\xi \int_{\R^3}d\zeta\, e^{{\rm i}(\xi+\zeta)\cdot x} \, \theta^\varepsilon(t,\xi)  \theta^\varepsilon(t,\zeta)
    ( \xi_\perp \otimes \zeta_\perp +\zeta_\perp \otimes \xi_\perp)\,, 
  \end{equation}
  with
  $
  \theta^\varepsilon(t,\xi)= \big(\mathfrak m^\varepsilon(t,\xi)/|\xi_\perp|\big) \cos(\sqrt{c}|\xi_\perp| t /\varepsilon) - \big({\psi(t,\xi)}/({\sqrt{c}|\xi_\perp|})\big)
  \sin (\sqrt{c}|\xi_\perp| t /\varepsilon)
  $. Then, we obtain 
  \begin{align}
    \label{eqn:divuu-R}
    \nabla_\perp \cdot (\Upphi^\varepsilon\otimes \Upphi^\varepsilon)
    &= -\frac{{\rm i}}{4(2\pi)^6} \int_{\R^3}d\xi \int_{\R^3}d\zeta\, e^{{\rm i}(\xi+\zeta)\cdot x} \, (\xi_\perp + \zeta_\perp)|\xi_\perp + \zeta_\perp|^2
    \theta^\varepsilon(t,\xi)  \theta^\varepsilon(t,\zeta) \nonumber \\
    & \hspace{-10pt}+\frac{{\rm i}}{4(2\pi)^6} \int_{\R^3}d\xi \int_{\R^3}d\zeta\, e^{{\rm i}(\xi+\zeta)\cdot x} \, (\xi_\perp - \zeta_\perp)(|\xi_\perp|^2 - |\zeta_\perp|^2)
    \theta^\varepsilon(t,\xi)  \theta^\varepsilon(t,\zeta)\,. 
  \end{align}
  The first term of the right-hand side of \eqref{eqn:divuu-R}, is a gradient and thus its contribution in $\Gamma_5^\varepsilon$ is null since $\nabla_\perp \cdot \psi_\perp=0$. In fact, following estimates below, we can show that this term converges in $\mathcal D'(\R_+^* \times \R^3)$ to the pressure term $\nabla_\perp \uppi_3$, as $\varepsilon \rightarrow 0$. Then, it remains to show that  the second term of the right-hand side of \eqref{eqn:divuu-R} vanishes in $\mathcal D'(\R_+^*\times \R^3)$, as  $\varepsilon \rightarrow 0$. In fact, because of the presence of the factor $1/|\xi_\perp|$ in the definition of $\theta^\varepsilon$ and since, contrary to the periodic case, we now have $\hat \psi (t,0) \neq 0$, we have to consider a truncated version of $\Upphi^\varepsilon\otimes \Upphi^\varepsilon$ around low frequences ($\xi_\perp \simeq 0$). For this, for any $\delta \in (0,1)$, we consider $\psi_\delta^\varepsilon$,  $\uppsi_\delta^\varepsilon$, $\mathfrak  m_\delta^\varepsilon$, and $\Upphi_\delta^\varepsilon$ defined by inverse Fourier transforms of $\hat \psi_\delta^\varepsilon:=\hat \psi^\varepsilon \mathbbm{1}_{\{|\xi_\perp|\geq \delta\}}$, $ \hat \uppsi_\delta^\varepsilon=\hat \uppsi^\varepsilon \mathbbm{1}_{\{|\xi_\perp|\geq \delta\}}$, $\hat{\mathfrak  m}_\delta^\varepsilon:=|\xi_\perp| \hat\uppsi_\delta^\varepsilon$ and by $\Upphi_\delta^\varepsilon:=\mathcal S_2(t/\varepsilon)\,{}^t(\psi_\delta^\varepsilon,\,{}^t\nabla_\perp \uppsi_\delta^\varepsilon)$. Using Cauchy--Schwarz inequality, we then obtain the error term
  \begin{align*}
    &\| \Upphi^\varepsilon\otimes \Upphi^\varepsilon - \Upphi_\delta^\varepsilon \otimes \Upphi_\delta^\varepsilon\|_{L_{\rm loc}^\infty(\R_+;L^\infty(\R^3))} \\
    & \quad  \lesssim \int_{\R^3}d\xi \int_{\R^3}d\zeta\, (|\hat{\mathfrak  m}^\varepsilon(t,\xi)| + |\hat \psi^\varepsilon(t,\xi)|)
    \mathbbm{1}_{\{|\xi_\perp|<\delta\}} \, (|\hat{\mathfrak  m}^\varepsilon(t,\zeta)| + |\hat \psi^\varepsilon(t,\zeta)|) \frac{(1+|\zeta|^2)^{\upnu/2}}{(1+|\zeta|^2)^{\upnu/2}}\\
    & \quad  \lesssim \delta^{3/2}\big(\| {\mathfrak  m}^\varepsilon\|_{L_{\rm loc}^\infty(\R_+;L^2(\R^3))} +
    \|\psi^\varepsilon\|_{L_{\rm loc}^\infty(\R_+;L^2(\R^3))}\big)  \bigg(\int_{\R^3}d\zeta\,
    (|\hat{\mathfrak  m}^\varepsilon(t,\zeta)|^2 + |\hat \psi^\varepsilon(t,\zeta)|^2) (1+|\zeta|^2)^\upnu\bigg)^{1/2} \\
    & \quad  \lesssim \delta^{3/2} \big(\| {\mathfrak  m}^\varepsilon \|_{L_{\rm loc}^\infty(\R_+;H^\upnu(\R^3))}^2 +
    \|\psi^\varepsilon\|_{L_{\rm loc}^\infty(\R_+;H^\upnu(\R^3))}^2\big) \lesssim  \delta^{3/2}\,,
  \end{align*}
  for any $\upnu > 3/2$. Therefore, we just have to show the claim for $\Upphi_\delta^\varepsilon$ for each $\delta \in (0,1)$. This is equivalent to assuming that  $\hat \psi^\varepsilon$ and   $\hat{\mathfrak  m}^\varepsilon$ vanish in $\xi_\perp$ on a ball $\mathscr{B}_\delta$ of radius $\delta$, independently  of $(t,\xi_\myparallel,\varepsilon)$. With this assumption, we just need to estimate the second term of the right-hand side of  \eqref{eqn:divuu-R} that  we decompose as follows, 
  \begin{align*}
    \label{eqn:tfe1}
    &\int_{\R^3}d\xi \int_{\R^3}d\zeta\, e^{{\rm i}(\xi+\zeta)\cdot x} \, (\xi_\perp - \zeta_\perp)(|\xi_\perp|^2 - |\zeta_\perp|^2)
    \theta^\varepsilon(t,\xi)  \theta^\varepsilon(t,\zeta) \nonumber\\
    & \qquad \qquad  = \frac12\int_{\R^3}d\xi \int_{\R^3}d\zeta\, e^{{\rm i}(\xi+\zeta)\cdot x} \, (\xi_\perp - \zeta_\perp)(|\xi_\perp|^2 - |\zeta_\perp|^2) \bigg \{ \nonumber \\
    & \qquad \qquad \qquad  \qquad  \qquad  \Big [  \cos\big(\sqrt{c}(|\xi_\perp|+|\zeta_\perp|) \tfrac{t}{\varepsilon}\big) +  \cos\big(\sqrt{c}(|\xi_\perp|-|\zeta_\perp|) \tfrac{t}{\varepsilon}\big) \Big] \frac{\hat{\mathfrak  m}^\varepsilon(t,\xi) \hat{\mathfrak  m}^\varepsilon(t,\zeta) }{|\xi_\perp||\zeta_\perp|} \nonumber\\
    & \qquad \qquad \qquad  \qquad  \ \ \,  -\Big [  \sin\big(\sqrt{c}(|\xi_\perp|+|\zeta_\perp|) \tfrac{t}{\varepsilon}\big) +  \sin\big(\sqrt{c}(|\xi_\perp|-|\zeta_\perp|) \tfrac{t}{\varepsilon}\big) \Big]
    \frac{\hat \psi^\varepsilon(t,\xi) \hat{\mathfrak  m}^\varepsilon(t,\zeta) }{\sqrt{c}|\xi_\perp||\zeta_\perp|} \nonumber\\
    & \qquad \qquad \qquad  \qquad  \ \ \,  -\Big [  \sin\big(\sqrt{c}(|\xi_\perp|+|\zeta_\perp|) \tfrac{t}{\varepsilon}\big) -  \sin\big(\sqrt{c}(|\xi_\perp|-|\zeta_\perp|) \tfrac{t}{\varepsilon}\big) \Big]
    \frac{\hat{\mathfrak  m}^\varepsilon(t,\xi) \hat \psi^\varepsilon(t,\zeta) }{\sqrt{c}|\xi_\perp||\zeta_\perp|} \nonumber\\
     & \qquad \qquad \qquad  \qquad  \ \ \,   -\Big [  \cos\big(\sqrt{c}(|\xi_\perp|+|\zeta_\perp|) \tfrac{t}{\varepsilon}\big) -  \cos\big(\sqrt{c}(|\xi_\perp|-|\zeta_\perp|) \tfrac{t}{\varepsilon}\big) \Big]
    \frac{\hat \psi^\varepsilon(t,\xi) \hat \psi^\varepsilon(t,\zeta) }{c|\xi_\perp||\zeta_\perp|}
    \bigg\}\,.
  \end{align*}
  The eight terms in the right-hand side of the previous equation can be estimated in a similar way, hence we only treat one them, for instance,
  \begin{multline*}
    \int_{\R^3}d\xi \int_{\R^3}d\zeta\, e^{{\rm i}(\xi+\zeta)\cdot x} \, (\xi_\perp - \zeta_\perp)(|\xi_\perp|^2 - |\zeta_\perp|^2)
    \sin\big(\sqrt{c}(|\xi_\perp|-|\zeta_\perp|) \tfrac{t}{\varepsilon}\big)
    \frac{\hat \psi^\varepsilon(t,\xi) \hat{\mathfrak  m}^\varepsilon(t,\zeta) }{\sqrt{c}|\xi_\perp||\zeta_\perp|}\\ =
    \partial_t \mathcal T_1^\varepsilon  +  \mathcal T_2^\varepsilon\,,
  \end{multline*}
  where,
  \begin{equation*}
    \mathcal T_1^\varepsilon = -\varepsilon
    \int_{\R^3}d\xi \int_{\R^3}d\zeta\, e^{{\rm i}(\xi+\zeta)\cdot x} \, (\xi_\perp - \zeta_\perp)(|\xi_\perp| + |\zeta_\perp|)
    \cos\big(\sqrt{c}(|\xi_\perp|-|\zeta_\perp|) \tfrac{t}{\varepsilon}\big)
      \frac{\hat \psi^\varepsilon(t,\xi) \hat{\mathfrak  m}^\varepsilon(t,\zeta) }{c|\xi_\perp||\zeta_\perp|}\,,
  \end{equation*}
  and  
  \begin{equation*}
    \mathcal T_2^\varepsilon = \varepsilon
    \int_{\R^3}d\xi \int_{\R^3}d\zeta\, e^{{\rm i}(\xi+\zeta)\cdot x} \, (\xi_\perp - \zeta_\perp)(|\xi_\perp| + |\zeta_\perp|)
    \cos\big(\sqrt{c}(|\xi_\perp|-|\zeta_\perp|) \tfrac{t}{\varepsilon}\big)
     \frac{\partial}{\partial t}\bigg( \frac{\hat \psi^\varepsilon(t,\xi) \hat{\mathfrak  m}^\varepsilon(t,\zeta) }{c|\xi_\perp||\zeta_\perp|}\bigg)\,,
  \end{equation*}
  The proof of Lemma~\ref{lem:cv-rvpvp-R} will be complete if we prove that $ \|\mathcal T_1^\varepsilon \|_{L_{\rm loc}^\infty(\R_+;L^\infty(\R^3))}$ and $ \|\mathcal T_2^\varepsilon \|_{L_{\rm loc}^\infty(\R_+;L^\infty(\R^3))}$ vanish as $\varepsilon \rightarrow 0$. Indeed, for any $\upnu > 3/2 $, using Cauchy--Schwarz inequality, we obtain, for any $T\in(0,+\infty)$,
   \begin{align*}
    \|\mathcal T_1^\varepsilon \|_{L_{\rm loc}^\infty(\R_+;L^\infty(\R^3))}& \lesssim \varepsilon \sup_{t\in [0,T]}
    \int_{\R^3}d\xi \int_{\R^3}d\zeta\,(|\xi_\perp|^2 + |\zeta_\perp|^2)
    \frac{|\hat \psi^\varepsilon(t,\xi)| }{|\xi_\perp|}
    \frac{|\hat{\mathfrak  m}^\varepsilon(t,\zeta)| }{|\zeta_\perp|} \nonumber\\
    & \lesssim \varepsilon \sup_{t\in [0,T]}
    \bigg(\int_{\R^3}d\xi \,   \frac{1}{(1+|\xi_\perp|^2)^{\upnu/2}}(1+|\xi_\perp|^2)^{(\upnu+2)/2}|\hat \psi^\varepsilon(t,\xi)| \bigg) \nonumber\\
    & \qquad \qquad \ \, \bigg(\int_{\R^3}d\zeta \,   \frac{1}{(1+|\zeta_\perp|^2)^{\upnu/2}}(1+|\zeta_\perp|^2)^{(\upnu+2)/2}|\hat{\mathfrak  m}^\varepsilon(t,\zeta)| \bigg) \nonumber\\
    & \lesssim \varepsilon \|\psi^\varepsilon \|_{L_{\rm loc}^\infty(\R_+;H^{\nu+2}(\R^3))}\|{\mathfrak  m}^\varepsilon \|_{L_{\rm loc}^\infty(\R_+;H^{\nu+2}(\R^3))} \lesssim \varepsilon \,,
  \end{align*}
  and similarly,
  \begin{align*}
    \|\mathcal T_2^\varepsilon \|_{L_{\rm loc}^\infty(\R_+;L^\infty(\R^3))}& \lesssim
    \varepsilon \big( \|\partial_t\psi^\varepsilon \|_{L_{\rm loc}^\infty(\R_+;H^{\nu+2}(\R^3))}\|{\mathfrak  m}^\varepsilon \|_{L_{\rm loc}^\infty(\R_+;H^{\nu+2}(\R^3))} \\
   &   \qquad + \|\psi^\varepsilon \|_{L_{\rm loc}^\infty(\R_+;H^{\nu+2}(\R^3))}\|\partial_t{\mathfrak  m}^\varepsilon \|_{L_{\rm loc}^\infty(\R_+;H^{\nu+2}(\R^3))}
    \lesssim \varepsilon\,.
  \end{align*}
 \end{proof}  

%\section*{Acknowledgements}

%\section*{Data Availability Statement}
%Data sharing is not applicable to this article as no new data were created or analyzed in this study.

\appendix
\section{Toolbox}
\label{s:tools}
This section collects lemmas frequently used throughout the text. We start with  inequalities for the first-order Taylor expansion of the power function $x\mapsto x^\gamma$, with $\gamma>1$. These inequalities, mentioned in \cite{LM98}, can be seen as ``generalized'' strong convexity properties of the power function on the non-negative half-line.
\begin{lemma}[Convexity properties of the power function]
  \label{lem:convex}
  Let $\bar{x}>0$ and $\gamma >1$ be fixed positive real numbers. For all $ R \in ]\bar{x}, + \infty [ $, there exist positive constants $\nu_i$, $i=1,2,3$,
    depending on $\gamma$, $R$ and $\bar{x}$, such that
  \begin{equation*}
    \begin{aligned}
      x^\gamma  - \gamma x\bar{x}^{\gamma -1} &+ (\gamma-1)\bar{x}^\gamma \\
      & = x^\gamma -\bar{x}^\gamma - \gamma \bar{x}^{\gamma -1}(x-\bar{x}) \geq\\
     & 
    \end{aligned}  
    \left \{
    \begin{aligned}
      \nu_1 \, |x-\bar{x}|^2, & \ \ \ 0 \leq x , \ \  \gamma \geq 2\,, \\
      \nu_2 \, |x-\bar{x}|^2, & \ \ \  0\leq x \leq R,  \ \  1<\gamma < 2\,, \\
      \nu_3 \, |x-\bar{x}|^\gamma, & \ \  \ R < x , \ \   1<\gamma < 2\,.
    \end{aligned}  
    \right.
  \end{equation*}  
\end{lemma}
\begin{proof} Knowing that $\gamma>1$, the power function $\R^+ \ni x\mapsto x^\gamma \in \R^+$ is convex  on $\R^+$, and strongly convex on any compact set of $\R_+^* $. The details are left to the reader.
\end{proof}

Recall that the Orlicz space $L_2^\gamma(\R^3)$ is defined by \eqref{def_orlicz}, or see
Appendix~A in \cite{Lio98}.

\begin{lemma}[A criterion for  belonging to the Orlicz spaces $L_2^\gamma(\Omega)$]
  \label{lem:ORC}
  Let $\gamma >1$, $\bar f > 0$ and $\delta >0$ be fixed positive real numbers. Let $ f\in L_{\rm loc}^\gamma(\R^3;\R_+)$ be given. Define
  \[\begin{array}{ll}
    \Uppi_{ \bar f, \gamma } :  L_{\rm loc}^\gamma(\R^3;\R_+) \rightarrow L_{\rm loc}^1(\R^3;\R_+) \,, \quad & \Uppi_{\bar f,\gamma}(f) := f^\gamma - \gamma \bar{f}^{\gamma -1}f +(\gamma -1) {\bar f}^{\gamma} \,, \smallskip \\
    \mathfrak Z_{2,\delta}^{\gamma, \bar f}:   L_{\rm loc}^\gamma (\R^3;\R_+) \rightarrow L_{\rm loc}^1(\R^3;\R_+) \,, \quad & \mathfrak Z_{2,\delta}^{\gamma, \bar f}(f) := |f-\bar f|^2\mathbbm{1}_{\{|f-\bar f|\leq \delta \}} + |f-\bar f|^\gamma\mathbbm{1}_{\{|f-\bar f|> \delta \}} \, .
    \end{array} \]
There exist two constants $\kappa_1$ and $\kappa_2$, depending on $(\bar f, \gamma, \delta)$, such that
    \[
    \kappa_1\, \mathfrak Z_{2,\delta}^{\gamma,\bar f}(f) \leq  \Uppi_{\bar f,\gamma}(f) \leq  \kappa_2\, \mathfrak Z_{2,\delta}^{\gamma, \bar f}(f)\,. 
    \]
  It follows that $\Uppi_{\bar f,\gamma}(f) \in L^1(\R^2)$ if and only if $(f-\bar f) \in L_2^\gamma(\R^3)$.
\end{lemma}
\begin{proof}  
The proof is long but straightforward. It mainly relies on Lemma~\ref{lem:convex}, a Taylor formula with integral remainder, and on the convexity (resp. strong convexity) of $ x\mapsto x^\gamma $ on $\R^+$ (resp. on  any compact set  of $\R^+$). The details are left to the reader. Note that
this result is similar to the more general Lemma~5.3 in \cite{Lio98} to which we can also refer the reader for a  proof.
\end{proof}

\begin{lemma}[A space-time compactness lemma of Simon \cite{Sim86}]
  \label{lem:ALL}
  Let $\mathfrak B_0 \Subset \mathfrak B \subset \mathfrak B_1$  be Banach spaces (the embedding $\mathfrak B_0  \Subset \mathfrak B$ is compact and the embedding $\mathfrak B \subset \mathfrak B_1$ is continuous).
  Let $I$ be a compact interval. Fix $ q $ with $1<q \leq \infty$. Let $ f_\varepsilon : I \rightarrow \mathfrak B $ be a family of functions indexed by $ \varepsilon $ in a directed set\footnote{Since a directed set \cite{Bor98,Kel75} is  countable or uncountable, the one-parameter family of functions $\{f\}_{\varepsilon \in J}$, called sequences (respectively subsequences) by abuse of language, must be understood as generalized sequences (respectively subsequences) such as nets (respectively subnets) in the sense of Moore--Smith (see, e.g., Chapter 4, Sections 11 and 12 in \cite{Wil70})  or filters (respectively, finer filters) in the sense of Cartan (see, e.g., Chapter 1, Section 6 in \cite{NB89}).} $ J $. Thus, for all $ t \in I $, we have $ f_\varepsilon (t) \in \mathfrak B $. We assume that $\{f_\varepsilon\}_{\varepsilon \in J}$ is bounded uniformly with respect to $ \varepsilon $ in $L^q(I; \mathfrak B) \cap L^1(I; \mathfrak B_0)$
  and that $\{\partial_t f_\varepsilon \}_{\varepsilon \in J}$ is  bounded uniformly with respect to $\varepsilon$ in $L^1(I; \mathfrak B_1)$. Then  $\{f_\varepsilon \}_{\varepsilon \in J}$ is relatively compact in
  $L^p(I; \mathfrak B)$ for all $ p $ with $1\leq p <q$.
\end{lemma}

We continue with the following space-time compactness lemma established and used in \cite{Lio98}
for the proof of existence of global weak solutions to the compressible Navier--Stokes equations.  
\begin{lemma}[A space-time compactness lemma of P.-L. Lions \cite{Lio98}]
  \label{lem:compactness} Let $\Omega$ be $\T^N$ or $\R^N$ or an open set of  $\R^N$. Let $ J $ be a directed set.
  Let $\{g_\varepsilon\}_{\varepsilon \in J}$, and  $\{h_\varepsilon \}_{\varepsilon \in J}$ be sequences converging weakly to $g$ and $h$, respectively in $L_{\rm loc}^{p_1}(\R_+; L^{p_2}(\Omega))$
  and  $L_{\rm loc}^{q_1}(\R_+; L^{q_2}(\Omega))$, where $1\leq p_1,\, p_2 \leq \infty$ and
  \[
  \frac{1}{p_1} +  \frac{1}{q_1} = 1\,, \quad  \frac{1}{p_2} +  \frac{1}{q_2} = 1\,.
  \]
  Above, weak convergences are weak$-*$ convergences whenever some of the exponents are infinite. 
  Assume, in addition, that
  \begin{align*}
    &\partial_t g_\varepsilon  \mbox{ is bounded in } L_{\rm loc}^1(\R_+; W^{-\alpha,1}(\Omega))  \mbox{ for some } \alpha\geq 0\,, \mbox{ independent of } \varepsilon \,, \\
    & \|h_\varepsilon (t,\cdot) -h_\varepsilon(t,\cdot+ \xi)\|_{L_{\rm loc}^{q_1}(\R_+; L^{q_2}(\Omega))} \xrightarrow{\quad} 0 \ \mbox{as } \ |\xi| \rightarrow 0\,, \mbox{ uniformly in } \varepsilon \,.
  \end{align*}
  Then, the sequence $\{g_\varepsilon h_\varepsilon \}_{\varepsilon\in J}$ converges to $gh$ in the sense of distributions in $\R_+^* \times \Omega$.
\end{lemma}
\begin{proof}  
  This is Lemma~5.1 of \cite{Lio98}. The assumption in time is reminiscent to the Aubin--Lions
  theorem. The assumption in space is reminiscent to the Kolmogorov--Riesz--Fr\'echet criterion
  (e.g., see Theorem~4.26 in \cite{Bre11}) for compactness ("$L^p$-versions" of the Ascoli--Arzel\`a theorem).
\end{proof}

We end with results about mollifiers.
\begin{lemma}[Mollification operators]
  \label{lem:Mollif}
  Let $\chi: \R^d \mapsto \R_+$ be a non-negative function belonging to $\mathscr{C}_c^\infty(\R^d;\R_+)$, and with total mass one. For any $\upeta \in (0,1)$,
  define $\chi_\upeta(\cdot)=\upeta^{-d}\chi(\cdot/\upeta)$. The familly of non-negative functions of mass one $\{\chi_\upeta \}_{\upeta>0}$ are called a family of mollifiers,
  while the operator $ \mathcal J_{d,\upeta} : \mathcal D'(\R^d) \mapsto \mathscr{C}^\infty(\R^d)$,
  defined as $\mathcal J_{d,\upeta} f= \chi_\upeta\ast f $, for any distribution $f$, is called a mollification operator. The mollification
  operator $\mathcal J_{d,\upeta}$ has the following approximation property.  For all  $f\in H^s(\R^d)$, with  $s \in\R$ and any $\sigma \in \R$ such that $0\leq  s-\sigma \leq 1$, we have
  the following approximation error estimate
  \begin{equation}
      \label{eq:MolA}
      \| \mathcal J_{d,\upeta} f -f\|_{H^\sigma(\R^d)} \lesssim\upeta^{s-\sigma}\| f\|_{H^s(\R^d)}\,.
 \end{equation}   
\end{lemma}

\begin{proof}  The property $\mathcal J_{d,\upeta} f= \chi_\upeta\ast f \in \mathscr{C}^\infty(\R^d)$ for any $\upeta >0$ and any $f\in \mathcal D'(\R^d)$ comes from the following  convolution property $\mathcal D'(\R^d) \ast \mathcal D(\R^d) \subset  \mathscr{C}^\infty(\R^d)$. For the proof of \eqref{eq:MolA}, observe
\begin{align*}  
 \| \mathcal J_{d,\upeta} f -f\|_{H^\sigma(\R^d)}^2& =\int_{\R^d} d\xi\, (1+|\xi |^2)^{\sigma} \, |\widehat{\chi}(\upeta \xi)-1|^2 \, |\hat f(\xi)|^2 \\
&\leq \eta^{2(s-\sigma)} \int_{\R^d} d\xi\,\frac{|\widehat{\chi}(\upeta \xi)-1|^2}{(\upeta \,  \vert \xi \vert)^{2(s-\sigma)}} \, (1+|\xi |^2)^{s} \, |\hat f(\xi)|^2\,.
\end{align*}
Using the Lebesgue dominated convergence theorem, 
this last estimate leads to \eqref{eq:MolA}, since $\widehat \chi$ is smooth at the origin with $\widehat \chi(0)=1$.
\end{proof}

\end{document}